\documentclass{article}
\usepackage{amsthm}
\usepackage{amsmath}
\usepackage{amssymb}
\usepackage{graphicx}
\usepackage{tikz}
\usepackage{cases}

\usetikzlibrary{arrows,positioning,fit,shapes,calc}

\newtheorem{theorem}{Theorem}
\newtheorem{lemma}[theorem]{Lemma}
\newtheorem{definition}[theorem]{Definition}
\newtheorem{proposition}[theorem]{Proposition}
\newtheorem{corollary}[theorem]{Corollary}
\newtheorem{remark}[theorem]{Remark}

\DeclareMathAlphabet{\bi}{OML}{cmm}{b}{it}

\DeclareMathAlphabet\bfcal{OMS}{cmsy}{b}{n} 

\newcommand{\Imag}{\operatorname{Im}}
\newcommand{\Real}{\operatorname{Re}}

\newcommand{\bE}{\boldsymbol{E}}
\newcommand{\bH}{\boldsymbol{H}}

\newcommand{\bF}{\boldsymbol{F}}
\newcommand{\bG}{\boldsymbol{G}}

\newcommand{\bJ}{\boldsymbol{J}}
\newcommand{\bK}{\boldsymbol{K}}
\newcommand{\bx}{\boldsymbol{x}}
\newcommand{\bU}{\boldsymbol{U}}
\newcommand{\bV}{\boldsymbol{V}}

\newcommand{\bu}{\boldsymbol{u}}

\newcommand{\calD}{\mathcal{D}}

\newcommand{\calF}{\mathcal{F}}
\newcommand{\calH}{\mathcal{H}}

\newcommand{\calO}{\mathcal{O}}
\newcommand{\calW}{\mathcal{W}}
\newcommand{\calZ}{\mathcal{Z}}

\newcommand{\bbA}{\mathbb{A}}
\newcommand{\bbB}{\mathbb{B}}
\newcommand{\bbC}{\mathbb{C}}
\newcommand{\bbD}{\mathbb{D}}
\newcommand{\bbE}{\mathbb{E}}
\newcommand{\bbF}{\mathbb{F}}
\newcommand{\bbH}{\mathbb{H}}
\newcommand{\bbJ}{\mathbb{J}}
\newcommand{\bbK}{\mathbb{K}}
\newcommand{\bbM}{\mathbb{M}}
\newcommand{\bbP}{\mathbb{P}}
\newcommand{\bbR}{\mathbb{R}}
\newcommand{\bbS}{\mathbb{S}}
\newcommand{\bbT}{\mathbb{T}}
\newcommand{\bbV}{\mathbb{V}}
\newcommand{\bbW}{\mathbb{W}}

\newcommand{\fE}{f_E}
\newcommand{\bfH}{\boldsymbol{f_H}}
\newcommand{\fJ}{f_J}
\newcommand{\bfK}{\boldsymbol{f_K}}

\newcommand{\Omegam}{\Omega_{\rm m}}
\newcommand{\Omegae}{\Omega_{\rm e}}

\newcommand{\bhatU}{\boldsymbol{\hat{U}}}

\newcommand{\curlvec}{\operatorname{\bf Curl}}
\newcommand{\curl}{\operatorname{curl}}
\newcommand{\bcurl}{\operatorname{\bf curl}}
\newcommand{\curlk}{\operatorname{curl_{\mathit k}}}
\newcommand{\bcurlk}{\operatorname{{\bf curl}_{\mathit k}}}
\newcommand{\nablak}{\operatorname{\boldsymbol{\nabla}_{\!\mathit k}}}
\def\div{\operatorname{div}}
\newcommand{\divk}{\operatorname{{div}_{\mathit k}}}
\newcommand{\sgn}{\operatorname{sgn}}

\def\ker{\mathrm{Ker}}

\newcommand{\bbAk}{\mathbb{A}_k}
\newcommand{\hatH}{\boldsymbol{\hat{\calH}}}
\newcommand{\calAk}{\mathcal{A}_k}

\newcommand{\gzk}{g_{k,\zeta}}
\newcommand{\glk}{g_{k,\lambda}}
\newcommand{\Thetalk}{\Theta_{k,\lambda}}
\newcommand{\Thetazk}{\Theta_{k,\zeta}}
\newcommand{\thetazk}{\theta_{k,\zeta}}
\newcommand{\thetalk}{\theta_{k,\lambda}}
\def\szk{\mathrm{s}_{k,\zeta}}
\def\czk{\mathrm{c}_{k,\zeta}}
\def\slk{\mathrm{s}_{k,\lambda}}
\def\clk{\mathrm{c}_{k,\lambda}}

\def\psim{\psi_{k,\zeta,-1}}
\def\psip{\psi_{k,\zeta,+1}}
\def\psilm{\psi_{k,\lambda,-1}}
\def\psilp{\psi_{k,\lambda,+1}}

\def\psipm{\psi_{k,\zeta,\pm 1}}
\def\psilpm{\psi_{k,\lambda,\pm 1}}

\def\wlkj{\boldsymbol{W}_{\!\! k,\lambda,j}}
\def\wlkz{\boldsymbol{W}_{\!\! k,\lambda,0}}
\def\bw{\boldsymbol{W}}

\def\Rop{\mathrm{R}}
\def\bR{\boldsymbol{\rm{R}}}
\newcommand{\bPi}{\boldsymbol{\Pi}}

\newcommand{\eps}{\varepsilon}

\newcommand{\rmd}{{\mathrm{d}}}
\newcommand{\rmD}{{\mathrm{D}}}
\newcommand{\rme}{{\rm e}}
\newcommand{\rmi}{{\rm i}}

\newcommand{\indHxy}{{\rm\scriptscriptstyle 2D}}
\newcommand{\indHx}{{\rm\scriptscriptstyle 1D}}

\newcommand{\Hxy}{\boldsymbol{\mathcal{H}}_\indHxy}
\newcommand{\Hxk}{\boldsymbol{\mathcal{H}}_{\indHxy}^{\oplus}}
\newcommand{\Hx}{\boldsymbol{\mathcal{H}}_\indHx}
\newcommand{\Hxdiv}{\boldsymbol{\mathcal{H}}_\indHx (\divk 0)}
\newcommand{\Hxydiv}{\boldsymbol{\mathcal{H}}_\indHxy (\div 0)}
\newcommand{\Hxs}{\boldsymbol{\mathcal{H}}_{\indHx,s}}
\newcommand{\Hxps}{\boldsymbol{\mathcal{H}}_{\indHx,+s}}
\newcommand{\Hxms}{\boldsymbol{\mathcal{H}}_{\indHx,-s}}
\newcommand{\Hxys}{\boldsymbol{\mathcal{H}}_{\indHxy,s}}
\newcommand{\Hxyps}{\boldsymbol{\mathcal{H}}_{\indHxy,+s}}
\newcommand{\Hxyms}{\boldsymbol{\mathcal{H}}_{\indHxy,-s}}
\newcommand{\Dx}{\boldsymbol{\mathcal{D}}_{\!\indHx}}

\def\scD{\textsc{d}}
\def\scE{\textsc{e}}
\def\scI{\textsc{i}}
\def\scZ{\textsc{z}}
\def\DD{\textsc{dd}}
\def\DE{\textsc{de}}
\def\EI{\textsc{ei}}
\def\DI{\textsc{di}}
\def\EE{\textsc{ee}}
\def\zDD{\Lambda_\DD}
\def\zDE{\Lambda_\DE}
\def\zEI{\Lambda_\EI}
\def\zDI{\Lambda_\DI}
\def\zEE{\Lambda_\EE}

\newcommand{\bbQ}{\mathbb{Q}}

\begin{document}
\title{Spectral theory for Maxwell's equations at the interface of a metamaterial. Part I: Generalized Fourier transform.}
\author{Maxence Cassier$^{a,b}$, Christophe Hazard$^{b}$ and Patrick Joly$^{b}$\\ \ \\
{\footnotesize $^a$ Department of Mathematics of the University of Utah, Salt Lake City, UT, 84112, United States}\\ 
{\footnotesize $^b$ ENSTA / POEMS$^1$, 32 Boulevard Victor, 75015 Paris, France}\\ 
{\footnotesize (cassier@math.utah.edu, christophe.hazard@ensta-paristech.fr, patrick.joly@inria.fr)}}
\footnotetext[1]{POEMS (Propagation d'Ondes: Etude Math\'ematique et Simulation) is a mixed research team (UMR 7231) between CNRS (Centre National de la Recherche Scientifique), ENSTA ParisTech (Ecole Nationale Sup\'erieure de Techniques Avanc\'ees) and INRIA (Institut National de Recherche en Informatique et en Automatique).}

\maketitle

\begin{abstract}
We explore the spectral properties of the time-dependent Maxwell's equations for a plane interface between a metamaterial represented by the Drude model and the 
vacuum, which fill respectively complementary half-spaces. We construct 
explicitly a generalized Fourier transform which diagonalizes the 
Hamiltonian that describes the propagation of transverse electric waves. 
This transform appears as an operator of decomposition on a family of 
generalized eigenfunctions of the problem. It will be used in a 
forthcoming paper to prove both limiting absorption and 
limiting amplitude principles.
\end{abstract}


%XXXXXXXXXXXXXXXXXXXXXXXXXXXXXXXXXXXXXXXXXXXXXXXXXXXXXXXXXXXXXXXXXXXXXXXXXXXXXXXXXXXXXXXXXXXXXXXXXXXXXXXXXXXXXXXXXXXXXXXXXXXXXXX
%XXXXXXXXXXXXXXXXXXXXXXXXXXXXXXXXXXXXXXXXXXXXXXXXXXXXXXXXXXXXXXXXXXXXXXXXXXXXXXXXXXXXXXXXXXXXXXXXXXXXXXXXXXXXXXXXXXXXXXXXXXXXXXX
\section{Introduction}
\label{s-intro}
In the last years, metamaterials have generated a huge interest among communities of physicists and mathematicians, owing to their extraordinary properties such as negative refraction \cite{Ves-68}, allowing the design of spectacular devices like the perfect flat lens \cite{Pen-00} or the cylindrical cloak in \cite{Nir-26}. Such properties result from the possibility of creating artificially microscopic structures whose macroscopic electromagnetic behavior amounts to negative electric permittivity $\eps$ and/or negative magnetic permeability $\mu$ within some frequency range. Such a phenomenon can also be observed in metals in the optical frequency range \cite{Jac-98,Mai-07}: in this case one says that this material is \emph{a negative material} \cite{Gra-10}. Thanks to these negative electromagnetic coefficients, waves can propagate at the interface between such a negative material and a usual dielectric material \cite{Gra-12}. These waves, often called \emph{surface plasmon polaritons}, are localized near the interface and allow then to propagate signals in the same way as in an optical fiber, which may lead to numerous physical applications. Mathematicians have so far little explored these negative materials and most studies in this context are devoted to the frequency domain, that is, propagation of time-harmonic waves \cite{Bon-14,Bon-14(2),Ngu-16}. In particular, it is now well understood that in the case of a smooth interface between a dielectric and a negative material (both assumed non-dissipative), the time-harmonic Maxwell's equations become ill-posed if both ratios of $\eps$ and $\mu$ across the interface are equal to $-1,$ which is precisely the conditions required for the perfect lens in \cite{Pen-00}. This result raises a fundamental issue which can be seen as the starting point of the present paper.

Indeed, for numerous scattering problems, a time-harmonic wave represents the large time asymptotic behavior of a time-dependent wave resulting from a time-harmonic excitation which is switched on at an initial time. Such a property is referred to as the \emph{limiting amplitude principle} in the context of scattering theory. It has been proved for a large class of physical problems in acoustics, electromagnetism or elastodynamics \cite{Eid-65,Eid-69,Mor-89,Rad-15,Roa-92,San-89}. But what can be said about the large time behavior of the time-dependent wave if the frequency of the excitation is such that the time-harmonic problem becomes ill-posed, that is, precisely in the situation described above? What is the effect of the surface plasmons on the large time behavior? Our aim is to give a precise answer to these questions in an elementary situation.

To reach this goal, several approaches are possible. The one we adopt here, which is based on spectral theory, has its own interest because it provides us a very powerful tool to represent time-dependent waves and study their behavior, not only for large time asymptotics. Our aim is to make the spectral analysis of a simple model of interface between a negative material and the vacuum, more precisely to construct a \emph{generalized Fourier transform} for this model, which is the keystone for a time--frequency analysis. Indeed this transform amounts to a \emph{generalized eigenfunction expansion} of any possible state of the system, which yields a representation of time-dependent waves as superpositions of time-harmonic waves. From a mathematical point of view, this transform offers a \emph{diagonal} form of the operator that describes the dynamics of the system. The existence of such a transform is ensured in a very general context \cite{Ber-She-Us-96,Haz-07}, but its practical construction highly depends of the considered model.

The situation studied here consists in the basic case of a plane interface between the vacuum and a negative material filling respectively two half-spaces. Our negative material is described by a non-dissipative Drude model, which is the simplest model of negative material. The technique we use to construct the generalized Fourier transform is inspired by previous studies in the context of stratified media \cite{Der-86,Gui-89,Wed-91,Wil-84}. Compared to these studies, the difficulty of the present work relies in the fact that in the Drude material, $\eps$ and $\mu$ depend on the frequency and become negative for low frequencies. For the sake of simplicity, instead of considering the complete three-dimensional physical problem, we restrict ourselves to the so-called transverse electric (TE) two-dimensional problem, \textit{i.e.}, when the electric field is orthogonal to the plane of propagation. The transverse magnetic (TM) case can be studied similarly. As shown in \cite{Wed-91}, in stratified media, the spectral theory of the three-dimensional problem follows from both TE and TM cases, but is not dealt with here. 

The present paper is devoted to the construction of the generalized Fourier transform of the TE Maxwell's equations. It will be used in a forthcoming paper \cite{Cas-Haz-Jol-Vin} to study the validity of a limiting amplitude principle in our medium, but the results we obtain in the present paper are not limited to this purpose. The generalized Fourier transform is also the main tool to study scattering problems as in \cite{Der-88,Haz-07,Wed-91,Wilc-84,Wil-84}, numerical methods in stratified media as in \cite{Gui-89} and has many other applications. In Let us mention that both present and forthcoming papers are an advanced version of the preliminary study presented in \cite{Cas-14}.

The paper is organized as follows. In \S\ref{s-mod-meth}, we introduce the above mentioned plane interface problem between a Drude material and the vacuum, more precisely the TE Maxwell's equations. These equations are formulated as a conservative Schr\"{o}dinger equation in a Hilbert space, which involves a self-adjoint Hamiltonian. We briefly recall some basic notions of spectral theory which are used throughout the paper. Section \ref{s-Spec-th-Ak} is the core of our study: we take advantage of the invariance of our medium in the direction of the interface to reduce the spectral analysis of our Hamiltonian to the analysis of a family of one-dimensional \emph{reduced Hamiltonians}. We diagonalize each of them by constructing a adapted generalized Fourier transform. We finally bring together in \S\ref{s-Spec-th-A} this family of results to construct a generalized Fourier transform for our initial Hamiltonian and conclude by a spectral representation of the solution to our Schr\"{o}dinger equation. 

%XXXXXXXXXXXXXXXXXXXXXXXXXXXXXXXXXXXXXXXXXXXXXXXXXXXXXXXXXXXXXXXXXXXXXXXXXXXXXXXXXXXXXXXXXXXXXXXXXXXXXXXXXXXXXXXXXXXXXXXXXXXXXXX
%XXXXXXXXXXXXXXXXXXXXXXXXXXXXXXXXXXXXXXXXXXXXXXXXXXXXXXXXXXXXXXXXXXXXXXXXXXXXXXXXXXXXXXXXXXXXXXXXXXXXXXXXXXXXXXXXXXXXXXXXXXXXXXX
\section{Model and method}
\label{s-mod-meth}

% ===============================================================================================================================
\subsection{The Drude model}
We consider a metamaterial filling the half-space $\bbR^3_{+}:=\{\bx := (x,y,z) \in \bbR^3\mid x>0\}$ and whose behavior is described by a Drude model (see, \textit{e.g.}, \cite{Li-13}) recalled below. The complementary half-space $\bbR^3_{-}:= \bbR^2 \times \bbR_-$ is composed of vacuum (see Figure \ref{fig.med}). The triplet $(\boldsymbol{e_x},\boldsymbol{e_y},\boldsymbol{e_z})$ stands for the canonical basis of $\bbR^3$.

We denote respectively by $\bbD$ and $\bbB$ the electric and magnetic inductions, by  $\bbE $ and $\bbH$ the electric and magnetic fields.
We assume that in the presence of a source current density $\bbJ_s$, the evolution of  $(\bbE,  \bbD, \bbH, \bbB)$ in the whole space is governed by the macroscopic Maxwell's equations
(in the following, the notation $\curlvec$ refers to the usual $3$D curl operator)\begin{equation*}
   \partial_t \bbD -\curlvec \bbH = -\bbJ_s \quad\mbox{and}\quad \partial_t \bbB +\curlvec \bbE = 0,
\end{equation*}
which must be supplemented by the constitutive laws of the material
\begin{equation*}
   \bbD= \eps_0 \bbE+  \, \bbP \quad\mbox{and}\quad \bbB= \mu_0 \bbH+ \bbM
\end{equation*}
involving two additional unknowns, the electric and magnetic polarizations $\bbP$ and $\bbM$. The positive constants $\eps_0$ and $\mu_0$ stand respectively for the permittivity and the permeability of the vacuum. 

In the vacuum, $\bbP = \bbM = 0$ so that Maxwell's equations become
\begin{equation}\label{eq.maxwelldiel}
\eps_0 \:\partial_t \bbE -\curlvec \bbH = -\bbJ_s \quad\mbox{and}\quad \mu_0 \:\partial_t \bbH +\curlvec \bbE =0 \quad\mbox {in } \bbR^3_-.
\end{equation}
On the other hand, for a homogeneous non-dissipative Drude material, the fields $\bbP$ and $\bbM$ are related to $\bbE $ and $\bbH$ through 
\begin{equation*}
\partial_t \bbP = \bbJ, \quad \partial_t \bbJ=\eps_0 \, \Omegae^2\, \bbE  \  \ \mbox{ and }  \  \  \partial_t \bbM = \bbK,  \quad \partial_t \bbK = \mu_0 \, \Omegam^2\,\bbH ,
\end{equation*}
where the two unknowns $\bbJ$ and $\bbK$ are called usually the induced electric and magnetic currents. Both parameters $\Omegae$ and $\Omegam$ are positive constants which characterize the behavior of a Drude material. We can eliminate $\bbD$, $\bbB$, $\bbP$ and $\bbM,$ which yields the time-dependent Maxwell equations in a Drude material:
\begin{equation}\label{eq.maxwelldr}
	\left\{ \begin{array}{ll}
\eps_0 \:\partial_t \bbE -\curlvec \bbH + \bbJ = -\bbJ_s & \quad \partial_t \bbJ=\eps_0 \, \Omegae^2\, \bbE \\[12pt]
 \mu_0 \:\partial_t \bbH +\curlvec \bbE +\bbK=0 & \quad \partial_t \bbK = \mu_0 \, \Omegam^2\,\bbH
\end{array} \right.
\quad\mbox{ in } \bbR^3_+.
\end{equation}
The above equations in $\bbR^3_-$ and $\bbR^3_+$ must be supplemented by the usual transmission conditions
\begin{equation}\label{eq.transmission}
[\boldsymbol{e_x} \times \bbE]_{x=0}=0 \quad\mbox{and}\quad [\boldsymbol{e_x} \times \bbH]_{x=0}=0,
\end{equation}
which express the continuity of the tangential electric and magnetic fields through the interface $x =0$ (the notation $[f]_{x=0}$ designates the gap of a quantity $f$ across $x=0,$ \textit{i.e.}, $\lim_{x \searrow 0} \{f(+x) - f(-x)\}).$
\begin{figure}[!t]
\centering
 \includegraphics[width=0.45\textwidth]{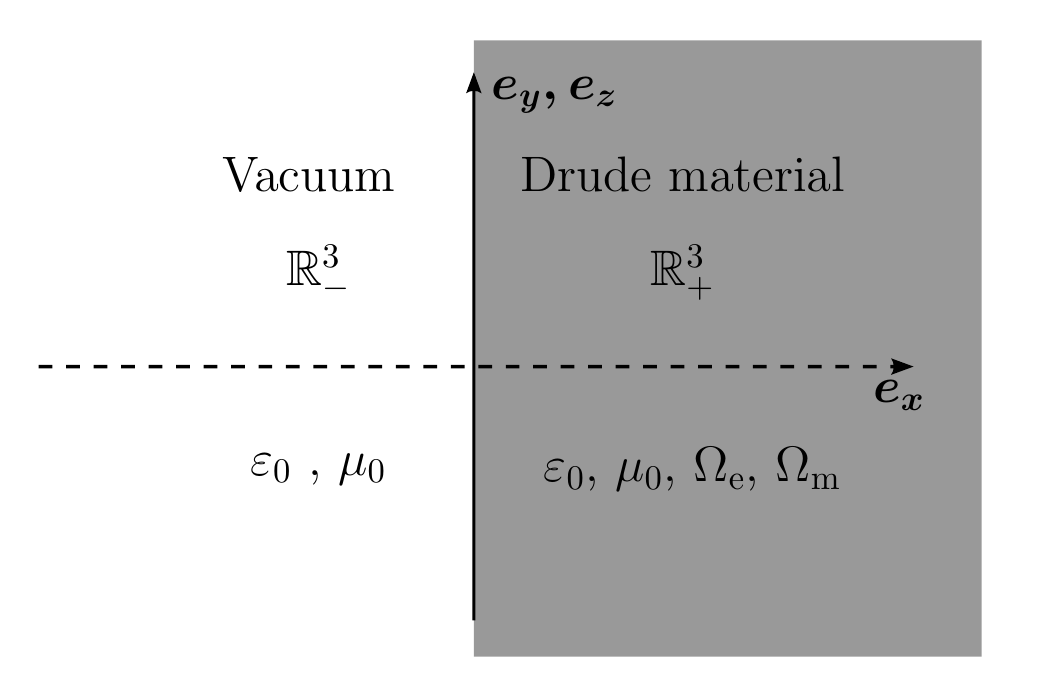}
 \caption{Description of the transmission problem}
 \label{fig.med}
\end{figure}

When looking for time-harmonic solutions to these equations for a given (circular) frequency $\omega \in \bbR,$ \textit{i.e.},
$$(\bbE(\bx,t),\bbH(\bx,t), \bbJ(\bx,t),\bbK(\bx,t)) =(\bbE_{\omega}(\bx),  \bbH_{\omega}(\bx) ,  \bbJ_{\omega}(\bx),  \bbK_{\omega}(\bx))   \  \rme^{-\rmi\omega t},$$   
for a periodic current density $\bbJ_s(\bx,t)=\bbJ_{s,\omega}(\bx)\, \rme^{-\rmi\omega t} $, we can eliminate $\bbJ_{\omega}$ and $\bbK_{\omega}$ and obtain the following time-harmonic Maxwell equations:
\begin{equation*}
\rmi\omega\, \eps_\omega(\bx) \, \bbE_{\omega} + \curlvec \bbH_{\omega}= \bbJ_{s,\omega} 
\quad\mbox{and}\quad 
-\rmi\omega\, \mu_\omega(\bx) \,\bbH_{\omega} + \curlvec\bbE_{\omega} = 0
\quad\mbox{in }\bbR^3, 
\end{equation*}
where $\eps_\omega(\bx) = \eps_0$ and $\mu_\omega(\bx) = \mu_0$ if $\bx \in \bbR^3_-,$ whereas
\begin{equation}\label{eq.drude}
\eps_\omega(\bx) = \eps_\omega^+ := \eps_0 \left(1-\frac{\Omegae^2}{\omega^2}\right) 
\quad\mbox{and}\quad
\mu_\omega(\bx) = \mu_\omega^+ := \mu_0 \left(1-\frac{\Omegam^2}{\omega^2}\right)
\quad\mbox{if } \bx \in \bbR^3_+.
\end{equation}
Functions $\eps_\omega^+$ and $\mu_\omega^+$ define the frequency-dependent electric permittivity and magnetic permeability of a Drude material (see Figure \ref{fig.drude}). Several observations can be made. First notice that one recovers the permittivity and the permeability of the vacuum if $\Omegae=\Omegam=0.$ Then, a Drude material behaves like the vacuum for high frequencies (since $\lim_{|\omega| \to \infty} \eps_\omega^+=\eps_0$ and $\lim_{|\omega| \to \infty} \mu_\omega^+=\mu_0),$ whereas for low frequencies, it becomes a \emph{negative material} in the sense that
\begin{equation*}
\eps_\omega^+<0 \mbox{ for } |\omega| \in (0,\Omegae) \quad\mbox{and}\quad \mu_\omega^+<0 \mbox{ for } |\omega| \in (0,\Omegam).
\end{equation*}
Note that if $\Omegae\neq\Omegam$, there is a frequency gap $\big(\min(\Omegae, \Omegam), \max(\Omegae, \Omegam)\big)$ of width $|\Omegae -\Omegam|$ where  $\eps_\omega^+$ and $\mu_\omega^+$ have opposite signs. At these frequencies, waves cannot propagate through the material: by this we means that corresponding plane waves are necessarily evanescent, in other words associated to non real wave vectors. It is precisely what happens in metals at optical frequencies \cite{Jac-98}. Finally there exists a unique frequency for which the relative permittivity $\eps_\omega^+ / \eps_0$ (respectively the relative permeability  $\mu_\omega^+ / \mu_0$) is equal to $-1$:
\begin{equation*}
\frac{\eps_\omega^+}{\eps_0} = -1  \mbox{ if } |\omega| = \frac{\Omegae}{\sqrt{2}}
\quad\mbox{and}\quad
\frac{\mu_\omega^+}{\mu_0} = -1  \mbox{ if } |\omega| = \frac{\Omegam}{\sqrt{2}}.
\end{equation*}
Note that both ratios can be simultaneously equal to $-1$ at the same frequency if and only if $\Omegae = \Omegam.$ 
\begin{figure}[!t]
\centering
\includegraphics[width=0.35\textwidth]{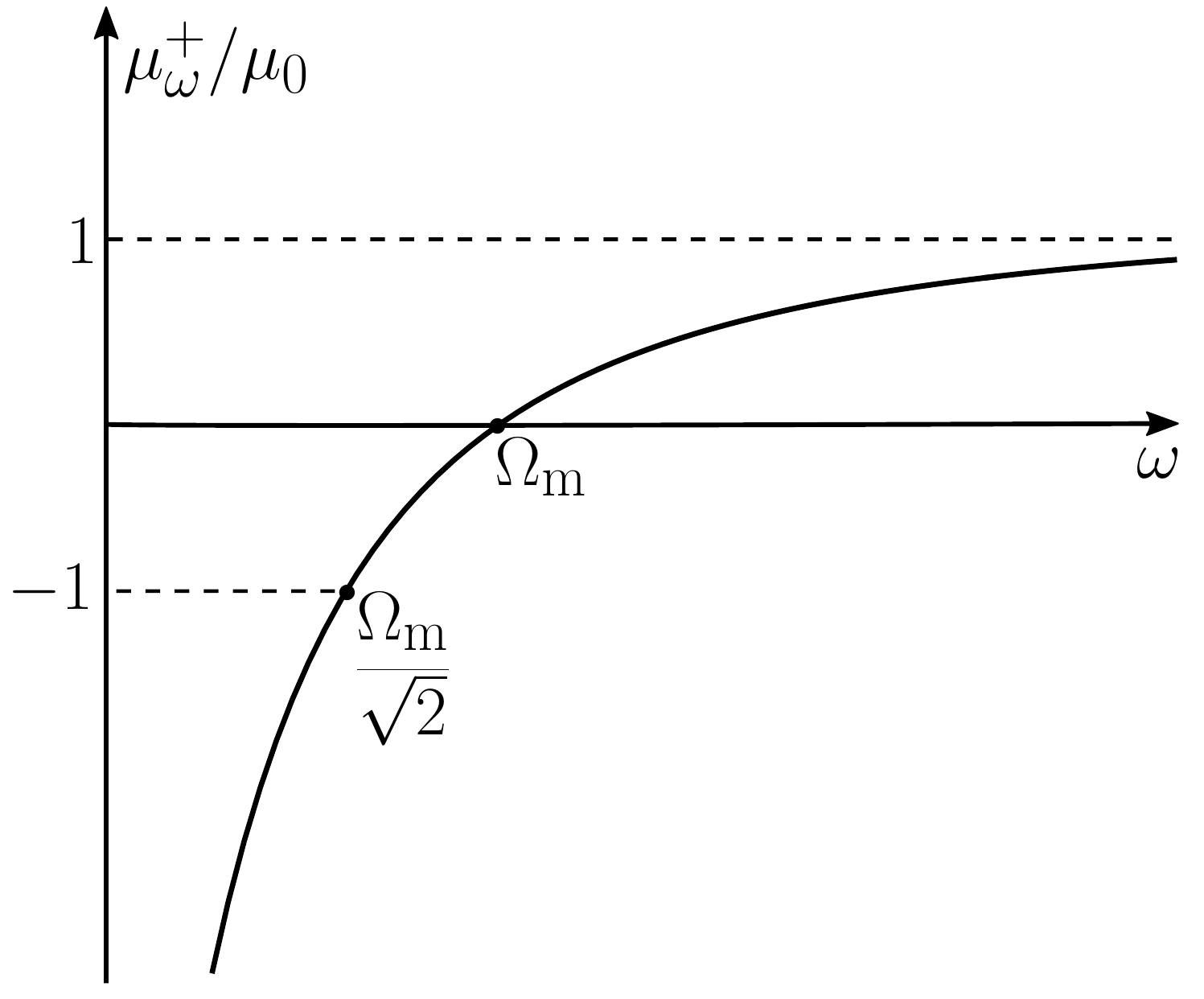}
\caption{ Relative permeability $\mu_\omega^+/\mu_0$ as a function of the frequency $\omega$.}
\label{fig.drude}
\end{figure}

\begin{remark}
In the physical literature, the Drude model {\rm (\ref{eq.drude})} consists in a simple but useful approximation of a metamaterial's behavior {\rm \cite{Pen-00,Ves-68}}. But one can find more intricate models to express the frequency dependency of  $\eps_\omega^+$ and  $\mu_\omega^+$ in the time-harmonic Maxwell's equations, for instance, the Lorentz model {\rm \cite{Gra-10,Gra-12}}:
\begin{equation*}
\eps_\omega^+ = \eps_0 \left(1 - \frac{\Omegae^2}{\omega^2-\omega_{\rm e}^2}\right)
\quad\mbox{and}\quad
\mu_\omega^+ = \mu_0 \left(1 - \frac{\Omegam^2}{\omega^2-\omega_{\rm m}^2}\right)
\end{equation*}
where $\omega_{\rm e}$ and $\omega_{\rm m}$ are non negative parameters. For generalized Lorentz materials {\rm \cite{Tip-04}}, functions $\eps_\omega^+$ and $\mu_\omega^+$ are defined by finite sums of similar terms for various poles $\omega_{\rm e}$ and $\omega_{\rm m}$.
\end{remark}

% ===============================================================================================================================
\subsection{A two-dimensional transmission problem}

As mentioned in \S\ref{s-intro}, in this paper, we restrict ourselves to the study of the so-called transverse electric (TE) equations which result from equations (\ref{eq.maxwelldiel}), (\ref{eq.maxwelldr}) and (\ref{eq.transmission}) by assuming that $\bbJ_s (x,y,z,t) = J_s(x,y,t) \; \boldsymbol{e_z}$ and searching for solutions independent of $z$ in the form
\begin{eqnarray*}
\bbE (x,y,z,t) = E(x,y,t) \; \boldsymbol{e_z} &\mbox{ and } & \bbH (x,y,z,t) = \big(H_x(x,y,t), H_y(x,y,t),0 \big)^{\top},\\
\bbJ (x,y,z,t) = J(x,y,t) \; \boldsymbol{e_z} &\mbox{ and } & \bbK (x,y,z,t) = \big(K_x(x,y,t), K_y(x,y,t),0 \big)^{\top}.
\end{eqnarray*}
Setting $\bH := (H_x,H_y)^{\top}$ and $\bK := (K_x,K_y)^{\top},$ we obtain a two-dimensional problem for the unknowns $(E,\bH, J,\bK ),$ which will be written in the following concise form:
\begin{equation}\label{TE}
\left\{ \begin{array}{ll}
  \eps_0 \:\partial_t E -\curl \bH + \Pi \, J = -J_s & \mbox{in }  \bbR^2, \\[5pt]
  \mu_0\: \partial_t \bH + \bcurl E + \bPi\, \bK= 0 & \mbox{in } \bbR^2,\\[5pt]
  \partial_t J = \eps_0 \Omegae^2 \, \Rop \, E & \mbox{in } \bbR_+^2, \\[5pt]
	\partial_t \bK= \mu_0 \Omegam^2 \, \bR \, \bH & \mbox{in } \bbR_+^2,
\end{array} \right.
\end{equation}
where we have used the 2D curl operators of scalar and vector fields respectively:
$$
\bcurl u := (\partial_y u, -\partial_x u)^{\top}
\quad\mbox{and}\quad 
\curl \bu := \partial_x u_y -\partial_y u_x \mbox{ where } \bu=(u_x,u_y)^{\top}.
$$
Moreover, $\Pi$ (respectively, $\bPi$) denotes the extension by $0$ of a scalar function (respectively, a 2D vectorial field)  defined on $\bbR^2_+$ to the whole space $\bbR^2$, whereas $\Rop$ (respectively, $\bR$) stands for the restriction to  $\bbR^2_+$ of a function defined on the whole plane $\bbR^2.$ Note that in (\ref{TE}) where equations are understood in the sense of distributions, we assume implicitly that the two-dimensional version of the transmission conditions (\ref{eq.transmission}) are satisfied, namely
\begin{equation}\label{eq.transTE}
[E]_{x=0}=0 \quad\mbox{and}\quad [H_{y}]_{x=0}=0.
\end{equation}

The theoretical study of (\ref{TE}) is based on a reformulation of this system as a Schr\"{o}dinger equation
\begin{equation}\label{eq.schro}
\frac{\rmd \, \bU}{\rmd\, t} + \rmi\, \bbA \, \bU=\bG,
\end{equation}
where the \emph{Hamiltonian} $\bbA$ is an unbounded operator on the Hilbert space 
\begin{equation}\label{eq.defHxy}
\Hxy := L^2(\bbR^2) \times L^2(\bbR^2)^2 \times L^2(\bbR^2_+) \times L^2(\bbR^2_+)^2.\end{equation}
We assume that this space is equipped with the inner product defined for all $\bU:=(E, \bH, J, \bK)^{\top}$ and $\bU':=(E^{\prime}, \bH^{\prime}, J^{\prime}, \bK^{\prime})^{\top}$ by
\begin{equation}
(\bU,\bU')_{\indHxy} := \eps_0\ (E,E^{\prime})_{\bbR^2} + \mu_0\ (\bH,\bH^{\prime})_{\bbR^2} + \eps_0^{-1} \Omegae^{-2}\ (J,J^{\prime})_{\bbR^2_+} + \mu_0^{-1} \Omegam^{-2}\ (\bK,\bK^{\prime})_{\bbR^2_+},
\label{eq.inn-prod-2D}
\end{equation}
where $(u,v)_{\calO} := \int_{\calO} u \cdot \bar{v} \, \rmd x \rmd y$ denotes the usual $L^2$ inner product, with $\calO =\bbR^2$ or $\bbR_+^2.$ We easily verify that (\ref{TE}) writes as the Schr\"{o}dinger equation (\ref{eq.schro}) with $\bG := (- \eps_0^{-1} \, J_{s} \,,0\,, \,0 ,\, 0)^{\top}$ if we choose for $\bbA$ the operator defined by
\begin{equation}\label{eq.defA}
\bbA \, \bU := {\cal A} \, \bU \quad\forall \; \bU \in \rmD(\bbA) := H^{1}(\bbR^2) \times \bH_{\!\curl}(\bbR^2) \times  L^2(\bbR^2_+) \times L^2(\bbR^2_+)^2 \subset \Hxy,
\end{equation} 
where $\bH_{\!\curl}(\bbR^2) := \{ \bu\in  L^2(\bbR^2)^2 \mid \curl \bu \in L^2(\bbR^2)\}$ and ${\cal A}$ is the following matrix differential operator (all derivatives are understood in the distributional sense):
\begin{equation}\label{eq.opA}
{\cal A} := \ \rmi\, \begin{pmatrix}
0 &\eps_0^{-1}\,\curl & -\eps_0^{-1} \, \Pi & 0\\
- \mu_0^{-1}\,\bcurl& 0 &0 & - \mu_0^{-1} \,\bPi \\
\eps_0 \Omegae^2 \, \Rop & 0 & 0 &0 \\
0 & \mu_0 \Omegam^2\, \bR & 0 & 0
\end{pmatrix}.
\end{equation}
Note that the transmission conditions (\ref{eq.transTE}) are satisfied as soon as $(E,\bH) \in H^1(\bbR^2) \times \bH_{\!\curl}(\bbR^2).$

\begin{proposition}\label{prop.autoadjoint}
The operator $\bbA: \rmD(\bbA)\subset\Hxy \longmapsto \Hxy$ is self-adjoint.
\end{proposition}
\begin{proof}
The symmetry property $(\bbA\,\bU,\bU')_{\indHxy} = (\bU,\bbA\,\bU')_{\indHxy}$ for all $\bU,\bU' \in \rmD(\bbA)$ follows from our choice (\ref{eq.inn-prod-2D}) of an inner product and the fact that the operators of each of the pairs $(\curl,\bcurl),$ $(\Rop,\Pi)$ and $(\bR,\bPi)$ are adjoint to each other. Besides, it is readily seen that the domain of the adjoint of $\bbA$ coincide with $\rmD(\bbA).$
\end{proof}

By virtue of the Hille--Yosida theorem \cite{Paz-83}, Proposition \ref{prop.autoadjoint} implies that the Schr\"{o}dinger equation (\ref{eq.schro}) is well-posed, hence also the evolution system (\ref{TE}). More precisely, we have the following result.

\begin{corollary}\label{cor.Hil}
If $\bG \in C^{1}(\bbR^{+},\Hxy)$, then the Schr\"{o}dinger equation {\rm (\ref{eq.schro})} with zero initial condition $\bU(0) = 0$ 
admits a unique solution $\bU\in C^{1}(\bbR^{+},\Hxy)\cap C^{0}(\bbR^{+},\rmD(\bbA))$ given by the Duhamel integral formula:
\begin{equation}\label{eq.duhamel}
\bU(t)=\int_{0}^{t} \rme^{-\rmi \bbA\, (t-s)}\, \bG(s) \,\rmd s,\quad \forall t \geq 0, 
\end{equation}
where $(\rme^{-\rmi\bbA\, t})_{t\in \bbR}$ is the group of unitary operators generated by the self-adjoint operator $\bbA$.
\end{corollary}

As a consequence of the Duhamel formula (\ref{eq.duhamel}), we see that if $t \mapsto \|\bG(t)\|_{\indHxy}$ is bounded on $\bbR^+$ (for instance a time-harmonic source), then $\|\bU(t)\|_{\indHxy}$ increases at most linearly in time. More precisely, as $\rme^{-i \bbA\, (t-s)}$ is unitary, we have
\begin{equation}\label{eq.incrlint}
\|\bU(t)\|_{\indHxy} \leq t \ \sup_{s \in \bbR^+} \|\bG(s)\|_{\indHxy},\quad \forall t \geq 0.
\end{equation}

% ===============================================================================================================================
\subsection{Method of analysis: spectral decomposition of the Hamiltonian}
By spectral decomposition of the operator $\bbA$, we mean its \emph{diagonalization} with generalized eigenfunctions, which extends the usual diagonalization of matrices in the sense that
$$
\bbA = \bbF^* \, \hat{\bbA} \, \bbF, 
$$
where $\bbF$ is a unitary transformation from the physical space $\Hxy$ to a spectral space $\hatH$ and $\hat{\bbA}$ is a multiplication operator in this spectral space (more precisely, the multiplication by the spectral variable). The operator $\bbF$ is often called a \emph{generalized Fourier transform} for $\bbA.$ The above decomposition of $\bbA$ will lead to a modal representation of the solution $\bU$ to (\ref{eq.duhamel}). This spectral decomposition of $\bbA$ relies on general results on spectral theory of self-adjoint operators \cite{Ree-80,Sch-12}, mainly the so-called \emph{spectral theorem} which roughly says that any self-adjoint operator is \emph{diagonalizable}.

For non-expert readers, we collect below some basic materials about elementary spectral theory which allow to understand its statement, using elementary measure theory. The starting point is the notion of \emph{spectral measure} (also called \emph{projection valued measure} or \emph{resolution of the identity}).

\begin{definition}\label{def.spec}
A spectral measure on a Hilbert space $\calH$ is a mapping $\bbE$ from all Borel subsets of $\bbR$ into the set of orthogonal projections on $\calH$ which satisfies the following properties: 
\begin{enumerate}
\item $\quad\bbE(\bbR)=\mathrm{Id}$, %and $\bbE(\varnothing)=0$,
\item $ \displaystyle \quad\bbE\big(\bigcup_{n=0}^{\infty} \Lambda_n\big) u=\sum_{n=0}^{\infty}\bbE(\Lambda_n) u$  for any $u\in \calH$ and any sequence $(\Lambda_n)_{n \in\mathbb{N}}$ of disjoint Borel sets,
%\item $\quad\bbE(\Lambda_1\cap \Lambda_2)=\bbE(\Lambda_1)\, \bbE(\Lambda_2)$.
\end{enumerate}
where the convergence of the series holds in the space $\calH$. Property 2 is known as $\sigma$-additivity property. Note that 1 and 2 imply $\bbE(\varnothing)=0$ and $\bbE(\Lambda_1\cap \Lambda_2)=\bbE(\Lambda_1)\, \bbE(\Lambda_2)$ for any Borel sets $\Lambda_1$ and $\Lambda_2$.
\end{definition}
Suppose that we know some such $\bbE,$ choose some $u \in \calH$ and define $\mu_u(\Lambda) := (\bbE(\Lambda)u,u) = \| \bbE(\Lambda)u \|^2$ for $\Lambda\subset\bbR$ (where $(\cdot,\,\cdot)$ and $\| \cdot \|$ are the inner product and associated norm in $\calH$), which maps all the Borel sets of $\bbR$ into positive real numbers. Translated into $\mu_u(\Lambda),$ the above properties for $\bbE$ mean exactly that $\mu_u$ satisfies the $\sigma$-additivity property required to become a (positive) \emph{measure}, which allows us to define integrals of the form 
$$
\int_{\bbR} f(\lambda) \, \rmd\mu_u(\lambda)=\int_{\bbR} f(\lambda) \,\rmd\| \bbE(\lambda)u \|^2
$$ 
for any measurable function $f:\bbR \to \bbC$ and any $u \in \mathcal{V}_f := \{ u \in \calH \mid \int_{\bbR} |f(\lambda)| \, \rmd \| \bbE(\lambda)u \|^2<\infty\}$. Measure theory 
provides the limiting process which yields such integrals starting from the case of simple functions:
\begin{equation*}
\int_{\bbR} f(\lambda) \, \rmd\| \bbE(\lambda)u \|^2 = \sum_{n=1}^N f_n \ \| \bbE(\Lambda_n)u \|^2\quad\mbox{if }f = \sum_{n=1}^N f_n\ \boldsymbol{1}_{\Lambda_n},
\end{equation*}
where $\boldsymbol{1}_{\Lambda_n}$ denotes the indicator function of $\Lambda_n$ (the $\Lambda_n$'s are assumed disjoint to each other). 

Going further, choose now two elements $u$ and $v$ in $\calH$ and define $\mu_{u,v}(\Lambda) := (\bbE(\Lambda)u,v),$ which is no longer positive. Integrals of the form 
$\int_{\bbR} f(\lambda) \, \rmd\mu_{u,v}(\lambda) = \int_{\bbR} f(\lambda) \, \rmd(\bbE(\lambda)u,v)$ can nevertheless be defined for $u,v \in \mathcal{V}_f.$ They are simply deduced from the previous ones thanks to the the polarization identity
\begin{equation*}
4\ \mu_{u,v} = \mu_{u+v} - \mu_{u-v} + \rmi\,\mu_{u+\rmi v} - \rmi\,\mu_{u-\rmi v}.
\end{equation*}

Consider then the subspace $\rmD_f$ of $\cal H$ defined by:
 $$
 \rmD_f := \left\{ u \in \calH \mid \int_{\bbR} |f(\lambda)|^2 \, \rmd \| \bbE(\lambda)u \|^2<\infty \right\}.
 $$
By the Cauchy--Schwarz inequality, this subspace is included in $\mathcal{V}_f$ and one can prove that it is dense in $\calH$. The key point is that for $u\in \rmD_f,$ the linear form $v \mapsto \int_{\bbR} f(\lambda) \, \rmd(\bbE(\lambda)u,v)$ is continuous in $\calH$.
Thus, by Riesz theorem, we can define an operator denoted $\int_{\bbR} f(\lambda) \, \rmd\bbE(\lambda)$ with domain $\rmD_f$ by 
\begin{equation*}
 \forall u \in \rmD_f \,, \ \forall v \in \calH, \quad \left( \left\{\int_{\bbR} f(\lambda) \, \rmd\bbE(\lambda)\right\}u\, , \, v\right) = \int_{\bbR} f(\lambda) \, \rmd(\bbE(\lambda)u,v).
\end{equation*}
The operator usually associated to the spectral measure $\bbE$ corresponds to the function
$f(\lambda)=\lambda$. This operator is shown to be self-adjoint. If we choose to denote it $\bbA$, one has
\begin{equation}\label{eq.defAauto}
\bbA := \int_{\bbR} \lambda \, \rmd\bbE(\lambda).
\end{equation}

The above construction provides us a \emph{functional calculus}, \textit{i.e.}, a way to construct \emph{functions} of $\bbA$ defined by
\begin{equation}\label{eq.def-functionA}
f(\bbA) := \int_{\bbR}f( \lambda) \, \rmd\bbE(\lambda) \quad \mbox{with domain $\rmD(f(\bbA))\equiv \rmD_f$}.
\end{equation}
These operators satisfy elementary rules of composition, adjoint and normalization:
$$
f(\bbA)\,g(\bbA) = (fg)(\bbA) = g(\bbA)\,f(\bbA) , \quad f(\bbA)^* = \overline{f}(\bbA) \quad\mbox{and}\quad 
\| f(\bbA)\,u \|^2 = \int_{\bbR} |f(\lambda)|^2 \, \rmd\|\bbE(\lambda)u\|^2.
$$
The first rule confirms in particular that this functional calculus is consistent with composition and inversion, that is, the case of rational functions of $\bbA.$ The second one shows that $f(\bbA)$ is self-adjoint as soon as $f$ is real-valued. The third one tells us that $f(\bbA)$ is bounded if $f$ is bounded on the support of $\bbE,$ whereas it becomes unbounded if $f$ is unbounded. The functions which play an essential role in this paper are 
the functions $r_\zeta(\lambda) := (\lambda - \zeta)^{-1}$ associated with the \emph{resolvent} of $\bbA,$ that is, $R(\zeta) := (\bbA - \zeta)^{-1} = r_\zeta(\bbA)$ for $\zeta \in \bbC,$
exponential functions $\exp(\rmi \lambda t)$ which appears in the solution to Schr\"odinger equations and the indicator function $\boldsymbol{1}_{\Lambda}$ of an interval $\Lambda$ for which we have by construction
\begin{equation}{\label{eq.indicator}}\boldsymbol{1}_{\Lambda}(\bbA) = \bbE(\Lambda).
\end{equation}
  
We have shown above that every spectral measure give rise to a self-adjoint operator. The \emph{spectral theorem} tells us that the converse statement holds true.
\begin{theorem}\label{th.spec}
For any self-adjoint operator $\bbA$ on a Hilbert space $\calH,$ there exists a spectral measure $\bbE$ which diagonalizes $\bbA$ in the sense of {\rm (\ref{eq.defAauto})} and {\rm (\ref{eq.def-functionA})}.
\end{theorem}
\begin{remark}\label{rem.spec}
The support the spectral  measure $\bbE$ is defined as the smallest closed Borel set $\Lambda$ of $\bbR$ such that $\bbE(\Lambda)=\mathrm{Id}$. One can show that the spectrum $\sigma(\bbA)$ of $\bbA$  coincide with the  support of $\bbE$. Moreover the point spectrum $\sigma_p(\bbA)$ is the set 
$\big \{ \lambda \in \bbR \mid  \bbE(\{\lambda\})\neq 0 \big \}$.
\end{remark}

Theorem \ref{th.spec} does not answer the crucial issue: how can we find $\bbE$ if we know $\bbA?$ A common way to answer is to use the following Stone's formulas.
\begin{theorem}
Let $\bbA$ be a self-adjoint operator on a Hilbert space $\calH$. Its associated spectral measure $\bbE$ is constructed as follows, for all $u \in \calH:$
\begin{eqnarray}
\mbox{ if }a<b : & \displaystyle \Big \| \frac{1}{2} \Big(\bbE((a,b))+\bbE([a,b])\Big) \,u  \Big \|^2  = 
 \lim_{\eta \searrow 0}\ \frac{1}{\pi}\int_{a}^{b}\Imag\big(R(\lambda+\rmi\eta) \, u,u\big)\,  \rmd \lambda,  \label{eq.stone-ab}\\
\mbox{ if }a\in \bbR : &  \displaystyle \Big \|\bbE(\left\{a\right\}) u   \Big \|^2 =
 \lim_{\eta \searrow 0}\ \eta \, \Imag\big(R(a+\rmi\eta)u,u\big).  \label{eq.stone-a}
\end{eqnarray}
\end{theorem}
Note that formulas $(\ref{eq.stone-ab})$ and $(\ref{eq.stone-a})$ are sufficient by $\sigma$-additivity to know the spectral measure $\bbE$ on all Borel sets. 
According to Remark \ref{rem.spec}, formula  $(\ref{eq.stone-a})$ permits to characterize the point spectrum $\sigma_{\rm p}(\bbA)$ whereas $(\ref{eq.stone-ab})$ enables us to determine the whole spectrum $\sigma(\bbA)$ and thus its continuous spectrum.

%XXXXXXXXXXXXXXXXXXXXXXXXXXXXXXXXXXXXXXXXXXXXXXXXXXXXXXXXXXXXXXXXXXXXXXXXXXXXXXXXXXXXXXXXXXXXXXXXXXXXXXXXXXXXXXXXXXXXXXXXXXXXXXX
%XXXXXXXXXXXXXXXXXXXXXXXXXXXXXXXXXXXXXXXXXXXXXXXXXXXXXXXXXXXXXXXXXXXXXXXXXXXXXXXXXXXXXXXXXXXXXXXXXXXXXXXXXXXXXXXXXXXXXXXXXXXXXXX
\section{Spectral theory of the reduced Hamiltonian}
\label{s-Spec-th-Ak}
The invariance of our medium in the $y$-direction allows us to reduce the spectral theory of our operator $\bbA$ defined in (\ref{eq.defA}) to the spectral theory of a family of self-adjoint operators $(\bbAk)_{k\in \bbR}$ defined on functions which depend only on the variable $x.$ In the present section, we introduce this family and perform the spectral analysis of each operator $\bbAk.$ In \S\ref{s-Spec-th-A}, we collect all these results to obtain the spectral decomposition of $\bbA.$

% ===============================================================================================================================
\subsection{The reduced Hamiltonian $\bbAk$}
Let $\calF$ be the Fourier transform in the $y$-direction defined by
\begin{equation}\label{eq.deffour}
\calF u(k) := \frac{1}{\sqrt{2\pi}} \int_{\bbR} u(y)\, \rme^{-\rmi k \,y}\, \rmd y  \quad \forall u \in L^1(\bbR)\cap L^2(\bbR),
\end{equation}
which extends to a unitary transformation from $L^2(\bbR_y)$ to $L^2(\bbR_k).$ For functions of both variables $x$ and $y,$ we still denote by $\calF$ be the partial Fourier transform in the $y$-direction. In particular, the partial Fourier transform of an element $\bU \in \Hxy$ is such that
\begin{equation}\label{eq.defH1D}
\calF \bU(\cdot,k) \in \Hx := L^2(\bbR) \times L^2(\bbR)^2 \times L^2(\bbR_+) \times L^2(\bbR_+)^2 \quad \mbox{for a.e. } k \in \bbR,
\end{equation}
where the Hilbert space $\Hx$ is endowed with the inner product $(\cdot\,,\cdot)_\indHx$ defined by the same expression (\ref{eq.inn-prod-2D}) as $(\cdot\,,\cdot)_\indHxy$ except that $L^2$ inner products are now defined on one-dimensional domains.

Applying $\calF$ to our transmission problem (\ref{TE}) leads us to introduce a family of operators $(\bbAk)_{k\in \bbR}$ in $\Hx$ related to $\bbA$ (defined in (\ref{eq.defA})) by the relation
\begin{equation}\label{eq.AtoAk}
\calF(\bbA \bU)(\cdot\,,k) = \bbAk \, \calF \bU(\cdot\,,k)  \quad \mbox{for a.e. } k  \in \bbR.
\end{equation}
Therefore $\bbAk$ is deduced from the definition of $\bbA$ by replacing the $y$-derivative by the product by $\rmi k,$ \textit{i.e.},
$$
\bbA_k \, \bU := {\calAk} \, \bU, \quad\forall \bU \in \rmD(\bbAk) := H^{1}(\bbR) \times \bH_{\!\curlk}(\bbR) \times  L^2(\bbR_+) \times L^2(\bbR_+)^2 \subset \Hx,
$$
where 
\begin{equation}\label{eq.opAk}
\calAk := \rmi\ \begin{pmatrix}
0 & \eps_0^{-1}\,\curlk & -\eps_0^{-1} \, \Pi & 0 \\[4pt]
-\mu_0^{-1}\,\bcurlk & 0 & 0 & - \mu_0^{-1} \,\bPi \\[4pt]
\eps_0 \Omegae^2 \,\Rop & 0 & 0 &0 \\[4pt]
0 & \mu_0 \Omegam^2\,\bR & 0 & 0
\end{pmatrix},
\end{equation} 
\begin{equation*}
\bcurlk u := \left(\rmi k u, -\frac{\rmd u}{\rmd x}\right)^{\top}, \quad \ \curlk \bu := \frac{\rmd u_y}{\rmd x}-\rmi k u_x \mbox{ for }   \bu :=(u_x,u_y)^{\top},
\end{equation*} 
and the operators $\Pi$, $\bPi$, $\Rop$ and $\bR$ are defined as in (\ref{eq.opA}) but for functions of the variable $x$ only. Finally,
\begin{equation*}
\bH_{\!\curlk}(\bbR) := \{\bu\in L^2(\bbR)^2 \mid \curlk \bu \in L^2(\bbR)\} = L^2(\bbR) \times H^1(\bbR).
\end{equation*}
Note again that the transmission conditions (\ref{eq.transTE}) are satisfied as soon as $(E,\bH) \in H^1(\bbR) \times \bH_{\!\curlk}(\bbR).$
  
As in Proposition \ref{prop.autoadjoint}, it is readily seen that $\bbAk: \rmD(\bbAk)\subset\Hx \to \Hx$ is self-adjoint for all $k \in \bbR.$ The following proposition shows the particular role of the values $0$ and $\pm \Omegam$ in the spectrum of $\bbA_k.$ 

\begin{proposition}\label{prop.eigvalAk}
For all $k\in \bbR^{*}$, the values $0$ and $\pm \Omegam$ are eigenvalues of infinite multiplicity of $\bbA_k$ whose respective associated eigenspaces $\ker(\bbA_k)$ and $\ker(\bbA_k\mp \Omegam)$ are given by
\begin{eqnarray*}
\ker(\bbA_k) & = & \{ (0, \widetilde{\bPi}\, \nablak \phi, 0 , 0)^{\top} \mid \phi\in H_0^1(\bbR_-)\} \quad\mbox{and}\\
\ker(\bbA_k\mp \Omegam ) & = & \{ (0, \,\bPi \,\nablak\phi, \,0 , \pm \rmi\mu_0\Omegam \,\nablak \phi )^{\top} \mid \phi\in H_0^1(\bbR_+)\},
\end{eqnarray*}
where $\nablak \phi=(\rmd \phi /\rmd x,\rmi k\phi )^{\top}$, $\widetilde{\bPi}$ is the extension by $0$ of a 2D vector field defined on $\bbR_-$ to the whole line $\bbR$ and $H_0^1(\bbR_{\pm}) := \{ \phi \in H^1(\bbR_{\pm}) \mid \phi(0) = 0 \}.$ 
\label{p.vp-de-Ak}
Moreover the orthogonal complement of the direct sum of these three eigenspaces, \textit{i.e.}, 
$\big(\ker \bbA_k \oplus \ker(\bbA_k+\Omegam)\oplus \ker(\bbA_k-\Omegam)\big)^{\perp},$
is
\begin{equation}\label{eq.vuk}
\Hxdiv := \left\{ (E,\bH,J,\bK)\in \Hx \mid \divk \, \bH=0 \, \mbox{ in } \bbR_{\pm} \mbox{ and }  \divk \, \bK=0  \, \mbox{ in } \bbR_{+} \right\} 
\end{equation}
where $\divk \bu= \rmd u_x/\rmd x+\rmi k u_y$.
\end{proposition}

\begin{proof}
We detail the proof only for $\pm \Omegam.$ The case of the eigenvalue $0$ can be dealt with in the same way. Suppose that $\bU:=(E,\bH,J,\bK)^{\top} \in \rmD(\bbA_k)$ satisfies $\bbA_k\, \bU= \pm \Omegam\,\bU,$ which is equivalent to
\begin{numcases}{}
\rmi \, \eps_0^{-1} (\curlk\bH-\Pi \, J  ) = \pm \Omegam\, E, \label{eq.syst-eigen-1}\\[5pt]
-\rmi \, \mu_0^{-1}  ( \bcurlk E +\bPi \, \bK ) = \pm \Omegam\, \bH, \label{eq.syst-eigen-2} \\[5pt]
\rmi\, \eps_0 \,\Omegae^2 \, \Rop E = \pm \Omegam\, J, \label{eq.syst-eigen-3} \\[5pt]
\rmi\, \mu_0\, \Omegam^2 \, \bR \bH = \pm \Omegam\, \bK, \label{eq.syst-eigen-4}
\end{numcases}
thanks to the above definition of $\bbA_k.$ Using (\ref{eq.syst-eigen-3}) and  (\ref{eq.syst-eigen-4}), we can eliminate the unknowns $J$ and $\bK$ in (\ref{eq.syst-eigen-1}) and (\ref{eq.syst-eigen-2}) which become
\begin{equation}\label{eq.vpEH}
\curlk \bH= \mp \rmi \eps_0 \left( \Omegam- \frac{\Omegae^2 }{\Omegam} \boldsymbol{1}_{\bbR_{+}}\right) E
\quad\mbox{and}\quad
\bcurlk E=\pm \rmi \mu_0 \Omegam (1-\boldsymbol{1}_{\bbR_{+}}) \, \bH,
\end{equation}
where $\boldsymbol{1}_{\bbR_{+}}$ denotes the indicator function of $\bbR_{+}.$ In particular, we have $\bcurlk E=0$ in $\bbR_+,$ thus $E|_{\bbR_+}=0$ (by defintion of $\bcurlk$), so $J=0$ by (\ref{eq.syst-eigen-3}). In $\bbR_-,$ we can eliminate $\bH$ between the two equations of (\ref{eq.vpEH}), which yields
\begin{equation*}
- \frac{\rmd^2 E}{\rmd x^2}  + (k^2-\eps_0\mu_0\Omegam^2)\, E=0 \mbox{ in }  \bbR_{-}\quad\mbox{and}\quad E(0)=0,
\end{equation*}
where the last condition follows from (\ref{eq.transTE}) and $E|_{\bbR_+}=0$. Obviously the only solution in $H^1(\bbR_{-})$ is $E=0.$ Hence $E$ vanishes on both sides $\bbR_{+}$ and $\bbR_{-}.$ The second equation of (\ref{eq.vpEH}) then tells us that $\bH|_{\bbR_-}=0,$ whereas the first one (together with (\ref{eq.transTE})) shows that
\begin{equation*}
\curlk \bH = 0 \mbox{ in } \bbR_+ \quad\mbox{and}\quad H_y(0)=0,
\end{equation*}
which implies that $\bH|_{\bbR_+} = \nablak\phi$ where $\phi := -\rmi k^{-1} H_y \in H_0^1(\bbR_+),$ hence $\bK=\pm\rmi \mu_0 \Omegam\,\nablak \phi$ by (\ref{eq.syst-eigen-4}). 

Conversely, for all $\phi \in H_0^1(\bbR_+),$ the vector $\left(0, \,\bPi \,\nablak\phi, \,0 , \pm\rmi\mu_0\Omegam \,\nablak \phi \right)^{\top}$ belongs to $\rmD(\bbA_k)$ and satisfies (\ref{eq.syst-eigen-1})--(\ref{eq.syst-eigen-4}).

Using these characterizations of $\ker(\bbA_k \pm \Omegam)$ and $\ker \bbA_k,$ we finally identify the orthogonal complement of their direct sum, or equivalently, the intersection of their respective orthogonal complements. We have 
$$
(E,\bH,J,\bK) \in \ker(\bbA_k)^{\bot} \ \Longleftrightarrow \ \int_{\bbR_{-}}\bH \cdot \overline{\nablak \phi } \, \rmd x=0, \ \forall \phi \in H^{1}_0(\bbR_{-})  \Longleftrightarrow  \divk \bH =0  \, \mbox{ in } H^{-1}(\bbR_{-}).
$$
In the same way, 
$$
(E,\bH,J,\bK) \in \ker(\bbA_k\pm \Omegam)^{\perp} \Longleftrightarrow \divk \big (\mu_0\, \bH\mp \rmi \Omegam^{-1} \bK \big)=0  \mbox{ in } H^{-1}(\bbR_{+}).
$$
This yields the definition (\ref{eq.vuk}) of $\Hxdiv$.
\end{proof}

% ===============================================================================================================================
\subsection{Resolvent of the reduced Hamiltonian}
In order to apply Stone's formulas $(\ref{eq.stone-ab})$ and $(\ref{eq.stone-a})$ to $\bbAk,$ we first derive an integral representation of its resolvent $R_k(\zeta):=(\bbAk-\zeta)^{-1}.$ We begin by showing how to reduce the computation of $R_k(\zeta)$ to a scalar Sturm--Liouville equation, then we give an integral representation of the solution of the latter and we finally conclude.

% -------------------------------------------------------------------------------------------------------------------------------
\subsubsection{Reduction to a scalar equation}
Let $\zeta\in \bbC\setminus \bbR$. Suppose that $\bU = R_k(\zeta)\,\bF$ for some $\bF\in\Hx$ or equivalently that $(\bbAk-\zeta ) \,\bU=\bF.$ Setting $\bU=(E, \bH, J, \bK)^{\top},$ $\bF=(\fE, \, \bfH, \, \fJ, \, \bfK)^{\top}$
and using definition (\ref{eq.opAk}), the latter equation can be rewritten as
\begin{numcases}{}
\rmi \, \eps_0^{-1} (\curlk\bH-\Pi \, J  )-\zeta\, E = \fE, \nonumber \\[5pt]
-\rmi \, \mu_0^{-1}  ( \bcurlk E +\bPi \, \bK )-\zeta\, \bH =\bfH, \nonumber \\[5pt]
\rmi\, \eps_0 \,\Omegae^2 \, \Rop E-\zeta\, J =\fJ, \nonumber \\[5pt]
\rmi\, \mu_0\, \Omegam^2 \, \bR \bH-\zeta\, \bK =\bfK. \nonumber
\end{numcases}
The last two equations provide us expressions of $J$ and $\bK$ that can be substituted in the first two which become a system for both unknowns $\bH$ and $E.$ We can then eliminate $\bH$ and obtain a Sturm--Liouville equation for $E:$
\begin{equation}\label{eq.E}
- \frac{\rmd}{\rmd x}\left( \frac{1}{\mu_{\zeta}}\frac{\rmd E}{\rmd x}\right)+\frac{ \Thetazk }{\mu_{\zeta}}E = f,
\end{equation}
where
\begin{equation}\label{eq.righthandsidesturm}
f :=\zeta \left(  \eps_0\, \fE + \frac{\rmi\mu_0}{\zeta}\,\curlk \left(\frac{\bfH}{\mu_{\zeta}}\right) - \frac{\rmi}{\zeta}\,\Pi\fJ + \frac{1}{\zeta^2} \, \curlk \left(\frac{\bPi\bfK}{\mu_{\zeta}}\right)\right) 
\end{equation}
and the following notations are used hereafter:
\begin{eqnarray}
\eps_{\zeta}(x) & := & \left\lbrace\begin{array}{ll}
         \eps_{\zeta}^{-}:=\eps_{0} & \mbox{ if } x<0,\\[2pt]
       \displaystyle   \eps_{\zeta}^{+}:=\eps_0 \left( 1-\frac{ \Omegae^2}{\zeta^2}\right) & \mbox{ if } x>0,\\
       \end{array}
       \right. \label{eq.def-eps} \\
\mu_{\zeta}(x) & := & \left\lbrace\begin{array}{ll}
         \mu_{\zeta}^{-}:=\mu_{0} & \mbox{ if } x<0,\\[2pt]
         \displaystyle \mu_{\zeta}^{+}:=\mu_0 \left( 1-\frac{ \Omegam^2}{\zeta^2}\right) & \mbox{ if } x>0,
       \end{array}
       \right.\label{eq.def-mu} \\
\Thetazk(x) & := & k^2-\eps_{\zeta}(x)\,\mu_{\zeta}(x)\,\zeta^2 = \left\lbrace\begin{array}{ll} \Thetazk^{-} :=k^2-\eps_0\,\mu_0 \,\zeta^2 & \mbox{if }x < 0,\\[4pt]
\Thetazk^{+} :=k^2-\eps_{\zeta}^{+}\,\mu_{\zeta}^{+} \,\zeta^2 & \mbox{if }x > 0.
\end{array}
\right.\label{eq.defTheta}
\end{eqnarray} 
The eliminated unknowns $\bH,$ $J$ and $\bK$ can finally be deduced from $E$ by the relations
\begin{equation}\label{eq.vectorialization}
 \begin{array}{lllll}
\bH & = & \displaystyle\frac{-\rmi}{\mu_{\zeta}\,\zeta} \,\bcurlk E &- &\displaystyle  \frac{\mu_0}{\mu_{\zeta}\,\zeta}\,\bfH + \frac{\rmi}{\mu_{\zeta}\,\zeta^2}\, \bPi\,\bfK, \\ [10pt]
J & = & \displaystyle \frac{\rmi\, \eps_0 \,\Omegae^2}{\zeta} \,\Rop \,E&\displaystyle -&  \displaystyle\frac{1}{\zeta}\,\fJ, \\ [10pt]
\bK & = & \displaystyle \frac{\mu_0\, \Omegam^2 }{\mu_{\zeta}^+\,\zeta^2}\,\bR \,\bcurlk E &-&\displaystyle \frac{\rmi \mu_0^2 \, \Omegam^2}{\mu_{\zeta}^+\,\zeta^2}\,\bR\bfH - \frac{\mu_0}{\mu_{\zeta}^+\,\zeta}\,\bfK.
\end{array}
\end{equation}
We can write these results in a condensed form by introducing several operators. First, we denote by $\bbC_{k,\zeta}$ the operator which maps the right-hand side $f$ of the Sturm--Liouville equation (\ref{eq.E}) to its solution: $E = \bbC_{k,\zeta}\,f.$ By the Lax--Milgram theorem, it is easily seen that $\bbC_{k,\zeta}$ is continuous from $H^{-1}(\bbR)$ to  $H^{1}(\bbR)$ (where $ H^{-1}(\bbR)$ denotes the dual space of $ H^{1}(\bbR)).$ Next, associated to the expression (\ref{eq.righthandsidesturm}) for the right-hand side of the Sturm--Liouville equation (\ref{eq.E}), we define
\begin{equation}
\bbS_{k,\zeta}\, \bF  :=  
  \eps_0\, \fE + \frac{\rmi\mu_0}{\zeta}\,\curlk \left(\frac{\bfH}{\mu_{\zeta}}\right) - \frac{\rmi}{\zeta}\,\Pi\fJ + \frac{1}{\zeta^2} \, \curlk \left(\frac{\bPi\bfK}{\mu_{\zeta}}\right), \label{eq.opS} 
\end{equation}	
The operator $\bbS_{k,\zeta}$ is a ``scalarizator'' since it maps the vector datum $\bF$ to a scalar quantity. It is clearly continuous from $\Hx$ to $H^{-1}(\bbR).$ Finally, associated to the two columns of the right-hand side of (\ref{eq.vectorialization}) which distinguish the role of the electrical field $\bE$ from the one of the vector datum $\bF$, we define 
\begin{eqnarray}
\bbV_{k,\zeta} E & := & 
  \left( E \,,\, -\frac{\rmi}{\mu_{\zeta}\,\zeta}\,\bcurlk E \,,\, \frac{\rmi\, \eps_0 \,\Omegae^2 }{\zeta}\,\Rop \,E \,,\, \frac{\mu_0\, \Omegam^2 }{\mu_{\zeta}^+\,\zeta^2}\,\bR \,\bcurlk E \right), \label{eq.opV} \\
\bbT_{k,\zeta} \bF & := & 
  \left(0 \,,\, -\frac{\mu_0}{\mu_{\zeta}\,\zeta}\,\bfH + \frac{\rmi}{\mu_{\zeta}\,\zeta^2}\,\bPi\,\bfK \,,\, -\frac{1}{\zeta}\,\fJ \,,\,  -\frac{\rmi \, \mu_0^2 \, \Omegam^2}{\mu_{\zeta}^+\,\zeta^2}\,\bR\bfH - \frac{\mu_0}{\mu_{\zeta}^+\,\zeta}\,\bfK \right). \label{eq.opT}
\end{eqnarray} 
The operator $\bbV_{k,\zeta}$ is a ``vectorizator'' since it maps the scalar field $E$ to a vector field of $\Hx$. It is continuous from $H^1(\bbR)$ to $\Hx.$ Finally $\bbT_{k,\zeta}$ maps the vector datum $\bF$ to a vector field of $\Hx$ and is continuous from $\Hx$ to $\Hx.$ The solution of our  Sturm--Liouville equation (\ref{eq.E}) can now be expressed as $E=\zeta \, \bbC_{k,\zeta}\, \bbS_{k,\zeta} \,\bF,$ so that $\bU=(\bbT_{k,\zeta}+\zeta \,\bbV_{k,\zeta}\, \bbC_{k,\zeta}\,\bbS_{k,\zeta})\, \bF.$ To sum up, we have the following proposition.

\begin{proposition}\label{prop.res}
Let $k \in \bbR.$ The resolvent of the self-adjoint operator $\bbAk$ can be expressed as
\begin{equation*}
R_k(\zeta) = \bbT_{k,\zeta} + \zeta \,\bbV_{k,\zeta}\, \bbC_{k,\zeta}\,\bbS_{k,\zeta}, \quad\forall \zeta \in \bbC\setminus \bbR,
\end{equation*}
where $\bbC_{k,\zeta}\,f$ is the solution to the Sturm--Liouville equation {\rm (\ref{eq.E})} and the operators $\bbS_{k,\zeta}$, $\bbT_{k,\zeta}$ and $\bbV_{k,\zeta}$ are respectively defined in {\rm (\ref{eq.opS})}, {\rm (\ref{eq.opV})} and {\rm (\ref{eq.opT})}.
\end{proposition}
It is readily seen that the respective adjoints of the above operators satisfy the following relations:
\begin{equation}
\bbC_{k,\zeta}^{*}=\bbC_{k,\bar\zeta}, \quad \bbS_{k,\zeta}^{*}=\bbV_{k,\bar\zeta}, \quad \bbV_{k,\zeta}^{*}=\bbS_{k,\bar\zeta} \quad\mbox{and}\quad \bbT_{k,\zeta}^{*}=\bbT_{k,\bar\zeta},
\label{eq.adjoints}
\end{equation}
from which we retrieve the usual formula $R_k(\zeta)^{*}=R_k(\bar\zeta)$ which is valid for any self-adjoint operator.

\begin{remark}
Notice that in comparison with previous studies on stratified media which inspire our approach {\rm \cite{Der-86,Gui-89,Wed-91,Wil-84}}, the essential difference lies in the fact that our Sturm--Liouville equation {\rm (\ref{eq.E})} depends nonlinearly on the spectral variable $\zeta,$ which is a consequence of the frequency dispersion in a Drude material. This dependence considerably complicates the spectral analysis of $\bbAk.$
\end{remark}

% -------------------------------------------------------------------------------------------------------------------------------
\subsubsection{Solution of the Sturm--Liouville equation}
In order to use the expression of $R_k(\zeta)$ given by Proposition \ref{prop.res} in Stone's formulas, we need an explicit expression of $\bbC_{k,\zeta}.$ We recall here some classical results about the solution of a Sturm--Liouville equation, which provide us an integral representation of $\bbC_{k,\zeta}:$ 
\begin{equation}\label{eq.kernelop}
\bbC_{k,\zeta}\,f(x') =\int_{\bbR}\gzk(x,x') \, f(x)  \, \rmd x, \quad \forall f \in L^2(\bbR),
\end{equation}
where the kernel $\gzk$ is the Green function of the Sturm--Liouville equation (\ref{eq.E}). For all $x' \in \bbR,$ function $\gzk(\cdot\,,x')$ is defined as the unique solution in $H^1(\bbR)$ to
\begin{equation*}\label{eq.green}
- \frac{\partial}{\partial x}\left( \frac{1}{\mu_{\zeta}}\frac{\partial }{\partial x}\,\gzk(\cdot\,,x')\right)+\frac{ \Thetazk}{\mu_{\zeta}} \, \gzk(\cdot\,,x')=\delta_{x'},
\end{equation*} 
where $\delta_{x'}\in H^{-1}(\bbR)$ is the Dirac measure at $x'$. Note that formula (\ref{eq.kernelop}) is only valid for $f\in L^2(\bbR).$ If $f \in H^{-1}(\bbR)$, we just have to replace the integral by a duality product between $H^{-1}(\bbR)$ and $H^{1}(\bbR).$

In order to express $\gzk$, we first introduce the following basis of the solutions to the homogeneous Sturm--Liouville equation associated to (\ref{eq.E}), \textit{i.e.}, for $f=0:$
\begin{equation*}\label{eq.basis}
\czk(x):=\cosh\big(\thetazk(x)\,x\big) \quad \mbox{and} \quad \szk(x):=  \displaystyle \mu_{\zeta}(x) \frac{ \sinh\big(\thetazk(x)\,x\big)}{\thetazk(x)},
\end{equation*}
where
\begin{equation}\label{eq.deftheta}
\thetazk(x) := \sqrt{\Thetazk(x)}
= \left\lbrace \begin{array}{ll} \displaystyle \thetazk^{-} := \sqrt{k^2-\eps_0\,\mu_0 \,\zeta^2} & \mbox{if } x < 0,\\[4pt]
\thetazk^{+} := \sqrt{k^2-\eps_{\zeta}^{+}\,\mu_{\zeta}^{+} \,\zeta^2} & \mbox{if } x > 0,
\end{array}
\right.
\end{equation} 
and $\sqrt{\cdot}$ denotes the principal determination of the complex square root, \textit{i.e.},
\begin{equation}\label{eq.defrac}
\sqrt{z}=|z|^{\frac{1}{2}}\,\rme^{\rmi\arg z/2} \ \mbox{ for } \ \arg{z} \in (-\pi,\pi).
\end{equation}
The special feature of the above basis is that both $\czk(x)$ and $\szk(x)$ are analytic functions of $\zeta \in \bbC\setminus\{0\}$ for all $x$ (since they can be expanded as power series of $\Thetazk(x) = \thetazk(x)^2).$ In particular, they do not depend on the choice of the determination of $\sqrt{\cdot},$ whereas $\thetazk$ depends on it. Note that this definition of $\thetazk$ makes sense since $\Thetazk^{\pm}\in \bbC \setminus \bbR^{-}$ for all $\zeta\in \bbC\setminus \bbR$ and $k\in\bbR.$ This is obvious for $\Thetazk^-$ since $\zeta \notin \bbR$ implies $\zeta^2 \notin \bbR^+,$ hence $\Thetazk^- = k^2 - \eps_0\mu_0\,\zeta^2 \notin (-\infty,k^2].$ On the other hand, by (\ref{eq.def-eps}), (\ref{eq.def-mu}) and (\ref{eq.defTheta}), we have $\Thetazk^+ = k^2 - \eps_0\mu_0(\zeta - \Omegae^2 / \zeta)(\zeta - \Omegam^2 / \zeta),$ which cannot belong to $(-\infty,k^2]$ for the same reasons, since the imaginary parts of both quantities $\zeta - \Omegae^2 / \zeta$ and $\zeta - \Omegam^2 / \zeta$ have the same sign as $\Imag\zeta.$

\begin{proposition}\label{prop.green}
Let $\zeta\in\bbC\setminus \bbR$ and $k\in\bbR$. For all $f \in L^2(\bbR),$ the function $\bbC_{k,\zeta}f$ satisfies the integral representation {\rm (\ref{eq.kernelop})} where the Green function $\gzk$ of the Sturm--Liouville equation {\rm (\ref{eq.E})} is given by
\begin{equation}\label{eq.defgreen}
\gzk(x,x')= \frac{1}{\calW_{k,\zeta}} \ \psim\big(\min(x,x')\big) \ \psip\big(\max(x,x')\big),
\end{equation}
where
\begin{equation}\label{eq.coeff-psi}
\psipm(x) :=\czk(x)\mp \frac{\thetazk^{\pm}}{\mu_{\zeta}^{\pm}} \, \szk(x) \ \forall x \in \bbR
\quad\mbox{and}\quad
\calW_{k,\zeta} := \frac{\thetazk^{-}}{\mu_{\zeta}^{-}}+ \frac{\thetazk^{+}}{\mu_{\zeta}^{+}}
\end{equation}
is the (constant) Wronskian of $\psip$ and $\psim,$ \textit{i.e.}, $\mu_\zeta^{-1} (\psip \; \psim^\prime - \psip^\prime \; \psim)$.
\end{proposition}
We omit the proof of this classical result (see, e.g., \cite{Tes-14}). The expression (\ref{eq.defgreen}) of $\gzk$ involves another basis $\{\psim,\psip\}$ of the solutions to the homogeneous Sturm--Liouville equation associated to (\ref{eq.E}) which has the special feature to be evanescent as $x$ tends either to $-\infty$ or $+\infty.$ More precisely,
\begin{eqnarray}
\psim(x) & = & \left\lbrace
      \begin{array}{ll}
        \rme^{+\thetazk^{-}\, x}  & \mbox{ if } x<0,\\[2pt]
        A_{k,\zeta,-1}\ \rme^{+\thetazk^{+}\, x} + B_{k,\zeta,-1}\ \rme^{-\thetazk^{+}\, x} & \mbox{ if } x>0,
      \end{array}
\right. \label{eq.psim} \\
\psip(x) & = & \left\lbrace\begin{array}{ll}
        A_{k,\zeta,+1}\ \rme^{-\thetazk^{-}\, x} + B_{k,\zeta,+1}\ \rme^{+\thetazk^{-}\, x} & \mbox{ if } x<0,\\[2pt]
        \rme^{-\thetazk^{+}\,x}&  \mbox{ if } x>0,
       \end{array}
 \right. \label{eq.psip} 
\end{eqnarray}
where
\begin{equation*}
A_{k,\zeta,\pm 1} := \frac{1}{2} \left( 1 + \frac{\thetazk^\pm\,/\,\mu^\pm}{\thetazk^\mp\,/\,\mu^\mp}\right)
\quad\mbox\quad
B_{k,\zeta,\pm 1} := \frac{1}{2} \left( 1 - \frac{\thetazk^\pm\,/\,\mu^\pm}{\thetazk^\mp\,/\,\mu^\mp}\right).
\end{equation*}
Formulas (\ref{eq.psim}) and (\ref{eq.psip}) show that $\psipm$ decreases exponentially when $x\to \pm\infty,$ since $\Real(\thetazk^{\pm})>0.$ Note that the Wronskian $\calW_{k,\zeta}$ cannot vanish for $\zeta\in \bbC\setminus \bbR,$ otherwise the resolvent $R_k(\zeta)$ would be singular, which is impossible, because it is analytic in $\bbC\setminus \bbR$ since $\bbAk$ is self-adjoint.

% -------------------------------------------------------------------------------------------------------------------------------
\subsubsection{Integral expression of the resolvent}
We are now able to give an explicit expression of $R_k(\zeta),$ more precisely of the quantity $(R_k(\zeta)\bU,\bU)_{\indHx}$ involved in Stone's formulas $(\ref{eq.stone-ab})$ and $(\ref{eq.stone-a})$. Thanks to Proposition \ref{prop.res} and (\ref{eq.adjoints}), we can rewrite this quantity as 
\begin{equation*}
(R_k(\zeta)\bU,\bU)_{\indHx} = \left(\bbT_{k,\zeta}\,\bU,\bU\right)_{\indHx} + \left\langle \zeta \, \bbC_{k,\zeta}\,\bbS_{k,\zeta}\, \bU,\bbS_{k,\bar\zeta}\, \bU\right\rangle_{\bbR}, \quad\forall \bU \in \Hx,
\end{equation*}
where $\langle \cdot\,,\cdot \rangle_{\bbR}$ denotes the duality product between $H^{1}(\bbR)$ and $H^{-1}(\bbR).$ In order to express this duality product as an integral (and to avoid other difficulties which occur when applying Stone's formulas), we restrict ourselves to particular $\bU$ chosen in the \emph{dense} subspace of $\Hx$ defined by
\begin{equation}\label{eq.defDx}
\Dx := \left\{ \bU = (E,\bH,J,\bK)^{\top} \in \calD(\bbR) \times \calD(\bbR)^2 \times \calD(\bbR_+) \times \calD(\bbR_+)^2 \mid H_y(0)=0 \right\},
\end{equation}
where $\calD(\calO),$ for $\calO = \bbR$ or $\bbR_+,$ denotes the space of bump functions in $\calO$ (compactly supported in $\calO$ and smooth) and the condition $H_y(0)=0$ ensures that $\bbS_{k,\zeta}(\Dx) \subset L^2(\bbR).$ Using the integral representation (\ref{eq.kernelop}), the above formula becomes
\begin{equation}\label{eq.resint2}
(R_k(\zeta)\bU,\bU)_{\indHx} = 
\big(\bbT_{k,\zeta}\,\bU,\bU\big)_{\indHx} + 
\int_{\bbR^2}\zeta \, \gzk(x,x') \ \bbS_{k,\zeta}\,\bU(x) \ \overline{\bbS_{k,\bar\zeta}\,\bU(x')} \ \rmd x \, \rmd x', \quad\forall \; \bU \in \Dx.
\end{equation}

% ===============================================================================================================================
\subsection{Boundary values on the spectral real axis}
Before applying Stone's formulas $(\ref{eq.stone-ab})$ and $(\ref{eq.stone-a})$, we have to make clear the behavior of the resolvent $R_k(\zeta)$ as $\zeta$ tends to the real axis, hence the behavior of all quantities involved in the integral representation (\ref{eq.resint2}), in particular the Green function $\gzk$. This is the subject of this paragraph.

In the following, for any quantity $f_{\zeta}$ depending on the parameter $\zeta \in \bbC\setminus \bbR$, we choose to denote $f_{\lambda}$ the one-sided limit (if it exists) of $f_{\zeta}$ when $\zeta$ tends to $\lambda \in \bbR$ from the \emph{upper half-plane}, \textit{i.e.},
\begin{equation}\label{eq.convention}
f_{\lambda} :=  \lim_{\eta \searrow 0} f_{\lambda+\rmi\eta}.
\end{equation}
Notice that $\bbS_{k,\zeta}$ $\bbV_{k,\zeta}$  and $\bbT_{k,\zeta}$ defined in (\ref{eq.opS}), (\ref{eq.opV}) and (\ref{eq.opT}) have obviously one-sided limits $\bbS_{k,\lambda}$ $\bbV_{k,\lambda}$  and $\bbT_{k,\lambda}$ if $\lambda$ differs from 0 and $\pm\Omegam$ (they actually depend analytically on $\zeta$ outside these three values).

% -------------------------------------------------------------------------------------------------------------------------------
\subsubsection{Spectral zones and one-sided limit of $\thetazk(\cdot)$}
\label{s.spectr-zones}
The determination of the one-sided limit of $\thetazk(x)$ defined in (\ref{eq.deftheta}) requires us to identify the zones, in the $(k,\lambda)$ plane, where $\Thetalk^+$ or $\Thetalk^-$ (defined in (\ref{eq.defTheta})) is located on the branch cut $\bbR^-$ of the complex square root (\ref{eq.defrac}). Using (\ref{eq.def-eps}) and (\ref{eq.def-mu}), one computes that
\begin{equation*}
\Thetalk(x) =\left\lbrace
\begin{array}{ll} 
\Thetalk^{-} = k^2-\eps_0\,\mu_0 \,\lambda^2 & \mbox{if }x < 0,\\[4pt]
\Thetalk^{+} = \displaystyle \frac{-\eps_0\mu_0\, \lambda^4 + \left(k^2+\eps_0\mu_0(\Omegae^2+\Omegam^2)\right) \lambda^2 - \eps_0\mu_0 \, \Omegae^2\Omegam^2}{\lambda^2} & \mbox{if }x > 0.
\end{array}
\right.
\end{equation*}
We denote by $\lambda_0(k)$ (respectively, $\lambda_\scD(k)$ and $\lambda_\scI(k)$, with $0 < \lambda_\scI(k) \leq \lambda_\scD(k)$) the non-negative values of $\lambda$ for which $\Thetalk^{-}$ (respectively, $\Thetalk^{+})$ vanishes, \textit{i.e.},
\begin{eqnarray*}
\lambda_0(k) & := & |k|/\sqrt{\eps_0\mu_0}, \\[8pt]
\lambda_\scD(k) & := & \sqrt{\frac{k^2}{2\eps_0\mu_0} + \frac{\Omegae^2+\Omegam^2}{2} 
+ \sqrt{\left( \frac{k^2}{2\eps_0\mu_0} + \frac{\Omegae^2-\Omegam^2}{2} \right)^2 + \frac{k^2\,\Omegam^2}{\eps_0\mu_0}}}, \\
\lambda_\scI(k) & := & \sqrt{\frac{k^2}{2\eps_0\mu_0} + \frac{\Omegae^2+\Omegam^2}{2} 
-\sqrt{\left( \frac{k^2}{2\eps_0\mu_0} + \frac{\Omegae^2-\Omegam^2}{2} \right)^2 + \frac{k^2\,\Omegam^2}{\eps_0\mu_0}}},
\end{eqnarray*}
where $\lambda_\scI(k)$ and $\lambda_\scD(k)$ are related by $\lambda_\scI(k) \,\lambda_\scD(k) =  \Omegae\Omegam$. 
In the $(k,\lambda)$-plane, the graphs of these functions are curves through which the sign of $\Thetalk^{-}$ or $\Thetalk^{+}$ changes. The main properties of these functions are summarized in the following lemma, whose proof is obvious. 

\begin{lemma}\label{lem.disp}
For all $(k,\lambda)\in \bbR \times \bbR,$ we have
\begin{equation*}
\Big( \Thetalk^{-} < 0 \Longleftrightarrow |\lambda| > \lambda_0(k) \Big) \quad\mbox{and}\quad \Big( \Thetalk^{+} < 0 \Longleftrightarrow |\lambda| \notin \left[\lambda_\scI(k), \lambda_\scD(k) \right] \Big).
\end{equation*}
The function $k \mapsto \lambda_\scI(k)$ is a $\mathcal{C}^{\infty}$ strictly decreasing function on $\bbR^{+}$ whose range is $(0, \min(\Omegae,\Omegam)]$, whereas $k \mapsto \lambda_\scD(k)$ is a $\mathcal{C}^{\infty}$ strictly increasing function on $\bbR^{+}$ whose range is $[\max(\Omegae,\Omegam),\infty)$ and $\lambda_\scD(k) = \lambda_0(k) + O(k^{-1})$ as $k \to +\infty.$ Moreover, denoting $k_{\rm c}$ the unique non negative value of $k$ for which $\lambda_0(k) = \lambda_\scI(k)$, that is to say
$$
k_{\rm c} = \sqrt{\eps_0 \mu_0} \, \lambda_{\rm c} \quad \mbox{where} \quad \lambda_{\rm c} := \Omegae \, \Omegam/ \sqrt{\Omegae^2+\Omegam^2} \quad (\equiv \lambda_0(k_{\rm c}) = \lambda_\scI(k_{\rm c}) ),
$$
one has
\begin{equation*}
 0< \lambda_0(k) < \lambda_\scI(k) < \lambda_\scD(k) \mbox{ for } 0 < k < k_{\rm c} \ \mbox{ and }\ 0< \lambda_\scI(k)< \lambda_0(k) <\lambda_\scD(k) \mbox{ for } \ k> k_{\rm c}.
\end{equation*}
Finally, $\lambda_\scI(k) = \lambda_\scD(k)$ if and only if $k=0$ and $\Omegae=\Omegam.$ 
\end{lemma}

\begin{figure}
\begin{minipage}{7.5cm}%
\includegraphics[width=\textwidth]{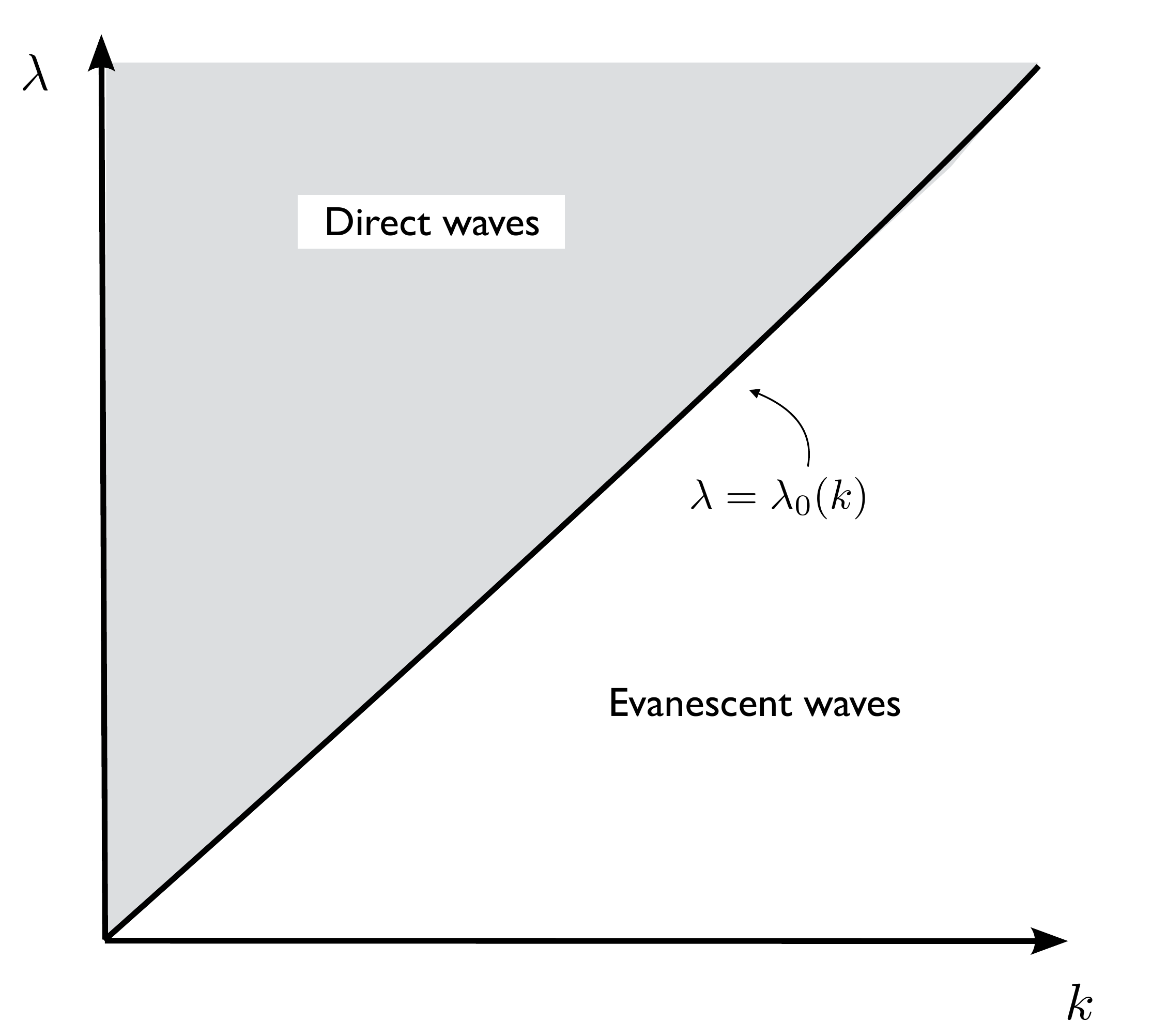}
\end{minipage}
 \hspace{1cm}
\begin{minipage}{7.5cm}%
\includegraphics[width=\textwidth]{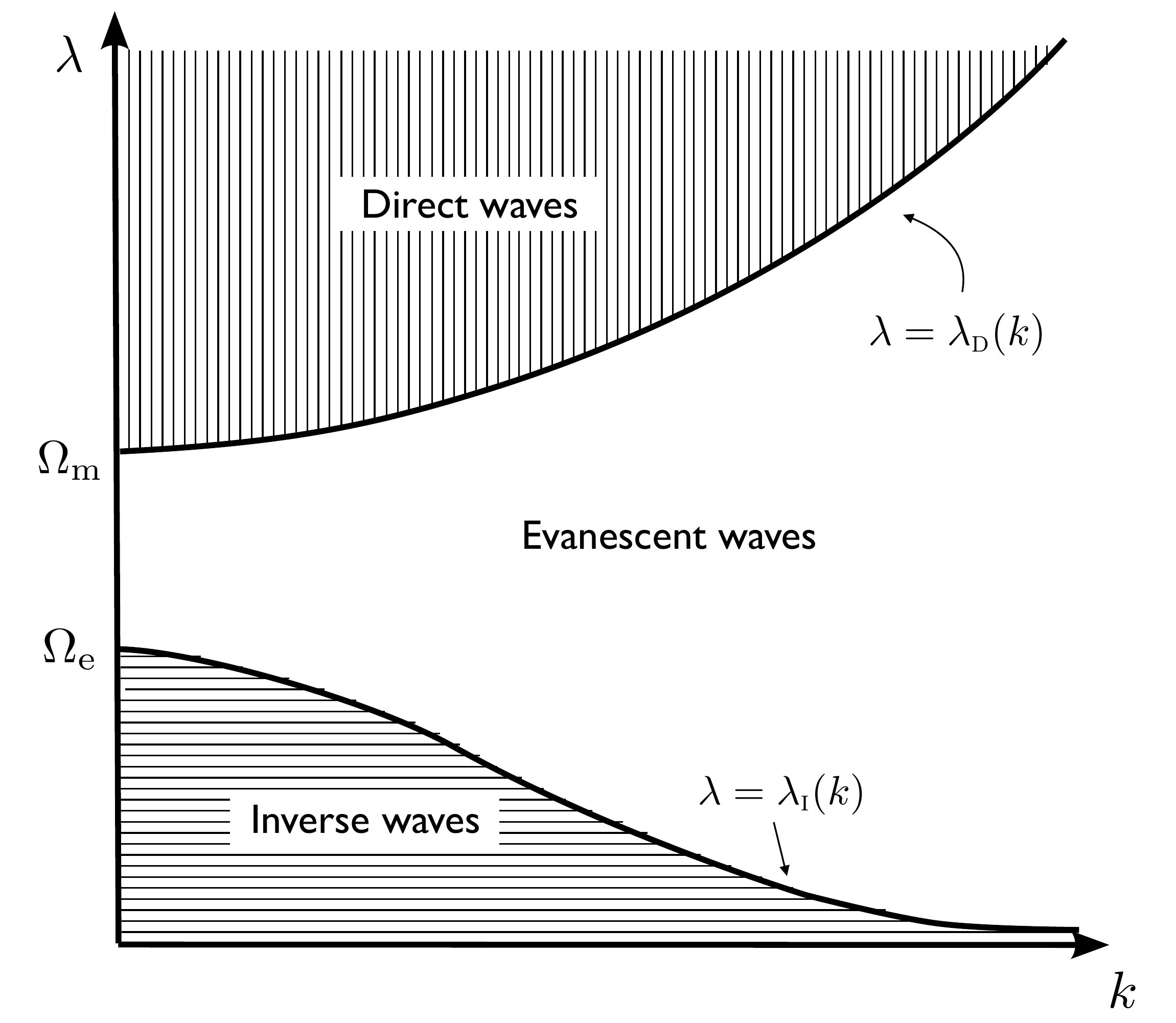}
\end{minipage}
\caption{Spectral zones associated to the vaccum (left) and the Drude material (right, case where $\Omegae<\Omegam).$ Grey color, vertical and horizontal hatched lines mean respectively propagation in the vacuum, direct propagation and inverse propagation in the Drude material.}
 \label{fig.dieldru}
\end{figure}
 
\begin{figure}
\begin{minipage}{7.0cm}%
\includegraphics[width=\textwidth]{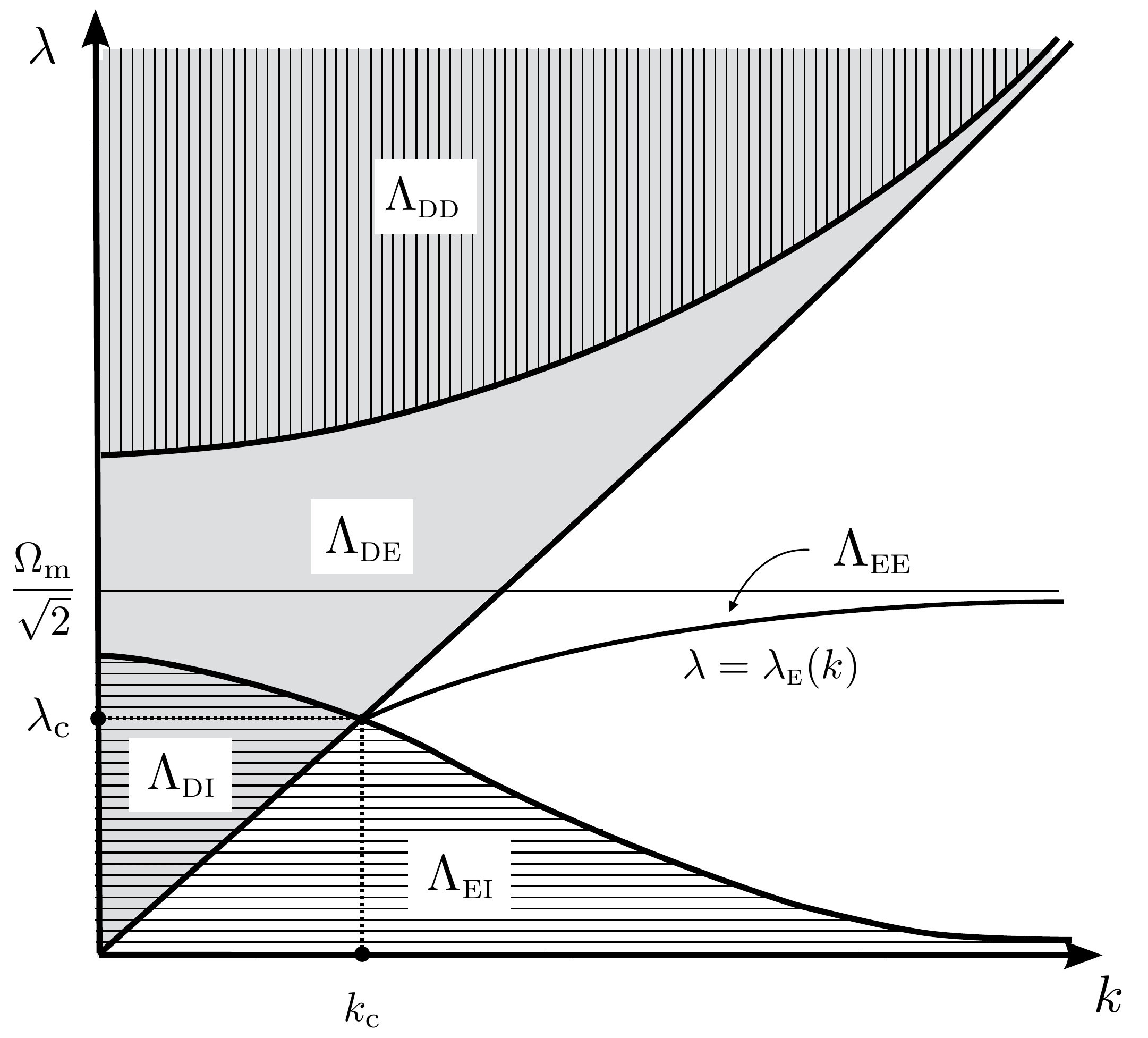}
\end{minipage}
 \hspace{0.5cm}
\begin{minipage}{8.0cm}%
\includegraphics[width=\textwidth]{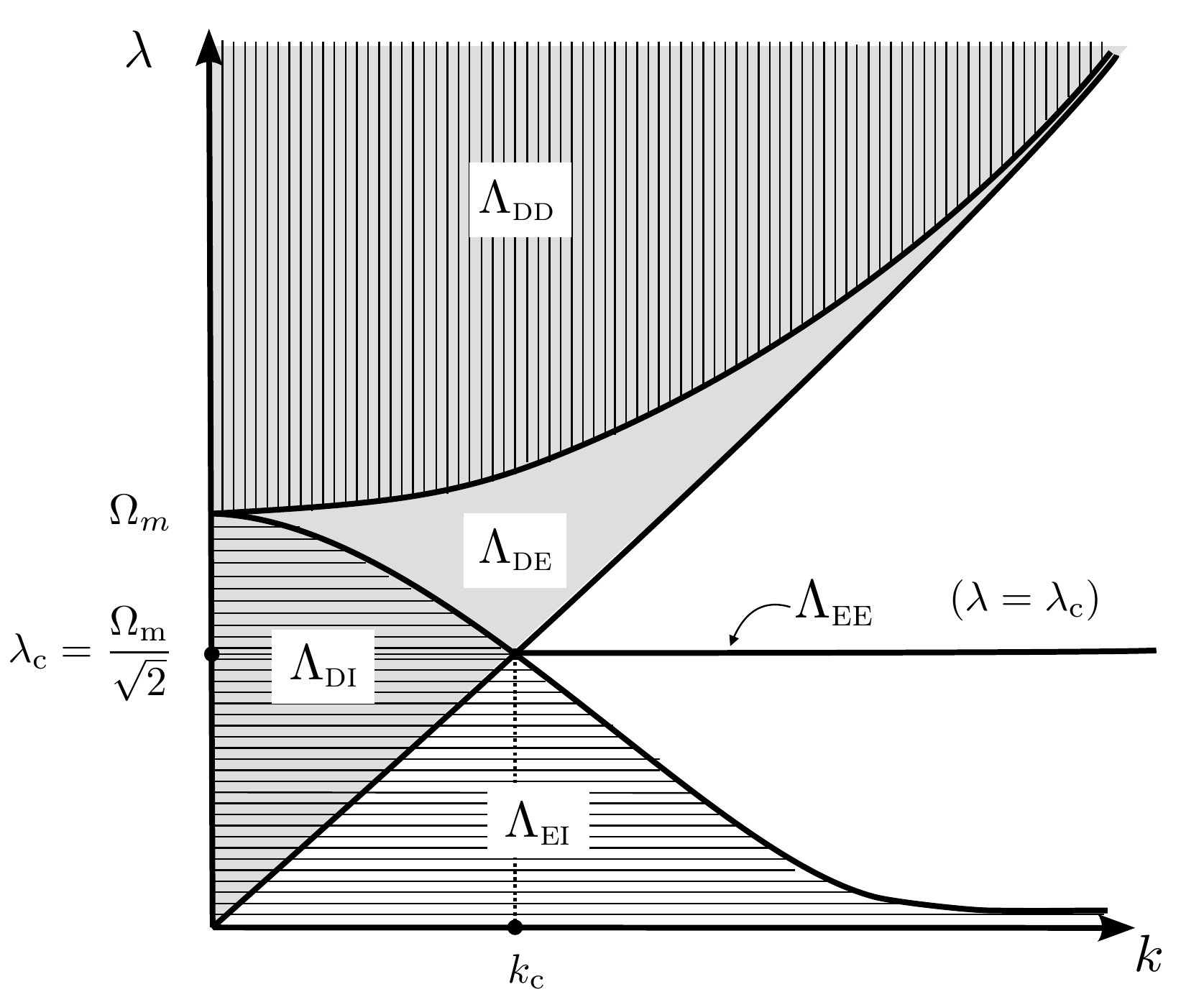}
\end{minipage}
\caption{Spectral zones associated to our medium  for $\Omegae<\Omegam$ (left) and $\Omegae=\Omegam$} (right). See caption of Figure \ref{fig.dieldru}.
\label{fig.speczones}
 \end{figure}

Figure \ref{fig.dieldru} illustrates this lemma in the quarter $(k,\lambda)$-plane $\bbR^+ \times \bbR^+$ (all zones are symmetric with respect to both $k$ and $\lambda$ axes, since $\Thetalk^{\pm}$ are even functions of $k$ and $\lambda).$ The signs of $\Theta^{+}_{k,\lambda}$ and $\Theta^{-}_{k,\lambda}$ indicate the regime of vibration in the vacuum and in the Drude material: \emph{propagative} or \emph{evanescent}. As it will be made clear in \S\ref{s.phys-spect-zones}, two kinds of propagative waves occur in the Drude material, which will be called \emph{direct} and \emph{inverse}, whereas propagative waves can only be \emph{direct} in the vacuum. This explains the choice of the indices $\scD,$ $\scI$ and $\scE$ used hereafter, which mean respectively \emph{direct}, \emph{inverse} and \emph{evanescent}.

Figure \ref{fig.speczones}, constructed as the superposition of the left and right graphics of Figure \ref{fig.dieldru}, then shows the coupling of both half-spaces. This leads us divide the $(k,\lambda)$-plane in several \emph{spectral zones} defined as follows (see \S\ref{s.phys-spect-zones} for the justification of the notations):
\begin{equation}
\begin{array}{rll}
\zDD & := & \left\{(k,\lambda) \in \bbR^2 \mid |\lambda| > \lambda_\scD(k)\right\},\\[4pt]
\zDE & := & \left\{(k,\lambda) \in \bbR^2 \mid |\lambda| \neq \Omegam \mbox{ and } \max\big(\lambda_0(k),\lambda_\scI(k)\big) < |\lambda| < \lambda_\scD(k)\right\},\\[4pt]
\zDI & := & \left\{(k,\lambda) \in \bbR^2 \mid |k| < k_{\rm c} \mbox{ and } \lambda_0(k) < |\lambda| < \lambda_\scI(k)\right\}, \\[4pt]
\zEI & := & \left\{(k,\lambda) \in \bbR^2 \mid 0 < |\lambda| < \min\big(\lambda_0(k),\lambda_\scI(k)\big)\right\}.
\end{array}
\label{eq.defspeczones}
\end{equation}
Note that the eigenvalues $\lambda = 0$ and  $\lambda = \pm\Omegam$ of $\bbAk$ (see Proposition \ref{p.vp-de-Ak}) are excluded from these zones. 

With our choice (\ref{eq.defrac}) for the complex square root, we notice that if $\Theta^{\pm}_{k,\lambda}:=\lim_{\eta\searrow 0} \Theta^{\pm}_{k,\lambda+\rmi \eta}$ is positive, then $\theta^{\pm}_{k,\lambda}$ is a positive real number (more precisely, $\theta^{\pm}_{k,\lambda}=|\Theta^{\pm}_{k,\lambda}|^{1/2}).$ On the other hand, if $\Theta^{\pm}_{k,\lambda}$ is negative, then $\theta^{\pm}_{k,\lambda}$ is purely imaginary and its sign coincides with the sign of $\Imag(\Theta_{k,\lambda+\rmi\eta}^{\pm})$ for small positive $\eta.$ From (\ref{eq.def-eps}), (\ref{eq.def-mu}) and (\ref{eq.defTheta}), one computes that
\begin{equation*}
\Imag\Thetazk^- = -\eps_0\mu_0\, \Imag(\zeta^2) 
\quad\mbox{and}\quad
\Imag\Thetazk^+ = - \frac{\eps_0\mu_0}{|\zeta|^4}\, \Imag(\zeta^2) \ \Big( |\zeta|^4 - \Omegae^2\,\Omegam^2 \Big). 
\end{equation*}
Note that for $\zeta=\lambda+\rmi\eta$, $\Imag(\zeta^2)$ has the same sign as $\lambda$, as a consequence $\Imag(\Theta^{-}_{k,\lambda+\rmi\eta})$ has the opposite sign of $\lambda$. Moreover, the sign of $\lambda^4-\Omegae^2\Omegam^2$, the limit as $\eta \searrow 0$ of $|\zeta|^4-\Omegae^2\Omegam^2$ when $\zeta=\lambda+\rmi \eta$, depends on the position of $\left|\lambda\right|$ with respect $\sqrt{\Omegae\,\Omegam}$.
From Lemma \ref{lem.disp}, we deduce that 
$$\lambda_\scI(k) \leq \min(\Omegae,\Omegam) \leq \sqrt{\Omegae\,\Omegam} \leq \max(\Omegae,\Omegam) \leq \lambda_\scD(k),$$ 
thus the sign of $\lambda^4-\Omegae^2\Omegam^2$ is positive in the spectral zone $\zDD$ (where $\lambda >\lambda_\scD(k)$) and negative (where $\lambda <\lambda_\scI(k)$) in the spectral zones $\zDI$ and $\zEI$. Finally, we obtain
\begin{eqnarray}
\thetalk^-  & = & \left\lbrace \begin{array}{ll} 
- \rmi \,\sgn(\lambda)\,|\Thetalk^-|^{1/2} & \mbox{if } (k,\lambda) \in \zDI \cup \zDE \cup \zDD,\\[4pt]
|\Thetalk^-|^{1/2} & \mbox{otherwise},
\end{array}\right.
\label{eq.expr-thetam} \\
\thetalk^+ & = & \left\lbrace \begin{array}{ll} 
+ \rmi \,\sgn(\lambda)\,|\Thetalk^+|^{1/2} & \mbox{if } (k,\lambda) \in \zEI \cup \zDI,\\[4pt]
- \rmi \,\sgn(\lambda)\,|\Thetalk^+|^{1/2} & \mbox{if } (k,\lambda) \in \zDD,\\[4pt]
|\Thetalk^+|^{1/2} & \mbox{otherwise}.
\end{array}\right.
\label{eq.expr-thetap}
\end{eqnarray} 

In addition to the spectral zones defined above, we have to investigate the possible singularities of the Green function in the $(k,\lambda)$-plane, that is, the pairs $(k,\lambda)$ for which the one-sided limit $\calW_{k,\lambda}$ of $\calW_{k,\lambda+\rmi \eta}$ (see (\ref{eq.coeff-psi})) vanishes. Hence we have to solve the following \emph{dispersion equation}, seen as an equation in the $(k,\lambda)$-plane:
\begin{equation}\label{eq.disp}
\frac{\thetalk^{-}}{\mu_{\lambda}^{-}}+ \frac{\thetalk^{+}}{\mu_{\lambda}^{+}} = 0.
\end{equation}

\begin{lemma}\label{lem.sing}
Define $K := \eps_0 \mu_0 \,(\Omegam^2-\Omegae^2).$ The solutions $(k,\lambda)$ to {\rm (\ref{eq.disp})} are given by 
$$\overline{\zEE}=\zEE \cup \{( k_{\rm c},  \lambda_{\rm c}), ( k_{\rm c},  -\lambda_{\rm c}),( -k_{\rm c},  \lambda_{\rm c}), ( -k_{\rm c},  -\lambda_{\rm c}) \}$$ 
where (see Figure {\rm \ref{fig.speczones}} for an illustration)
\begin{equation}\label{eq.defspeczonesZEE}
\zEE := \left\{ (k,\lambda)\in \bbR^2 \mid |k| > k_{\rm c} \mbox{ and } |\lambda| = \lambda_\scE(k) \right\}.
\end{equation}
The critical  value $k_{\rm c}$ has been defined in Lemma {\rm \ref{lem.disp}} and
\begin{equation*}
\lambda_\scE(k) := 
\left\{ \begin{array}{ll}
\displaystyle \Omegam \ \sqrt{\frac{1}{2} + \frac{k^2}{K} - \sgn(K)\,\sqrt{\frac{1}{4} + \frac{k^4}{K^2}}} & \mbox{if } K \neq 0, \\
\displaystyle \frac{\Omegam}{\sqrt{2}} & \mbox{if } K = 0.
\end{array}\right.
\end{equation*}
The function $k \mapsto \lambda_\scE(k)$ is strictly decreasing on $[k_{\rm c},+\infty)$ if $K < 0,$ \textit{i.e.}, $\Omegam < \Omegae$ (respectively, strictly increasing if $K > 0,$ \textit{i.e.}, $\Omegam > \Omegae).$ Moreover $\lambda_\scE(k_{\rm c}) = \lambda_{\rm c} := \lambda_0(k_{\rm c}) = \lambda_\scI(k_{\rm c})$ and $\lambda_\scE(k) = \Omegam / \sqrt{2} + O(k^{-2})$ as $k \to +\infty.$
\end{lemma}

\begin{proof}
First notice that if $(k,\lambda)$ is a solution to (\ref{eq.disp}), then
\begin{equation}\label{eq.disp-carre}
 \left(\frac{\thetalk^{-}}{\mu_{\lambda}^{-}} +\frac{\thetalk^{+}}{\mu_{\lambda}^{+}} \right) \, \left(\frac{\thetalk^{-}}{\mu_{\lambda}^{-}} -\frac{\thetalk^{+}}{\mu_{\lambda}^{+}} \right)=0  \quad \Longleftrightarrow\quad \frac{\Thetalk^{-}}{(\mu_{\lambda}^{-})^2} = \frac{\Thetalk^{+}}{(\mu_{\lambda}^{+})^2} \,.
\end{equation}
The proof is made of two steps.\\[6pt]
{\it Step 1.} We first show that $\overline{\zEE} \subset \bbR^2 \setminus \big( \zDD\cup \zDE \cup \zDI \cup \zEI\big)$. \\[6pt]
Indeed, (\ref{eq.disp-carre}) implies that $\Thetalk^{-}$ and $\Thetalk^{+}$ either have the same sign or vanish simultaneously. Hence
\begin{itemize}
\item  $\zDE$ and $\zEI$ contain no solution since $\Thetalk^{-}$ and $\Thetalk^{+}$ have opposite signs in these zones (first statement of Lemma \ref{lem.disp}).
\item If $(k,\lambda) \in \zDD,$ $\thetalk^\pm = -\rmi \,\sgn(\lambda)\,|\Thetalk^\pm |^{1/2} $ (from (\ref{eq.expr-thetam}) and (\ref{eq.expr-thetap})) and $\mu_{\lambda}^{\pm} > 0$ (since $|\lambda| > \Omegam)$. This shows that $(k,\lambda)$ cannot satisfy (\ref{eq.disp}). A similar argument works for $\zDI$ ($\mu_{\lambda}^{-}$ and $\mu_{\lambda}^{+}$ have opposite signs in this zone, but $\Imag(\thetalk^{+})$ and $\Imag(\thetalk^{-})$ also have opposite signs).
\end{itemize}
On the other hand, the curves $\lambda = \lambda_\scE(k),$  $\lambda = \lambda_\scD(k)$ and $\lambda = \lambda_\scI(k)$ contain no solution to (\ref{eq.disp}), except the \emph{critical points} $(k_{\rm c},\lambda_{\rm c}),$ $(-k_{\rm c},\lambda_{\rm c}),$ $(k_{\rm c},-\lambda_{\rm c})$ and $(-k_{\rm c},-\lambda_{\rm c})$ where both $\Thetalk^{\pm}$ vanish. 
To sum up, apart from the critical points, the possible solutions to (\ref{eq.disp}) are located outside the closure of all previously defined zones, that is, in the ``white area'' of Figure \ref{fig.speczones}.\\[6pt]
{\it Step 2.} In this ``white area'', $\thetalk^\pm > 0,$ so solutions may occur only if $\mu_{\lambda}^{-}$ and $\mu_{\lambda}^{+}$ have opposite signs, that is, in the sub-area located in $|\lambda| < \Omegam.$ In this sub-area, (\ref{eq.disp}) and (\ref{eq.disp-carre}) are equivalent since $\thetalk^{-}/\mu_{\lambda}^{-}+\thetalk^{+}/\mu_{\lambda}^{-} \neq 0$. Using (\ref{eq.def-eps}), (\ref{eq.def-mu}) and (\ref{eq.defTheta}), (\ref{eq.disp-carre}) can be written as
\begin{equation}
Q_k(\lambda^2) = 0 \quad\mbox{where}\quad Q_k(X) := K\,X^2 - \Omegam^2\,(2k^2 + K)\,X + k^2\Omegam^4,
\label{eq.Def-Qk}
\end{equation}
from which we infer that
\begin{equation*}
\lambda^2 = X_k^\pm := \Omegam^2 \left( \frac{1}{2} + \frac{k^2}{K} \pm \sqrt{\frac{1}{4} + \frac{k^4}{K^2}} \right).
\end{equation*}
If $K<0,$ the only positive root is $X_k^+$ which is strictly decreasing on $[k_{\rm c},+\infty)$ with range $(\Omegam^2 / 2,\lambda_{\rm c}^2].$ Moreover, the curve $\lambda = (X_k^+)^{1/2}$ crosses $\lambda = \lambda_\scI(k)$ only at $(k_{\rm c},\lambda_{\rm c}),$ which shows that the former is located in the above mentioned sub-area. On the other hand, if $K>0,$ both roots $X_k^-$ and $X_k^+$ are positive, but $X_k^+ > \Omegam^2,$ so the curve $\lambda = (X_k^+)^{1/2}$ is now outside the sub-area. The other root yields the only solution in the admissible sub-area and it is strictly increasing on $[k_{\rm c},+\infty)$ with range $[\lambda_{\rm c}^2, \Omegam^2 / 2)$. Finally, if K=0, the only root of $Q_k(X)$ is $X = \Omegam^2 / 2 = \lambda_{\rm c}^2.$  
\end{proof}

% -------------------------------------------------------------------------------------------------------------------------------
\subsubsection{Physical interpretation of the spectral zones}
\label{s.phys-spect-zones}
The various zones introduced above are related to various types of waves in both media, which can be either propagative or evanescent. As already mentioned, the indices $\scD,$ $\scI$ and $\scE$ we chose to qualify these zones stand respectively for \emph{direct}, \emph{inverse} and \emph{evanescent}. The first two, $\scD$ and $\scI,$ are related to propagative waves which can be either direct or inverse waves (in the Drude medium), whereas $\scE$ means evanescent, that is, absence of propagation. As shown below, for each pair of indices characterizing the various zones $\zDD,$ $\zDE,$ $\zDI,$ $\zEI$ and $\zEE,$ the former indicates the behavior of the vacuum $(x<0)$ and the latter, that of the Drude material $(x>0).$

To see this, substitute formulas (\ref{eq.expr-thetam}) and (\ref{eq.expr-thetap}) in (\ref{eq.psim}) and (\ref{eq.psip}), which yields the one-sided limits $\psilpm$ of $\psipm$ in the sense of (\ref{eq.convention}). The aim of this section is to show the physical interpretation of these functions as superpositions of elementary waves. For simplicity, we restrict ourselves to $\lambda>0$ when interpreting their direction of propagation. 

First notice that in each half line $x<0$ and $x>0$, function $\rme^{-\rmi \lambda t}\, \psilpm$ is a linear combination of terms of the form $u_{\kappa}=\rme^{\rmi (\kappa x-\lambda t)}, \mbox{ with }\, \kappa\in \bbR$ or $\kappa\in \rmi \, \bbR$, where for a fixed $k$, $\kappa$ is a certain function of $\lambda$, denoted $\kappa=\kappa(\lambda)$.
\begin{itemize}
\item When $\kappa$ is purely imaginary, $u_{\kappa}$ is an \emph{evanescent} wave in the direction $x>0$ or $x<0$.
\item When $\kappa$ is real, $u_{\kappa}$ is a \emph{propagative} wave whose phase velocity is given by
\begin{equation}\label{eq.defphasevel}
v_{\phi}(\lambda) := \frac{\lambda}{\kappa}.
\end{equation}
Locally (if $\rmd \kappa(\lambda)/\rmd \lambda\neq 0$), $\kappa=\kappa(\lambda)$ defines $\lambda$ as a function of $\kappa$ and the group velocity of $u_{\kappa}$ is defined by
\begin{equation}\label{eq.defphasegroup}
v_{g}(\kappa) := \frac{\rmd \lambda(\kappa)}{\rmd \kappa}=\Big(\frac{\rmd \kappa(\lambda)}{\rmd \lambda}\Big)^{-1}.
\end{equation}
If the product $v_{g}\, v_{\phi}$ is positive, (\textit{i.e.}, when the group and phase velocities point in the same direction), one says that $u_{\kappa}$ is a \emph{direct propagative} wave. If the product $v_{g}\, v_{\phi}$ is negative, (\textit{i.e.}, when the group and phase velocities point in opposite directions), one says that $u_{\kappa}$ is an \emph{inverse propagative} wave. 

One uses also the notion of group velocity, which characterizes \emph{the direction of propagation} (the direction of the energy transport), to distinguish between incoming and outgoing propagative waves. One says that a propagative wave $u_{\kappa}$ is \emph{incoming} (respectively, \emph{outgoing}) in the region $\pm x>0$ if its group velocity points towards the interface (respectively, towards $\pm \infty$).
\end{itemize}

In our case, for a real  $\kappa$, one denotes by $v_{\phi}^{\pm}$ and $v_{g}^{\pm}$, the phase and group velocities of a propagative wave $u_{\kappa}$ in the region $\pm x>0$.
Let us derive a general formula for the product $v_{g}^{\pm}\, v_{\phi}^{\pm}$.
In our case, the dispersion relation (\ref{eq.defTheta}) can be rewritten as 
$$ -\kappa^{\pm}(\lambda)^2=\Theta_{\lambda,k}^{\pm}=k^2- \lambda^2\, \eps_\lambda^\pm\, \mu_\lambda^\pm.$$
By differentiating this latter expression with respect to $\lambda,$ one gets
$$
2 \kappa^{\pm}\,\frac{\rmd \kappa^{\pm}}{\rmd \lambda}=-\frac{\partial \Theta_{\lambda,k}}{\partial \lambda}.
$$
From (\ref{eq.defphasevel}) and (\ref{eq.defphasegroup}), this yields
\begin{equation}\label{eq.vphvg}
v_{\phi}^{\pm} v_{g}^{\pm}=-2 \lambda \Big(\frac{\partial \Theta_{\lambda,k}^{\pm}}{\partial \lambda}\Big)^{-1}.
\end{equation}
In the vacuum, one easily check that formula ({\ref{eq.vphvg}}) leads to the classical relation
$$
v_{\phi}^{-}v_{g}^{-}=(\eps_0 \mu_0)^{-1}>0.
$$
Thus all the propagative waves are direct propagative. In the Drude material, one has
$$
\frac{\partial \Theta_{\lambda,k}^{+}}{\partial \lambda}= -\frac{\rmd (\lambda \eps_\lambda^+)}{\rmd \lambda} \, \mu_\lambda^+ -\frac{\rmd (\lambda \mu_\lambda^+)}{\rmd \lambda} \, \eps_\lambda^+ .
$$
and by expressing the derivatives in the latter expression, one can rewrite ({\ref{eq.vphvg}}) as
\begin{equation}\label{eq.sign}
v_{\phi}^{+} v_{g}^{+}=2 \lambda \left[ \mu_0 \, \left( 1+\frac{\Omegam^2}{\lambda^2}\right) \eps_\lambda^++  \eps_0 \left( 1+\frac{\Omegae^2}{\lambda^2}\right) \mu_\lambda^+ \right]^{-1}.
\end{equation}
One sees with this last expression that the sign of $v_{\phi}^{+} v_{g}^{+}$ depends on the sign of $ \eps_\lambda^+$ and $ \mu_\lambda^+$.

Let us look at  what happens in the different spectral zones. This study is summarized in the tables of Figures \ref{fig.modes} and \ref{fig.plasmons}. 
 
$\bullet$ Suppose first that $(k,\lambda) \in \zDD.$ In this spectral zone, $\thetalk^\pm = - \rmi \,|\Thetalk^\pm|^{1/2}$ thus the corresponding  waves $u_{\kappa}$ are propagative on both 
side of the interface. Moreover, as $ \eps_\lambda^+>0$ and $ \mu_\lambda^+>0$, the product (\ref{eq.sign}) is positive: the propagative waves are direct propagative waves in both media and  thus their direction of propagation is the sign of their wave number $\kappa$. (\ref{eq.psim}) shows that for $x<0,$ function $\psilm$ is an oscillating wave of amplitude 1 and wave number $\kappa=-|\Thetalk^{-}|^{1/2}$ which propagates towards $-\infty$, whereas for $x>0,$ it is a superposition of a wave of amplitude $A_{k,\lambda,-1}$ and wave number $\kappa=-|\Thetalk^{+}|^{1/2}$ which propagates towards the origin  and a wave of amplitude $B_{k,\lambda,-1}$ and wave number $\kappa=|\Thetalk^{+}|^{1/2}$ which propagates towards $+\infty.$ In other words, $\psilm$ can be interpreted as an \emph{incoming incident} wave of amplitude $A_{k,\lambda,-1}$ propagating from $+\infty$ whose diffraction on the interface $x=0$ generates two \emph{outgoing} waves: a \emph{reflected} wave of amplitude $B_{k,\lambda,-1}$ and a \emph{transmitted} wave of amplitude 1. Similarly, $\psilp$ can be interpreted as an \emph{incoming incident} wave of amplitude $A_{k,\lambda,+1}$ and wave number $\kappa=|\Thetalk^{-}|^{1/2}$ propagating from $-\infty$ whose diffraction on the interface $x=0$ generates two \emph{outgoing} waves: a \emph{reflected} wave of amplitude $B_{k,\lambda,+1}$ and wave number $\kappa=-|\Thetalk^{-}|^{1/2}$ and a \emph{transmitted} wave of amplitude 1 and wave number $\kappa=|\Thetalk^{+}|^{1/2}$. 

$\bullet$ If $(k,\lambda) \in \zDE,$ we still have $\thetalk^- = - \rmi \,|\Thetalk^-|^{1/2}$ but now $\thetalk^+ = |\Thetalk^+|^{1/2},$ which means that propagative waves are no longer allowed in the Drude material: the only possible waves are exponentially decreasing or increasing. Proceeding as above, we see that on the one hand, $\psilm$ can be interpreted as an \emph{incident} wave of amplitude $A_{k,\lambda,-1}$ which is exponentially increasing as $x \to +\infty$ whose diffraction on the interface $x=0$ generates two waves: a \emph{reflected evanescent} wave of amplitude $B_{k,\lambda,-1}$ and a \emph{transmitted outgoing} wave of amplitude 1 and wave number $\kappa=-|\Thetalk^{-}|^{1/2}$. On the other hand, $\psilp$ can be interpreted as an \emph{incoming incident} wave of amplitude $A_{k,\lambda,+1}$ and wave number $\kappa=|\Thetalk^{-}|^{1/2}$ propagating from $-\infty$ whose diffraction on the interface $x=0$ generates two waves: a \emph{reflected outgoing} wave of amplitude $B_{k,\lambda,+1}$ and wave number $\kappa=-|\Thetalk^{-}|^{1/2}$ and a \emph{transmitted evanescent} wave of amplitude 1.

$\bullet$ Suppose now that $(k,\lambda) \in \zDI.$ We still have $\thetalk^- = - \rmi \,|\Thetalk^-|^{1/2}$ but propagative waves are again allowed in the Drude material since $\thetalk^+ = +\rmi\,|\Thetalk^+|^{1/2}.$  Compared with the case where $(k,\lambda) \in \zDD,$ we simply have a change of sign in the expression of $\thetalk^+,$ which amounts to reversing all sign of the wave numbers $\kappa$ associated to the propagative waves in the Drude material. Hence for $x >0,$ one might be tempted to interchange the words \emph{incoming} and \emph{outgoing} in the above interpretation of $\psilpm$ for $(k,\lambda) \in \zDD.$ It would be wrong! Indeed, the change of sign of the imaginary part of $\thetalk^+$ yields an opposite sign of the \emph{phase velocity} in the Drude material, but not of the \emph{group velocity} which characterizes the direction of the energy transport. As we see in (\ref{eq.sign}), since
$\eps_{\lambda}^{+}$ and $\mu_{\lambda}^{+}$ are both negative in this spectral zone, the phase and group velocity are pointing in opposite directions. Hence, the waves $u_{\kappa}$ are \emph{inverse propagative} in the Drude material. %The Drude material behaves here as a negative index material, which yields a negative refraction phenomenon \cite {Zio-01}, that is represented on the figure \ref{fig.modes} with the directions of the phase velocities of the different $u_{\kappa}$ which compose $\psilpm$. 

$\bullet$ Assuming now that $(k,\lambda) \in \zEI,$ we have $\thetalk^- = |\Thetalk^-|^{1/2}$ and $\thetalk^+ = +\rmi\,|\Thetalk^+|^{1/2}.$  The Drude material behaves  behaves as a negative material as in the previous case (since $\eps_{\lambda}^{+}$ and $\mu_{\lambda}^{+}$ are also both negative in this spectral zone), but propagative waves are no longer allowed in the vacuum.

\begin{figure}[!t]
\centering
\includegraphics[width=0.8\textwidth]{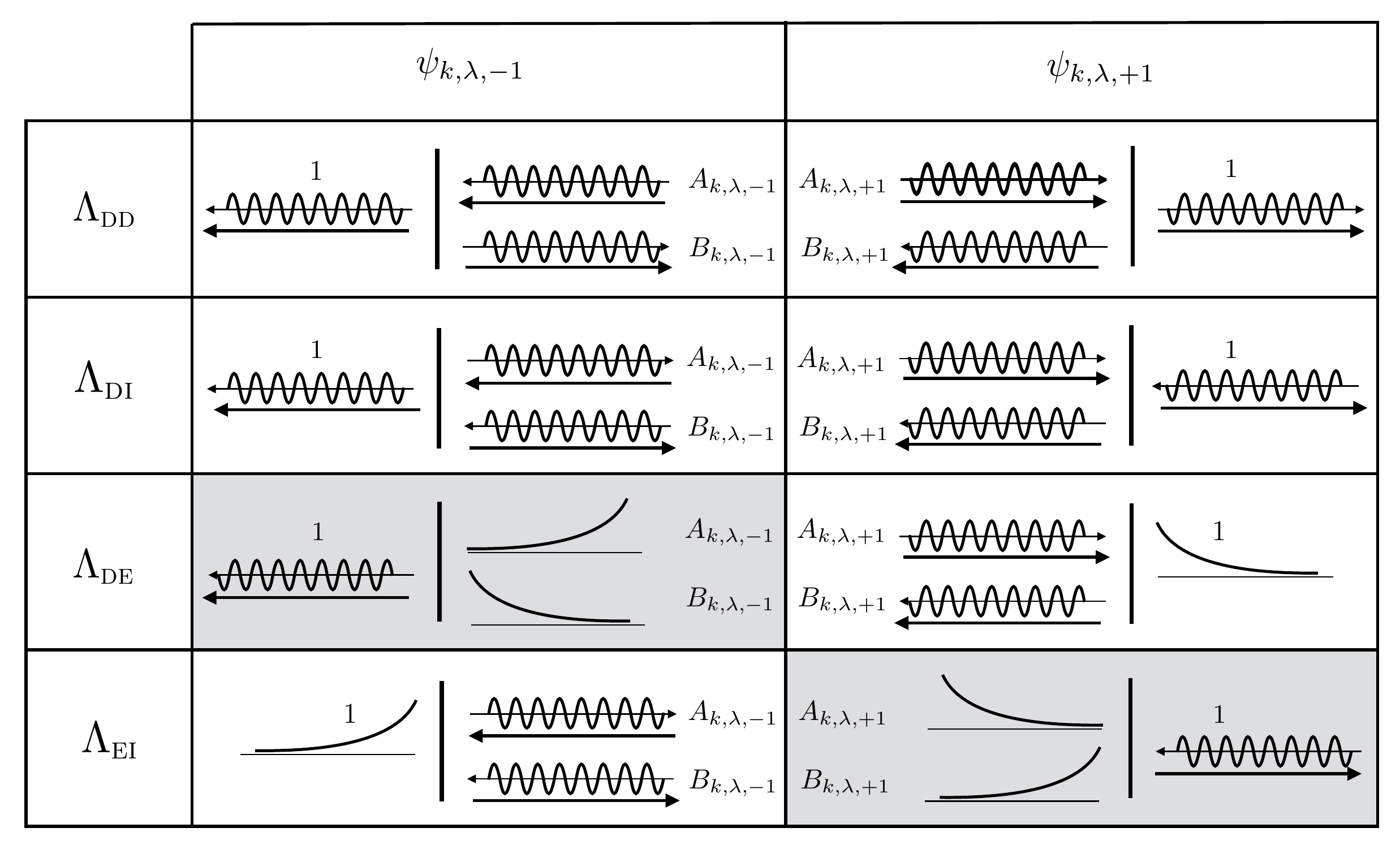}
\caption{Description of the functions $\psilpm$ in the different spectral zones $\Lambda_{\DD}$, $\Lambda_{\DI}$, $\Lambda_{\DE}$ and $\Lambda_{\EI}$. Propagative waves are represented by an oscillating function and the larger and smaller arrow indicate respectively the direction of the group velocity and the phase velocity. Evanescent waves are represented with decreasing exponential. The ``unphysical'' functions  $\psilpm$ which contain increasing exponential behavior are colored in gray.}
\label{fig.modes}
\end{figure}

\begin{figure}[!t]
\centering
\includegraphics[width=0.8\textwidth]{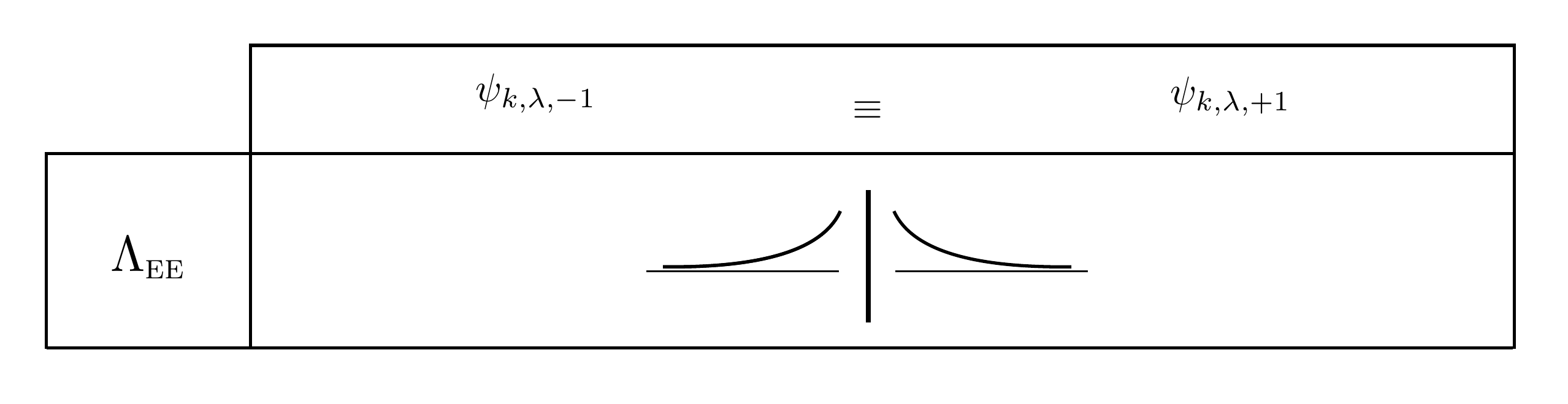}
\caption{Description of the plasmonic wave  $\psi_{k,\lambda,-1}=\psi_{k,\lambda,1}$ in the spectral zone $\Lambda_{\EE}$.}
\label{fig.plasmons}
\end{figure}

$\bullet$ Finally, if $(k,\lambda) \in \zEE,$ both $\thetalk^\pm$ are real and positive, so that waves are evanescent on both sides of the interface $x=0.$ In this case, functions $\psi_{k,\lambda,-1}$ and $\psi_{k,\lambda,+1}$ are equal and real. 
Such a wave is referred as plasmonic wave in the physical literature (see \cite{Mai-07}). Note that in the particular case $\Omegae=\Omegam$, we have $\mu_{\lambda}^{+}=-\mu_0$ and $\eps_{\lambda}^{+}=-\eps_0$ so that $\thetalk^{+}=\thetalk^{-}$. This shows that $\psi_{k,\lambda,-1} = \psi_{k,\lambda,+1}$ is an even function of $x.$\\

Among the various categories of waves described above, some of them may be called ``unphysical'' since they involve an exponentially increasing behavior at infinity, which occurs for $\psilm$ if $(k,\lambda) \in \zDE,$ as well as $\psilp$ if $(k,\lambda) \in \zEI.$ Fortunately these waves will disappear by limiting absoption in the upcoming Proposition \ref{prop.vecgen1}: only the ``physical'' ones are needed to describe the spectral behavior of $\bbAk.$ That is why in Figure \ref{fig.speczones}, all the curves represent \emph{spectral cuts} since they are the boundaries where some $\psilpm$ appear or disappear.

% -------------------------------------------------------------------------------------------------------------------------------
\subsubsection{Boundary values of the Green function}
We are now able to express the one-sided limit $g_{k,\lambda}$ of the Green function defined in (\ref{eq.defgreen}). The following two propositions, which distinguish the case of $\zEE$ from the other zones, provide us convenient expressions of the quantities related to $\gzk$ which are needed in Stone's formulas.

\begin{proposition}\label{prop.vecgen1}
For all $(k,\lambda) \in \Lambda_\scZ$ with $\scZ\in \left\{ \DD, \DI, \DE, \EI\right\}$, we have
\begin{equation*}
\frac{1}{\pi}\, \Imag( \lambda \,\glk(x,x'))
= \sum_{j\in J_{\scZ}} \overline{w_{k,\lambda,j}(x)}\,w_{k,\lambda,j}(x')
\quad\mbox{where}
\end{equation*}
\begin{equation} \label{defwkl}
w_{k,\lambda,\pm 1}(x) := \frac{1}{|\calW_{k,\lambda}|}\ \left| \frac{\lambda\,\thetalk^\mp}{\pi\,\mu_{\lambda}^\mp}\right|^{1/2}\,\psi_{k,\lambda,\pm 1}(x)
\quad\mbox{and}\quad
J_{\scZ}:=\left\lbrace\begin{array}{ll}
 \left\{ -1,+1\right\} & \mbox{ if }   \scZ = \DD \mbox{ or }\DI, \\
 \left\{ +1\right\}        &  \mbox{ if } \scZ = \DE,\\
 \left\{ -1\right\}        &  \mbox{ if } \scZ = \EI.  
       \end{array}
 \right.
\end{equation}
\end{proposition}
\begin{proof}
Suppose first that $(k,\lambda) \in \zDD.$ If this case, (\ref{eq.expr-thetam}) and (\ref{eq.expr-thetap}) tell us that $\thetalk^{\pm}=-\rmi\,\sgn(\lambda)\,|\thetalk^{\pm}|.$ In order to express the imaginary part of the Green function, we rewrite the expression of Proposition \ref{prop.green} in terms of functions $\clk$ and $\slk$ which are real-valued. We obtain that $\Imag\big(\lambda\, \glk(x,x')\big)$ is equal to
\begin{equation*}
|\lambda|\,|\calW_{k,\lambda}|^{-1} \, \left( \clk\big(\min(x,x')\big)\,\clk\big(\max(x,x')\big) 
           + \frac{|\thetalk^{-}\,\thetalk^{+}|}{\mu_{\lambda}^{-}\,\mu_{\lambda}^{+}} \,
           \slk\big(\min(x,x')\big)\,\slk\big(\max(x,x')\big) \right).
\end{equation*}
This expression shows that we can replace $\min(x,x')$ and  $\max(x,x')$ by $x$ and $x'$ respectively. It can be written in matrix form as
\begin{equation*}
\Imag\big(\lambda\, \glk(x,x')\big) = |\lambda|\,|\calW_{k,\lambda}|^{-1} \ 
\begin{pmatrix} \clk(x) \\ \slk(x) \end{pmatrix}^{\!*} 
\begin{pmatrix} 1 &0 \\ 0 & \displaystyle \frac{|\thetalk^{-}\,\thetalk^{+}|}{\mu_{\lambda}^{-}\,\mu_{\lambda}^{+}} \end{pmatrix}
\begin{pmatrix} \clk(x') \\ \slk(x') \end{pmatrix},
\end{equation*}
where the symbol $*$ denotes the conjugate transpose of a matrix. Note that the conjugation could be omitted since $\clk$ and $\slk$ are real-valued. However it is useful for the next step which consists in rewriting this expression in terms of $\psilpm$ by noticing that
\begin{equation*}
\begin{pmatrix} \clk(x) \\ \slk(x) \end{pmatrix}
= \frac{1}{\calW_{k,\lambda}} \, \begin{pmatrix} \thetalk^+/\mu_{\lambda}^+ & \thetalk^-/\mu_{\lambda}^- \\ 1 & -1 \end{pmatrix}
\begin{pmatrix} \psilm(x) \\ \psilp(x) \end{pmatrix},
\end{equation*}which follows from the definition (\ref{eq.coeff-psi}) of $\psipm.$ Therefore, after some calculations exploiting that $\thetalk^{\pm}$ is purely imaginary, one obtains
\begin{equation*}
\Imag\big(\lambda\, \glk(x,x')\big) = |\lambda|\,|\calW_{k,\lambda}|^{-2} \ 
\begin{pmatrix} \psilm(x) \\ \psilp(x) \end{pmatrix}^* 
\begin{pmatrix} |\thetalk^+|/\mu_{\lambda}^+ & 0 \\ 0 & |\thetalk^-|/\mu_{\lambda}^- \end{pmatrix}
\begin{pmatrix} \psilm(x') \\ \psilp(x') \end{pmatrix},
\end{equation*}
which yields the announced result, for $\mu_{\lambda}^\pm > 0$ (since $\lambda > \Omegam).$\\

Suppose now that $(k,\lambda) \in \zDI$. Compared with the previous case, we have now $\thetalk^{+}=+\rmi\,\sgn(\lambda)\,|\thetalk^{+}|$  and $\mu_{\lambda}^+$ becomes negative (since $\lambda < \Omegam).$ It is readily seen that in this case, the calculations above hold true if we replace $\mu_{\lambda}^+$ by $|\mu_{\lambda}^+|.$\\

When $(k,\lambda) \in \zDE$, we still have $\thetalk^{-}=-\rmi\,\sgn(\lambda)\,|\thetalk^{-}|,$ but $\thetalk^{+}$ is now a positive number, which shows that $\psilp$ is a real-valued function. Following the same steps as above, we obtain
\begin{equation*}
\Imag\big(\lambda\, \glk(x,x')\big) = |\lambda|\,|\calW_{k,\lambda}|^{-2} \ 
\begin{pmatrix} \psilm(x) \\ \psilp(x) \end{pmatrix}^* 
\begin{pmatrix} 0 & 0 \\ 0 & |\thetalk^-|/\mu_{\lambda}^- \end{pmatrix}
\begin{pmatrix} \psilm(x') \\ \psilp(x') \end{pmatrix}.
\end{equation*}

Finally, if $(k,\lambda) \in \zEI$, we have $\thetalk^{-}>0$ and $\thetalk^{+}=+\rmi\,\sgn(\lambda)\,|\thetalk^{+}|.$ As in $\zDI,$ the Drude material behaves as a negative material $(\mu_{\lambda}^+ < 0).$ In this case, we obtain
\begin{equation*}
\Imag\big(\lambda\, \glk(x,x')\big) = |\lambda|\,|\calW_{k,\lambda}|^{-2} \ 
\begin{pmatrix} \psilm(x) \\ \psilp(x) \end{pmatrix}^* 
\begin{pmatrix} |\thetalk^+/\mu_{\lambda}^+| & 0 \\ 0 & 0 \end{pmatrix}
\begin{pmatrix} \psilm(x') \\ \psilp(x') \end{pmatrix},
\end{equation*}
which completes the proof.
\end{proof}

\begin{proposition}\label{prop.vecgen2}
For all $(k,\lambda) \in \zEE$, we have
\begin{equation*}
 (\lambda+\rmi\eta) \,g_{k,\lambda+\rmi\eta}(x,x')= \rmi \eta^{-1} \overline{w_{k,\lambda,0}(x)} \, w_{k,\lambda,0}(x') + O(1) \quad\mbox{as } \eta \searrow 0,
\end{equation*}
where the real-valued function $w_{k,\lambda,0}$ function is given by 
\begin{equation} \label{defwklzero}
w_{k,\lambda,0}(x) := \frac{\lambda^2\, |\mu_{\lambda}^{+} \, \thetalk^{+}|^{1/2}}{\Omegam (4k^4 +(\eps_0\mu_0)^2(\Omegae^2-\Omegam^2)^2)^{1/4}} \ \rme^{-\thetalk^{\pm} |x|} 
\quad\mbox{if }\pm x > 0,
\end{equation}
and the remainder $O(1)$ is uniform in $(x,x')$ on any compact set of $\bbR^2$. Finally, if $|k|=k_{\rm c}$, then  
$g_{k,\zeta}(x,x') = O(|\zeta\mp \lambda_{\rm c} |^{-1/2})$ as $\zeta \to  \pm  \lambda_{\rm c} $ from the upper-half plane uniformly with respect to $(x,x')$ on any compact set of $\bbR^2.$
As a consequence,
$$
\lim_{\eta \searrow 0}(\lambda+\rmi \eta) \,g_{k,\lambda+\rmi\eta}(x,x')=0 \ \mbox{ for } \ |k|=k_{\rm c} \ \mbox{ and } \ |\lambda|=\lambda_{\rm c}.
$$
\end{proposition}

\begin{proof}
Let $(k,\lambda) \in \zEE.$ It is readily seen that the definition (\ref{eq.coeff-psi}) of $\calW_{k,\zeta}$ can be written equivalently (see  also the proof of Lemma \ref{lem.sing})
\begin{equation*}
\calW_{k,\zeta} = \frac{Q_k(\zeta^2)}{\zeta^4 \, (\mu_{\zeta}^{+})^{2} \, (\thetazk^{-}/\mu_{\zeta}^{-} - \thetazk^{+}/\mu^{+}_{\zeta})},
\end{equation*}
where $Q_k$ is the polynomial defined in (\ref{eq.Def-Qk}). As $Q_k(\lambda_\scE(k)^2)=0$ and one can compute that $Q'_k(\lambda_\scE(k)^2) = - \Omegam^2 \sqrt{4k^4 + K^2},$ we deduce after some manipulations that, for $\lambda = \pm\lambda_\scE(k),$
\begin{equation*}
\frac{1}{\calW_{k,\lambda+\rmi\eta}} = \frac{\rmi}{\eta} \, \frac{\lambda^3 \, |\mu_{\lambda}^{+}| \, \thetalk^{+}}{\Omegam^2 \sqrt{4k^4 + K^2}} + O(1) 
\quad\mbox{as }\eta \searrow 0.
\end{equation*}
(where we used the dispersion relation (\ref{eq.disp}) and the fact that $\mu^+_{\lambda}=- |\mu_{\lambda}^{+}|).$ The announced result follows from the expression (\ref{eq.coeff-psi}) of $\gzk,$ since $\psipm$ are analytic functions of $\zeta$ near $\lambda$ which both tend to $\psilm = \psilp.$\\

Finally, if $|k|= k_{\rm c}$, $\lambda_{\rm c}$ and $-\lambda_{\rm c}$ are simple zeros of $\zeta \mapsto \Theta_{k,\zeta}^{\pm},$ it is clear (see formula (\ref{eq.coeff-psi})) that $\calW_{ k,\zeta} = O(|\zeta\mp \lambda_{\rm c}|^{1/2})$ as $\zeta \to  \pm  \lambda_{\rm c} $ from the upper-half plane.  Thus, the conclusion follows again from the expression (\ref{eq.coeff-psi}) of $\gzk,$ since $\psipm$ are here uniformly bounded in $\zeta$, $x$ and $x'$ when $\zeta$ belongs to a vicinity of $\pm  \lambda_{\rm c}$ and $(x,x')$ to any compact set of $\bbR^2.$
\end{proof}

% ===============================================================================================================================
\subsection{Diagonalization of the reduced Hamiltonian}

% -------------------------------------------------------------------------------------------------------------------------------
\subsubsection{Spectral measure of the reduced Hamiltonian}
\label{s.spec-meas-Ak}
As we shall see, the spectral zones introduced in \S\ref{s.spectr-zones} actually show us the location of the spectrum of $\bbAk$ for each $k \in \bbR:$ we simply have to extract the associated sections of these zones, that is, the sets
$$
\Lambda_\scZ(k) := \{ \lambda\in\bbR \mid (k,\lambda) \in \Lambda_{\scZ}\} \quad\mbox{for } \scZ \in \calZ := \{\DD,\DI,\DE,\EI,\EE\},
$$
which all are unions of symmetric intervals with respect to $\lambda = 0.$ This is a by-product of the following proposition which tells us that apart from the three eigenvalues $-\Omegam,0$ and $\Omegam$ (see Proposition \ref{prop.eigvalAk}), the spectrum of $\bbA_k$ is composed of two parts: an  \emph{absolutely continuous spectrum} defined by $\cup_{\scZ \in \calZ\setminus\{\EE\} }\overline{\Lambda_{\scZ}(k)}$  if $|k|\leq k_{\rm c}$ and by $\cup_{\scZ \in \calZ\setminus\{\DI,\EE\} }\overline{\Lambda_{\scZ}(k)}$ if $|k|> k_{\rm c}$ and a \emph{pure point spectrum} given by $\zEE(k) = \{ -\lambda_\scE(k),+\lambda_\scE(k) \}$ if $|k|> k_{\rm c}$ (we point out that there is no \emph{singularly continuous spectrum}).

This proposition yields a convenient expression of the spectral projection $\bbE_k(\Lambda)$ of $\bbA_k$ for $\Lambda \subset \bbR\setminus\{- \Omegam, 0, \Omegam\}.$ As $\bbE_k(\Lambda)$ is a projection on an invariant subspace by $\bbA_k,$ the canonical way to express such a projection is to use a \emph{spectral basis}. Propositions \ref{prop.vecgen1} and \ref{prop.vecgen2} provide us such a basis: these are the vector fields deduced from the $w_{k,\lambda,j}$'s (see (\ref{defwkl}) and (\ref{defwklzero})) by the ``vectorizator'' defined in (\ref{eq.opV}), \textit{i.e.},
\begin{equation} \label{eq:defwijk}
\wlkj := \bbV_{k,\lambda}\ w_{k,\lambda,j} \quad\mbox{for }
\lambda \in \Lambda_\scZ(k) \mbox{ and } j \in J_\scZ:=
\left\lbrace\begin{array}{ll}
\{ -1,+1 \} & \mbox{ if } \scZ = \DD \mbox{ or }\DI, \\
\{ +1 \}    & \mbox{ if } \scZ = \DE,\\
\{ -1 \}    & \mbox{ if } \scZ = \EI, \\
\{ 0 \}     & \mbox{ if } \scZ = \EE,
\end{array}
\right.
\end{equation}
for each zone $\scZ \in \calZ.$ As we shall see afterwards, the knowledge of these vector fields leads to a diagonal form of $\bbA_k.$ These are \emph{generalized eigenfunctions} of $\bbA_k.$

\begin{proposition}\label{prop.specmeasureAk}
Let $k \in \bbR,$ $\bbE_k$ the spectral measure of $\bbA_k$ and $\bU \in \Dx$ (see {\rm (\ref{eq.defDx})}). For all interval $\Lambda \subset \Lambda_\scZ(k)$ with $\scZ \in \calZ := \{\DD,\DI,\DE,\EI\},$ we have
\begin{eqnarray}
\left\|\bbE_k(\Lambda)\,\bU\right\|^2_{\indHx}  & = & 
\int_\Lambda \sum_{j \in J_\scZ} \left| \langle \bU\,,\wlkj \rangle_{\indHx,s} \right|^2\, \rmd \lambda \quad\mbox{and} \label{eq.specmesAk-cont} \\
\left\|\bbE_k(\{\pm \lambda_\scE(k)\})\,\bU\right\|^2_{\indHx} & = &
\left\{ \begin{array}{ll}
\left| ( \bU\,,\boldsymbol{W}_{\!\! k,\pm \lambda_\scE(k),0} )_\indHx \right|^2 & \mbox{if } |k|>k_{\rm c}, \\[5pt]
0 & \mbox{if } |k| = k_{\rm c}.
\end{array}\right. \label{eq.specmesAk-ponct}
\end{eqnarray}
Moreover, $\left\|\bbE_k(\Lambda)\,\bU\right\|^2_{\indHx} = 0$ for all interval $\Lambda \subset \bbR \setminus \bigcup_{\scZ\in\calZ}  \Lambda_\scZ(k)$, which does not contain $0$ or $\pm\Omegam.$
\label{p.specmesAk}
\end{proposition}

\begin{remark}\label{rem.wL21D}
Note the symbol $s$ in the index ``$\,\indHx,s$'' in {\rm (\ref{eq.specmesAk-cont})}. It indicates that $\langle \cdot\,,\cdot\rangle_{\indHx,s}$ is a slight adaptation of the inner product $( \cdot\,,\cdot )_{\indHx},$ which is necessary because $\wlkj \notin \Hx$ since these are oscillating (bounded) functions at infinity (that is why they are called \emph{generalized}). A simple way to overcome this difficulty is to introduce \emph{weighted} $L^2$ spaces
\begin{equation*}
L^2_s(\calO) := \{ u \in L^2_{\rm loc}(\calO) \mid (1 + x^2)^{s/2}\,u \in L^2(\calO)\} \quad \mbox{for }s \in \bbR \mbox{ and } \calO=\bbR \mbox{ or } \calO=\bbR_+.
\end{equation*}
We can then define $\Hxs := L^2_s(\bbR) \times L^2_s(\bbR)^2 \times L^2_s(\bbR_+) \times L^2_s(\bbR_+)^2.$ It is readily seen that for positive $s,$ the spaces $\Hxps$ and $\Hxms$ are dual to each other if $\Hx$ is identified with its own dual space, which yields the continuous embeddings $\Hxps \subset \Hx \subset \Hxms.$ The notation $\langle \cdot\,,\cdot\rangle_{\indHx,s}$ represents the duality product between them, which extends the inner product of $\Hx$ in the sense that
\begin{equation*}
\langle \bU \,, \bU' \rangle_{\indHx,s} = (\bU \,, \bU')_{\indHx} \quad\mbox{if }U \in \Hx \mbox{ and }U' \in \Hxps.
\end{equation*}
The above proposition holds true as soon as we choose $s$ so that the $\wlkj$'s belong to $\Hxms,$ that is if $s > 1/2.$
\end{remark}

\begin{proof}
Let $\scZ \in \calZ\setminus\{\EE\},$ $\Lambda := [a,b] \subset \Lambda_\scZ(k)$ with $a<b$ and $\bU \in \Dx$. We show in the second part of the proof that $\bbE_k(\{a\})\,\bU = \bbE_k(\{b\})\,\bU = 0,$ which implies that $\bbE_k((a,b))\,\bU = \bbE_k([a,b])\,\bU.$ Hence Stone's formula $(\ref{eq.stone-ab})$ together with the integral representation (\ref{eq.resint2}) show that the quantity $\| \bbE_k(\Lambda)\,\bU \|_\indHx^2$ is given by the following limit, where $\zeta := \lambda+\rmi\eta:$
\begin{equation*}
\big\| \bbE_k(\Lambda)\,\bU \big\|_\indHx^2 = \lim_{\eta \searrow 0}\ \frac{1}{\pi}\int_{a}^{b}\Imag\left(  
\big(\bbT_{k,\zeta}\,\bU,\bU\big)_{\indHx} + 
\int_{\bbR^2}\zeta \, \gzk(x,x') \ \bbS_{k,\zeta}\,\bU(x) \ \overline{\bbS_{k,\bar\zeta}\,\bU(x')} \ \rmd x \, \rmd x'
\right) \rmd\lambda.
\end{equation*}
The first step is to permute the limit and the integrals in this formula thanks to the Lebesgue's dominated convergence theorem. According to the foregoing, as $\lambda \neq 0$ and $\lambda \neq \pm\Omegam,$ the integrand is a continuous function of $\eta$ for all $\lambda \in [a,b],$ $x$ and $x'$ in the compact support of $\bU$ (recall that $\bU \in \Dx).$  Moreover, the integrand is dominated by a constant (provided $\eta$ remains bounded). Therefore the permutation is justified: in the above formula, we can simply replace $\zeta$ and $\bar\zeta$ by $\lambda.$ 

Formula (\ref{eq.adjoints}) tells us that $\bbT_{k,\lambda}$ is self-adjoint, hence $\Imag (\bbT_{k,\lambda}\,\bU,\bU)_{\indHx} = 0.$ Besides, notice that 

\begin{equation*}
\Imag\left( \int_{\bbR^2}\lambda \, \glk(x,x') \ u(x) \ \overline{u(x')} \ \rmd x \, \rmd x' \right)
= \int_{\bbR^2} \Imag\Big(\lambda \, \glk(x,x')\Big) \ u(x) \ \overline{u(x')} \ \rmd x \, \rmd x',
\end{equation*}
where $u := \bbS_{k,\lambda}\,\bU$ (this is easily deduced from the symmetry of $\glk,$ \textit{i.e.}, $\glk(x,x') = \glk(x',x),$ see (\ref{eq.defgreen})). This allows us to use Proposition \ref{prop.vecgen1} so that
\begin{equation*}
\big\| \bbE_k(\Lambda)\,\bU \big\|_{\indHx}^2 = 
\int_{a}^{b}\int_{\bbR^2}  \sum_{j\in J_{\scZ}} \ \overline{w_{k,\lambda,j}(x)}\ w_{k,\lambda,j}(x') \ \bbS_{k,\lambda}\,\bU(x) \ \overline{\bbS_{k,\lambda}\,\bU(x')} \ \rmd x \, \rmd x'\, \rmd\lambda.
\end{equation*}
By Fubini's theorem, this expression becomes
\begin{equation*}
\big\| \bbE_k(\Lambda)\,\bU \big\|_{\indHx}^2 = 
\int_{a}^{b}\sum_{j\in J_{\scZ}}\left| \int_{\bbR} \bbS_{k,\lambda}\,\bU(x) \ \overline{w_{k,\lambda,j}(x)}\ \rmd x \right|^2 \rmd\lambda,
\end{equation*}
which is nothing but (\ref{eq.specmesAk-cont}), since an integration by parts shows that 
\begin{equation*}
\int_{\bbR} \bbS_{k,\lambda}\,\bU(x) \ \overline{w_{k,\lambda,j}(x)}\ \rmd x = \langle \bU\,,\wlkj \rangle_{\indHx,s}.
\end{equation*}
By virtue of the $\sigma$-additivity of the spectral measure (see Definition \ref{def.spec}), formula (\ref{eq.specmesAk-cont}) holds true for any interval $\Lambda := (a,b) \subset \Lambda_\scZ(k)$ even if the resolvent is singular. Indeed, this singular behavior occurs only if $|k|=k_{\rm c}$ at $\lambda=\pm \lambda_{\rm c}$ where Proposition \ref{prop.vecgen2} tells us that in this case, $g_{k,\zeta}(x,x^{`})=O(| \zeta\mp \lambda_{\rm c}|^{-1/2})$, which is an integrable singularity.

Suppose now that for $a<b$, $\Lambda = [a,b]$ is located outside $\bigcup_{\scZ\in\calZ} \Lambda_\scZ(k)$ and does not contain $0$ or $\pm\Omegam.$ In this case, the one-sided limit $\glk$ of the Green function is real-valued for all $\lambda \in \Lambda$ (since $\thetalk^\pm \in \bbR^+,$ see (\ref{eq.expr-thetam})-(\ref{eq.expr-thetap})) and the same steps as above yield $\left\|\bbE_k(\Lambda)\,\bU\right\|^2_{\indHx} = 0.$

Consider finally singletons $\Lambda = \{a\}$ for $a \in \bbR\setminus\{-\Omegam,0,+\Omegam\}.$ Stone's formula $(\ref{eq.stone-a})$ together with (\ref{eq.resint2}) show that for all $\bU \in \Dx,$ the quantity $\| \bbE_k(\{a\})\,\bU \|_\indHx^2$ is given by the following limit, where $\zeta := a+\rmi\eta:$
\begin{equation*}
\| \bbE_k(\{a\})\,\bU \|_\indHx^2 = \lim_{\eta \searrow 0} \ \eta \ \Imag\left( \big(\bbT_{k,\zeta}\,\bU,\bU\big)_{\indHx} + 
\int_{\bbR^2}\zeta \, \gzk(x,x') \ \bbS_{k,\zeta}\,\bU(x) \ \overline{\bbS_{k,\bar\zeta}\,\bU(x')} \ \rmd x \, \rmd x'
\right).
\end{equation*}
This shows that $\bbE_k(\{a\})$ can be nonzero only if $a$ is a singularity of $\bbT_{k,\zeta},$ $\bbS_{k,\zeta}$ or $\gzk.$ From (\ref{eq.opS}) and (\ref{eq.opT}), the singularities of $\bbT_{k,\zeta}$ and $\bbS_{k,\zeta}$ are $\zeta = 0$ and $\zeta = \pm\Omegam,$ but these points are excluded from our study since Proposition \ref{p.vp-de-Ak}. Hence, we are only interested in the singularities of $\gzk,$ that is, the zeros $\pm \lambda_\scE(k)$ of $\calW_{k,\zeta}$ defined in Lemma \ref{lem.sing}. Suppose then that $|k| > k_{\rm c}$ and $a =\pm \lambda_\scE(k).$ As above, using the Lebesgue's dominated convergence theorem and Proposition \ref{prop.vecgen2}, we obtain
\begin{equation*}
\| \bbE_k(\{a\})\,\bU \|_\indHx^2 
= \int_{\bbR^2} \overline{w_{k,a,0}(x)} \, w_{k,a,0}(x')\ \bbS_{k,a}\,\bU(x) \ \overline{\bbS_{k,a}\,\bU(x')} \ \rmd x \, \rmd x'
= \left| \int_{\bbR} \bbS_{k,a}\,\bU(x) \ \overline{w_{k,a,0}(x)}\ \rmd x \right|^2.
\end{equation*}
where the second inequality follows from Fubini's theorem. Integrating by parts, formula (\ref{eq.specmesAk-cont}) follows. On the other hand, if $|k| = k_{\rm c}$ and $|a| = \lambda_{\rm c},$ the Green function is singular near $a,$ but Proposition \ref{prop.vecgen2} tells us that $g_{k,a+\rmi\eta}(x,x') = O(\eta^{-1/2}),$ which shows that $\| \bbE_k(\{a\})\,\bU \|_\indHx^2 = 0.$
\end{proof} 

%===============================================================================================================================

\subsubsection{Generalized Fourier transform for the reduced Hamiltonian}
The aim of this subsection is to deduce from the knowledge of the spectral measure $\bbE_k(\cdot)$ a \emph{generalized Fourier transform} $\bbF_k$ for $\bbA_k,$ that is, an operator which provides us a \emph{diagonal} form of the reduced Hamiltonian  $\bbA_k$ as
$$\bbA_k = \bbF_k^* \,\lambda\,\bbF_k.$$
This transformation $\bbF_k$ maps $\Hx$ to a \emph{spectral space} $\hatH_k$ which contains fields that depend on the spectral variable $\lambda.$ In the above diagonal form of $\bbA_k,$ ``$\lambda$'' denotes the operator of multiplication by $\lambda$ in $\hatH_k.$ In short, $\bbF_k$ transforms the action of $\bbA_k$ in the physical space $\Hx$ into a linear spectral amplification in the spectral space $\hatH_k.$ We shall see that $\bbF_k$ is a partial isometry which becomes unitary if we restrict it to the orthogonal complement of the eigenspaces associated with the three eigenvalues 0 and $\pm\Omegam,$ that is, the space $\Hxdiv$ defined in Proposition \ref{p.vp-de-Ak}. 

The definition of $\bbF_k$ comes from formulas (\ref{eq.specmesAk-cont}) and (\ref{eq.specmesAk-ponct}): for a fixed $k\in \bbR$ and all $\bU\in \Hxs$, we denote
\begin{equation}\label{eq.vecgen1d}
\forall \scZ\in \calZ,  \forall \lambda \in \Lambda_{\scZ}(k) \mbox{ and } \forall \, j\in J_{\scZ},
\quad  \bbF_k\bU(\lambda,j) := \langle \bU\,,\wlkj\rangle_{\indHx,s} \,,
\end{equation}
which represents the ``decomposition'' of $\bU$ on the family of \emph{generalized eigenfunctions} $(\bw_{k,\lambda,j})$ of $\bbA_k$ defined in (\ref{eq:defwijk}). We show below that the codomain of $\bbF_k$ is given by 
\begin{equation}\label{eq.specspaceHk}
\hatH_{k}= \left\lbrace\begin{array}{ll} L^{2}(\Lambda_{\DD}(k))^{2} \oplus   L^{2}(\Lambda_{\DE}(k)) \oplus L^{2}(\Lambda_{\DI}(k))^{2} \oplus L^{2}(\Lambda_{\EI}(k))  & \mbox{ if } |k|\leq k_{\rm c} \\ [12pt]
 L^{2}(\Lambda_{\DD}(k))^{2} \oplus   L^{2}(\Lambda_{\DE}(k)) \oplus L^{2}(\Lambda_{\EI}(k)) \oplus  l^{2}(\Lambda_{\EE} (k))& \mbox{ if } |k|>k_{\rm c} .
 \end{array}
 \right.
\end{equation}
We point out that  the space $l^{2}(\Lambda_{\EE} (k))$ is here isomorphic to $\bbC^2$ since  $\Lambda_{\EE}(k)=\{-\lambda_\scE(k), \lambda_\scE(k)\}$. We denote by $ \bhatU(\lambda,j)$ the fields of $\hatH_k$, where it is understood that $\lambda\in \Lambda_\scZ(k)$ and $j\in J_{\scZ}$ for $\scZ\in \{ \DD,\DE,\DI,\EI\}$ if $|k|\leq |k_{\rm c}|$ and for  $\scZ\in \{ \DD,\DE,\EI,\EE\}$ if $|k|> |k_{\rm c}|$.
The Hilbert space $\hatH_k$ is endowed with the norm $\left\| \cdot \right\|_{\hatH_k}$ defined by
\begin{equation}
\| \bhatU \|^2_{\hatH_k}   =  
\displaystyle \left\{ \begin{array}{ll} \displaystyle
\sum_{\scZ \in \{ \DD,\DE,\DI,\EI\}}  \int_{\Lambda_\scZ(k)} \sum \limits_{j\in J_{\scZ}} | \bhatU(\cdot,j) |^2 \, \rmd \lambda & \mbox{if } |k|\leq k_{\rm c}, \\[10pt]
\displaystyle \sum_{\scZ \in \{ \DD,\DE,\EI\}}   \int_{\Lambda_\scZ(k)}\sum \limits_{j\in J_{\scZ}} | \bhatU(\cdot,j) |^2\,  \rmd \lambda +\sum_{\lambda \in \Lambda_{\EE}(k)}|\bhatU(\lambda,0)|^2 & \mbox{if } |k| > k_{\rm c}.
\end{array}\right. \label{eq.defnormhatHk}
\end{equation}

The following theorem expresses the diagonalization of the reduced Hamiltonian. Its proof is classical (see, e.g., \cite{Haz-07}) and consists of two steps.  
We first deduce from Theorem \ref{th.spec} that $\bbF_k$ is an isometry from $\Hxdiv$ to $\hatH_k$ which diagonalizes $\bbA_k.$ We then prove that $\bbF_k$ is surjective.

\begin{theorem}\label{th.diagAk}
For $k\in \bbR^{*}$, let $\bbP_k$ denote the orthogonal projection on the subspace $\Hxdiv$ of $\Hx,$ \textit{i.e.}, $\bbP_k = \chi(\bbA_k)$ where $\chi$ is the indicator function of $\bbR\setminus \{-\Omegam,0,\Omegam \}$. 
The operator $\bbF_k$ defined in {\rm (\ref{eq.vecgen1d})} extends by density to a partial isometry from $\Hx$ to $\hatH_k$
whose restriction to the range of $\bbP_k$ (that is, $\Hxdiv$) is unitary. Moreover, $\bbF_k$ diagonalizes the reduced Hamiltonian $\bbA_k$ in the sense that for any measurable function $f:\bbR\to \bbC$, we have
\begin{equation}\label{eq.projdiagAk1}
\bbP_k \, f(\bbA_k)=f(\bbA_k)\, \bbP_k =\bbF_k^{*} \,f(\lambda)\,\bbF_k  \ \mbox{ on } \rmD(f(\bbA_k)),
\end{equation} 
where $f(\lambda)$ stands for the operator of multiplication by the function $f$ in the spectral space $\hatH_k$.
\end{theorem}

\begin{proof}
First, notice that the orthogonal projection $\bbP_k$ onto $\Hxdiv$ is indeed the spectral projection $\bbE_k\big(\bbR\setminus \{-\Omegam, 0,\Omegam\}\big)=\chi(\bbA_k)$ (by (\ref{eq.indicator}) and Proposition \ref{prop.eigvalAk}).

From Proposition \ref{p.specmesAk} and the definition (\ref{eq.vecgen1d}) of $\bbF_k,$ one can rewrite the spectral measure of $\bbA_k$ for any interval $\Lambda \subset \bbR\setminus \{-\Omegam, 0,\Omegam\}$ and for all $\bU \in \Dx$ as
\begin{equation*}
\left\|\bbE_k(\Lambda)\,\bU\right\|^2_{\indHx}   =  
\displaystyle \left\{ \begin{array}{ll} \displaystyle
\sum_{\scZ \in \{ \DD,\DE,\DI,\EI\}}  \int_{\Lambda \cap \Lambda_\scZ(k)}  \sum \limits_{j\in J_{\scZ}}|\bbF_k\bU(\lambda,j) |^2 \,\rmd \lambda & \mbox{if } 0<|k|\leq k_{\rm c}, \\[5pt]
\displaystyle  \sum_{\scZ \in \{ \DD,\DE,\EI\}}   \int_{\Lambda \cap \Lambda_\scZ(k)}  \sum \limits_{j\in J_{\scZ}} |\bbF_k\bU(\lambda,j) |^2 \,\rmd \lambda+ \sum_{\Lambda \cap \Lambda_{k}(\EE)} |\bbF_k\bU(\lambda,0) |^2& \mbox{if } |k| > k_{\rm c}.
\end{array}\right.
\end{equation*}
In the particular case $\Lambda=\bbR\setminus \{-\Omegam, 0,\Omegam\}$, we have $\Lambda \cap \Lambda_\scZ(k)=\Lambda_\scZ(k)$ for all $\scZ\in \calZ$,  thus using the definition (\ref{eq.defnormhatHk}) of the norm $\| \cdot \|_{\hatH_k} $ leads to the following identity
$$
\|\bbP_k \bU \|_{\Hx}^2=\| \bbF_k\bU \|_{\hatH_k}^2, \quad  \forall \bU \in \Dx.
$$
Hence, as $\Dx$ is dense in $\Hx$, $\bbF_k$ extends to a bounded operator on $\Hx$ and the latter formula
holds for all $\bU\in \Hx$. Thus, $\bbF_k$ is a partial isometry which satisfies $\bbF_k^{*}\bbF_k=\bbP_k$ and its restriction to the range of $\bbP_k$ is an isometry. 

In the sequel, as the above expression of the spectral measure depends on $k$, we only detail the case $0<|k|\leq k_{\rm c}$. The case $|k|>k_{\rm c}$ can be dealt with in the same way. Using the polarization identity, the above expression of $\|\bbE_k(\Lambda)\,\bU\|^2_{\indHx}$ yields that of $(\bbE_k(\Lambda)\,\bU,\bV)_{\indHx}$ and the spectral theorem \ref{th.spec} then shows that
\begin{equation}\label{eq.thspecAk}
(f(\bbA_k)\, \bbP_k  \, \bU,\bV)_{\Hx}=\sum_{\scZ \in \{ \DD,\DE,\DI,\EI\}}   \int_{\Lambda_\scZ(k)} f(\lambda)  \sum \limits_{j\in J_{\scZ}}  \, \bbF_k \bU(\lambda,j) \, \overline{\bbF_k \bV(\lambda,j) }\,  \rmd \lambda ,
\end{equation}
which holds for all $\bU \in \rmD(f(\bbA_k))$ (note that $f(\bbA_k)$ and $\bbP_k$ commute, thus $ \bbP_k  \, \bU\in  \rmD(f(\bbA_k))$ for $\bU \in \rmD(f(\bbA_k))$) and $\bV \in \Hx$. Using the definition of the inner product in $\hatH_k$ (see (\ref{eq.defnormhatHk})), this latter formula can be rewritten as 
$$(f(\bbA_k)\, \bbP_k \bU,\bV)_{\Hx}=(f(\lambda) \bbF_k \bU,  \bbF_k \bV )_{\hatH_k},$$ 
which yields (\ref{eq.projdiagAk1}).

Let us prove now that the isometry $\bbF_k: \Hxdiv \to \hatH_k$ is unitary, \textit{i.e.}, that  $\bbF_k$ is surjective or equivalently that $\bbF_k^{*}$ is injective.  Let $\bhatU\in \hatH_k$ such that $\bbF_k^{*} \bhatU=0$, then
\begin{equation}\label{eq.scalarzero}
(\bhatU, \bbF_k \bV)_{\hatH_k}=0 , \quad \forall \bV \in \Hx.
\end{equation}
We now choose a spectral zone  $\scZ\in \{\DD,\DE,\DI,\EI\}$.
For any interval $\Lambda \subset \Lambda_\scZ(k)$,  one denotes $\boldsymbol{1}_{\Lambda} \in \mathcal{L}(\hatH_k)$, the orthogonal projection in $\hatH_k$  corresponding to the  multiplication by the indicator function of $\Lambda$. We shall show at the end of the proof the commutation property
\begin{equation}\label{eq.comut}
\bbF_k \bbE_k(\Lambda)=\boldsymbol{1}_{\Lambda}\bbF_k \ \mbox{ in } \mathcal{L}(\Hx, \hatH_k).
\end{equation}
Using this relation and (\ref{eq.scalarzero}) for $\bbE_k(\Lambda)\bV$ instead of $\bV,$ we get
$$
(\bhatU, \bbF_k \bbE_k(\Lambda)\bV )_{\hatH_k}= (\bhatU,  \boldsymbol{1}_{\Lambda}\bbF_k \bV )_{\hatH_k}= \int_{\Lambda }  \sum \limits_{j\in J_{\scZ}} \bhatU(\lambda,j) \,  \overline{\bbF_k \bV(\lambda,j) } \, \rmd \lambda=0,
$$
where we have used the definition of inner product in $\hatH_k$ (cf (\ref{eq.defnormhatHk})).
Hence, as the last formula holds for any interval $ \Lambda \subset \Lambda_\scZ(k)$, we get
$$
 \sum_{j \in J_{\scZ} } \bhatU(\lambda,j) \, \overline{\bbF_k \bV(\lambda,j) }= \sum_{j \in J_{\scZ}}  \bhatU(\lambda,j)  \langle  \bw_{k,\lambda,j},\bV \rangle_{\indHx,s} =0, \ \forall \bV\in    \Hxps  \mbox{ for a.e. } \lambda \in \Lambda_\scZ(k)
$$
\mbox{and thus} $\sum_{j \in J_{\scZ}}  \bhatU(\lambda,j)  \wlkj=0$ in $\Hxms$ for a.e. $\lambda  \in \Lambda_\scZ(k).$ The family $(\psi_{k,\lambda,j})_{j\in J_{\scZ}}$ is clearly linearly independent (see (\ref{eq.psim}) and (\ref{eq.psip})), so is also the family 
$(w_{k,\lambda,j})_{j\in J_{\scZ}}$ (see (\ref{defwkl})). Then it follows from (\ref{eq:defwijk}) that the  $(\wlkj)_{j\in J_{\scZ}}$ are linearly independent too.
Therefore 
$$\bhatU(\lambda,j) =0 \quad\mbox{for a.e. } \lambda  \in \Lambda_\scZ (k).$$ 
As it holds for any $\scZ\in \{\DD,\DE,\DI,\EI\}$, $\bhatU=0$. Hence, $\bbF^{*}$ is injective and $\bbF$ is surjective.

It remains to prove (\ref{eq.comut}). Using (\ref{eq.projdiagAk1}) with $f=\boldsymbol{1}_{\Lambda}$ leads to $\bbE_k(\Lambda)=\bbF_k^{*}\boldsymbol{1}_{\Lambda} \bbF_k$ and therefore to $\bbF_k\bbE_k(\Lambda)=\bbQ_k \boldsymbol{1}_{\Lambda}\bbF_k$ where $\bbQ_k=\bbF_k \, \bbF_k^{*}$ is the orthogonal projection onto the (closed) range of $\bbF_k.$ To remove $\bbQ_k$, we point out that
$$
\| \bbF_k\bbE_k(\Lambda)\bV\|^2_{\hatH_k}=( \bbE_k(\Lambda) \bV,\bV)_{\Hx}^2= \|\boldsymbol{1}_{\Lambda} \bbF_k \bV\|^2_{\hatH_k}, \forall \bV \in \Hx,
$$
where the first equality is an immediate consequence of the relations $\bbF_k^{*}\bbF_k =\bbP_k$ and $\bbP_k\bbE_k(\Lambda)=\bbE_k(\Lambda)$, whereas the second one is readily deduced from (\ref{eq.thspecAk}) by taking  $\bU=\bV$ and $f=\boldsymbol{1}_{\Lambda}=\boldsymbol{1}_{\Lambda}^2$.
\end{proof}

In the following proposition, we give an explicit expression of the adjoint $\bbF_k^{*}$ of the generalized Fourier transform $\bbF_k$ which is  a ``recomposition operator'' in the sense that its ``recomposes'' a function $\bU\in \Hx$ from its spectral components $\bhatU(\lambda,j)\in \hatH_k$ which appear as ``coordinates'' on the spectral basis $(\wlkj)$. 
As $\bbF^{*}$ is bounded in $\hatH_k$, if suffices to know it on a dense subspace of $\hatH_k$. We first introduce  $\hatH_{k,{\rm c}}$ the subspace of $\hatH_k$ made of compactly supported functions. Then we consider the subspace of functions whose support ``avoids'' values of $\lambda$ , namely
$$
\begin{array}{ll}
\hatH_{k,{\rm d}} :=\big \{ \bhatU \in \hatH_{k,{\rm c}} \; | \; \mbox{supp } \bhatU \cap \{ -\Omegam, 0, \Omegam \} \big \}& \mbox{if } |k| \neq k_{\rm c} \,, \\[12pt]
\hatH_{k,{\rm d}} := \big \{ \bhatU \in \hatH_{k,{\rm c}} \; | \; \mbox{supp } \bhatU \cap \big( \{ -\Omegam, 0, \Omegam \} \cup \{ -\lambda_{\rm c}, \lambda_{\rm c} \} \big) \big \}& \mbox{if } |k| = k_{\rm c}\, .
\end{array}
$$
Note that $\hatH_{k,{\rm d}}$ is clearly dense in $\hatH_k.$

\begin{proposition}\label{prop.adjointtransformFk}
For all $\bhatU \in \hatH_{k,{\rm d}}$, we have
\begin{equation} 
\bbF_k^{*}\,\bhatU  =  
\displaystyle \left\{ \begin{array}{ll} \displaystyle
\sum_{\scZ \in \{ \DD,\DE,\DI,\EI\}}  \sum \limits_{j\in J_{\scZ}}  \int_{ \Lambda_\scZ(k)}  \bhatU(\lambda,j) \, \wlkj \,\rmd \lambda & \mbox{if } 0<|k|\leq k_{\rm c}, \\[15pt]
\displaystyle  \sum_{\scZ \in \{ \DD,\DE,\EI\}}  \sum \limits_{j\in J_{\scZ}}   \int_{ \Lambda_\scZ(k)}  \bhatU(\lambda,j) \, \wlkj \,\rmd \lambda + \sum_{\lambda \in  \Lambda_\EE(k)} \bhatU(\lambda,0)\, \wlkz & \mbox{if } |k| > k_{\rm c}.
\end{array}\right. \label{eq.adjointtransformFk}
\end{equation}
where the integrals in the right-hand side of {\rm (\ref{eq.adjointtransformFk})} are vector-valued integrals (Bochner integrals {\rm \cite{Yos-74}}) with values 
in $\Hxms.$
\end{proposition}

\begin{remark} \label{rem.Fk} (i) The reason why we have to restrict ourselves to functions of $\hatH_{k,{\rm d}}$ in {\rm (\ref{eq.adjointtransformFk})} is that the $\Hxms$-norm of $ \wlkj $ remains uniformly bounded if $(k,\lambda)$ is restricted to vary in a compact set of $\bbR^2$ that does not intersect the points $\{ -\Omegam,0,\Omegam\}$, neither the points $\pm \lambda_{\rm c}$ when $|k|=k_{\rm c}$. On the other hand, these norms blow up 
as soon as $\lambda$ approaches any of these points. For $\pm \lambda_{\rm c}$ this results from the presence of the Wronskian $\calW_{k,\lambda}$ 
in the denominator of the expression {\rm (\ref{defwkl})} of $w_{k,\lambda,\pm 1}$. For $\pm \Omegam$, this is due to the term 
$\mu_\lambda^+$ (which vanishes for $\lambda = \pm \, \Omegam)$ in the same denominator. For $0$, this follows from the fact that
$\thetalk^+$ blows up when $\lambda$ tends to $0$ (cf. {\rm (\ref{eq.deftheta})} and {\rm (\ref{eq.drude})}).

(ii) Hence, when $\bhatU\in \hatH_{d,k}$, the integrals considered in {\rm (\ref{eq.adjointtransformFk})}, whose integrands are valued in $\Hxms$, are Bochner integrals {\rm \cite{Yos-74}} in $\Hxms$. However, as $\bbF^{*}$ is bounded from $\hatH_k$ to $\Hx,$ the values of these integrals belongs to $\Hx$. By virtue of the density of $\hatH_{k,{\rm d}}$ in $\hatH_k,$ the expression of $\bbF^{*}\bhatU$ for any $\bhatU\in \hatH_k$ follows by approximating $\bhatU$ by its restrictions to an increasing sequence of compact subsets of $\bbR$ as in the definition of $\hatH_{k,{\rm d}}.$ Of course, the limit we obtain belongs to $\Hx$ and does not depend on the sequence (note that this is similar to the limiting process used to express the usual Fourier transform of a $L^2$ function). This limit process will be implicitly understood in the sequel.
\end{remark}

\begin{proof}
We prove this result in the case $0<|k|\leq k_{\rm c}$, the case $|k|>k_{\rm c}$ can be dealt with in the same way.
Let $\bhatU\in \hatH_{k,{\rm d}}$. By definition of the adjoint, for all $\bV \in \Hxps,$ one has
$(\bbF_k^{*} \bhatU, \bV)_{\Hx} = (\bhatU, \bbF_k \bV)_{\hatH_k}$
and the expression of $\bbF_k \bV$ yields
$$
(\bbF_k^{*} \bhatU, \bV)_{\hatH_k} 
= \sum_{\scZ \in \{ \DD,\DE,\DI,\EI\}}  \sum \limits_{j\in J_{\scZ}}  \int_{ \Lambda_\scZ(k)} \bhatU(\lambda,j)  \,  \langle  \wlkj , \bV\rangle_{\indHx,s} \,\rmd \lambda.
$$
One can permute the duality product in $x$ and the Bochner integral in $\lambda$ to obtain
$$
(\bbF_k^{*} \bhatU, \bV)_{\hatH_k} = 
\Big\langle  \sum_{\scZ \in \{ \DD,\DE,\DI,\EI\}}  \sum \limits_{j\in J_{\scZ}}  
\int_{ \Lambda_\scZ(k)} \bhatU(\lambda,j) \,  \wlkj \,\rmd\lambda \,, \bV \Big\rangle_{\indHx,s}.
$$
As it holds for any $\bV \in \Hxps$, this yields (\ref{eq.adjointtransformFk}).
The permutation is here justified by the following arguments: for any $\scZ \in \{ \DD,\DE,\DI,\EI\}$ and $j\in J_{\scZ}$, $\bhatU(\cdot, j)$ is a $L^1(\Lambda_k(\scZ))$ compactly supported function in $\lambda$ and the generalized eigenfunctions  $\wlkj$ are uniformly bounded for $x\in \bbR$ and $\lambda$ on the compact support of $\bhatU(\cdot, j)$. Thus, the left-hand side of the duality product is a finite sum of  Bochner integrals, since the considered integrands (which are vector-valued in $\Hxms$) are integrable. Hence, the permutation follows from a standard property 
(Fubini's like) of Bochner integrals \cite{Yos-74}.   
\end{proof}

%XXXXXXXXXXXXXXXXXXXXXXXXXXXXXXXXXXXXXXXXXXXXXXXXXXXXXXXXXXXXXXXXXXXXXXXXXXXXXXXXXXXXXXXXXXXXXXXXXXXXXXXXXXXXXXXXXXXXXXXXXXXXXXX
%XXXXXXXXXXXXXXXXXXXXXXXXXXXXXXXXXXXXXXXXXXXXXXXXXXXXXXXXXXXXXXXXXXXXXXXXXXXXXXXXXXXXXXXXXXXXXXXXXXXXXXXXXXXXXXXXXXXXXXXXXXXXXXX
\section{Spectral theory of the full Hamiltonian and application to the evolution problem}
\label{s-Spec-th-A}
The hard part of the work is now done: for each fixed non zero $k$, we have obtained  a diagonal form of the reduced Hamiltonian $\bbA_k$. It remains to gather this collection of results for $k\in \bbR^{*}$, which yields a diagonal form of the full Hamiltonian $\bbA$. The proper tools to do so are the notions of direct integrals of Hilbert spaces and operators (see, e.g. \cite{Dix-69,Ree-78,Wan-06}) that we implicitly assume to be known by the reader (at least their definition and elementary properties). 

\subsection{Abstract diagonalization of  $\bbA$}
The first step consists in rewriting the link (\ref{eq.AtoAk}) between $\bbA$ and $\bbA_k$ in an abstract form using direct integrals. The partial Fourier transform in the $y$-direction $\calF$ led us to define $\bbA_k$ for each fixed $k$ as an operator in $\Hx$ (see (\ref{eq.defH1D})). Actually, the initial space $\Hxy$ introduced in (\ref{eq.defHxy}) is nothing but the tensor product of Hilbert spaces $\Hx\otimes L^2(\bbR_y)$ or equivalently the (constant fiber) integral
$$
\Hxy=\int_{\bbR}^{\oplus} \Hx  \,\rmd y.
$$ 
Hence the partial Fourier transform $\calF$ appears as a unitary operator from $\Hxy$ to
$$
\Hxk:=\Hx \otimes  L^2(\bbR_k)=\int_{\bbR}^{\oplus} \Hx  \,\rmd k.
$$ 
A vector $\bU_{\!\oplus}\in \Hxk$ is simply a $4$-uplet analogous to a $\bU\in \Hxy$ but depending on the pair of variables $(x,k)$ instead of $(x,y)$. For a.e. $k\in \bbR$, we denote  $\bU_{\!k}:=\bU_{\!\oplus}(\cdot,k)\in \Hx$ so that
$$
\| \bU_{\!\oplus}\|^2_{\Hxk}=\int_{\bbR} \| \bU_{\!k}\|_{\Hx}^2 \, \rmd k  .
$$
In this functional framework, we can gather the family of reduced Hamiltonian $\bbA_k$ for $k\in \bbR$ as a direct integral of operators $\bbA_{\oplus}$ defined by
$$
\bbA_{\oplus}:=\int_{\bbR}^{\oplus}\bbA_k \, \rmd k,
$$
which means that for a.e. $k\in \bbR$,
$$
\big(\bbA_{\oplus}\, \bU_{\!\oplus}\big)_k=\bbA_k \, \bU_{\!k}, \quad \forall \; \bU_{\!\oplus} \in \rmD(\bbA_{\oplus}),
$$
where
\begin{equation}\label{eq.domAoplus}
\rmD(\bbA_{\oplus}) 
:= \left\{ \bU_{\!\oplus}\in\Hxk \mid \bU_{\!k} \in \rmD(\bbA_k) \mbox { for a.e. } k\in \bbR \mbox{ and } \int_{\bbR} \|\bbA_k \bU_{\!k} \|_{\Hx}^2 \, \rmd k < \infty \right\}.
\end{equation}
Relation (\ref{eq.AtoAk}) can then be rewritten in the concise form $\calF \, \bbA= \bbA_{\oplus} \,\calF$,
or equivalently
\begin{equation}\label{eq.rel-A-Aoplus}
\bbA=\calF^{*}\bbA_{\oplus} \calF.
\end{equation}
General properties can now be applied to obtain a diagonal expression of $\bbA$ summarized in the following theorem and illustrated by the commutative diagram of Figure \ref{fig.diag2}.

\begin{theorem}\label{th.diagA}
Let $\bbP$ denote the orthogonal projection defined in $\Hxy$ by $\bbP:=\chi(\bbA)$ where $\chi$ is the indicator function of $\bbR\setminus\{-\Omegam,0,\Omegam\}$. Consider the direct integral $\bbF_{\oplus}$ of the family of all generalized Fourier transforms $\bbF_k$ for $k \in \bbR$ (see Theorem {\rm \ref{th.diagAk}}), that is,
\begin{equation} \label{eq.defspectralspaceH}
\bbF_{\oplus}:=\int_{\bbR}^{\oplus} \bbF_k \,\rmd k\ \text{ which maps }\Hxk \text{ to } 
\hatH :=\int_{\bbR}^{\oplus} \hatH_k \,\rmd k.
\end{equation}
Then, for any measurable function $f:\bbR \to \bbC$, we have
\begin{equation}\label{eq.diagA}
f(\bbA)\bbP=\bbP f(\bbA)=\bbF^{*}\,f(\lambda)\,\bbF \ \mbox{ in } \rmD(f(\bbA))
\end{equation}
where 
$$
\bbF: \Hxy \to \hatH \mbox{ is defined by } \bbF := \bbF_{\oplus} \, \calF \mbox{ and satisfies } \bbF\,\bbF^{*} =\mathrm{Id}_{\hatH}.
$$ 
Thus, the restriction of $\bbF$ on the range of $\bbP$ is a unitary operator.
\end{theorem}
\begin{figure}[!t]
\centering
\includegraphics[width=0.70\textwidth]{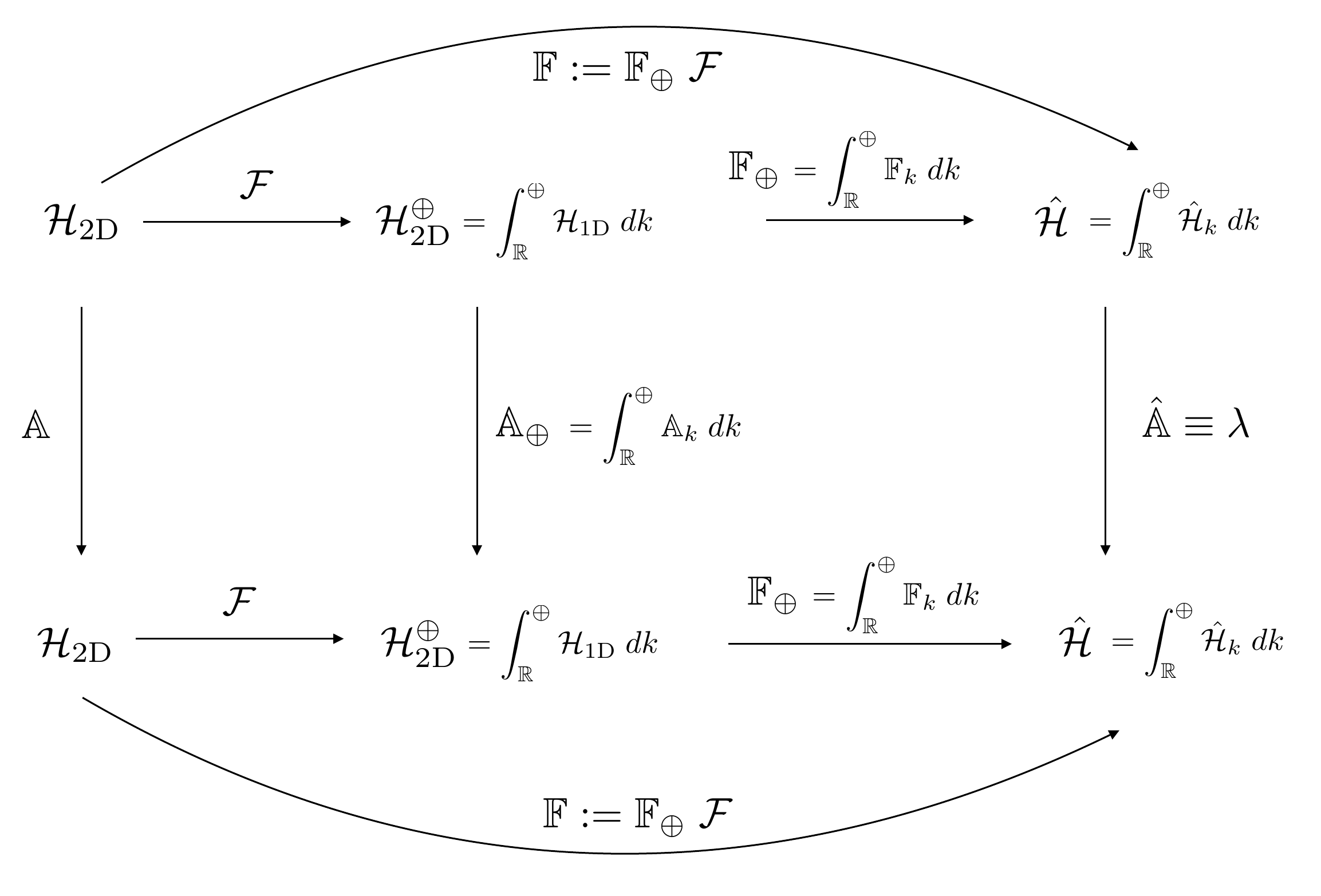}
\caption{Diagonalization diagram of $\bbA$.}
\label{fig.diag2}
\end{figure}

\begin{proof}
From the definition of $\bbP,$ the diagonalization formula (\ref{eq.diagA}) amounts to $f\chi(\bbA) = \chi f(\bbA) = \bbF^{*}f(\lambda)\bbF.$ To prove this, we start from formula (\ref{eq.rel-A-Aoplus}) which shows that $\bbA$ and $\bbA_{\oplus}$ are unitarily equivalent. So the same holds for $f(\bbA)$ and $f(\bbA_{\oplus})$ for any measurable function $f:\bbR \to \bbC$ (see \cite{Dun-63}). More precisely, using $f\chi$ instead of $f,$ we have
\begin{equation}\label{eq.calfuncAAplus}
f\chi(\bbA)=\calF^{*} f\chi(\bbA_{\oplus})\calF \, \mbox{ in } \rmD(f(\bbA))= \calF^{*}(\rmD(f(\bbA_{\oplus}))).
\end{equation}
It remains to diagonalize $f\chi(\bbA_{\oplus}).$ We first use the essential property (see \cite{Ree-78})
\begin{equation*}
f\chi(\bbA_{\oplus})=\int_{\bbR}^{\oplus}  f\chi(\bbA_k) \, \rmd k  \quad \mbox{in } \rmD(f(\bbA_{\oplus})),
\end{equation*}
(where the domain $\rmD \big(f(\bbA_{\oplus})\big)$ is defined as $\rmD(\bbA_{\oplus})$ in (\ref{eq.domAoplus}) by replacing $\bbA_{\oplus}$ and $\bbA_k$ by $f(\bbA_{\oplus})$ and $f(\bbA_k)$). Roughly speaking, this latter relation means that  the functional calculus ``commutes'' with direct integrals of operators. We deduce from this relation that for $\bU_{\!\oplus}\in \rmD(f(\bbA_{\oplus}))$ and $\bV_{\!\oplus}\in \Hxk$, one has
$$
(f\chi(\bbA_{\oplus}) \bU_{\!\oplus},  \bV_{\!\oplus})_{\Hxk} = \int_{\bbR} (f\chi(\bbA_{k}) \bU_{\!k},  \bV_{\!k})_{\Hx} \,\rmd k.
$$
Hence, using the family of diagonalization formulas (\ref{eq.projdiagAk1}), \textit{i.e.}, $f\chi(\bbA_k)= \bbP_k \, f(\bbA_k)  = \bbF^{*}_k\, f(\lambda) \,\bbF_k,$ we obtain
$$
(f\chi(\bbA_{\oplus}) \bU_{\!\oplus},  \bV_{\!\oplus})_{\Hxk}
= \int_{\bbR} (f(\lambda) \bbF_k \bU_{\!k}, \bbF_k \bV_{\!k})_{\hatH_k} \,\rmd k 
= (f(\lambda) \bbF_{\oplus} \bU_{\!\oplus},  \bbF_{\oplus} \bV_{\!\oplus})_{\hatH}
$$
which shows that
$$
f\chi(\bbA_{\oplus})=\bbF_{\oplus}^{*} f(\lambda)\bbF_{\oplus},
$$
where $\bbF_{\oplus}$ is defined in (\ref{eq.defspectralspaceH}). Note that this operator is bounded since $\operatorname{ess-sup}_{k \in \bbR} \| \bbF_k \| = 1$ (where  $\operatorname{ess-sup}$ denotes the essential sup) because $\|\bbF_k\|=1$ for all $k\in\bbR^{*}$. Combining the latter relation with (\ref{eq.calfuncAAplus}) yields (\ref{eq.diagA}).

In the particular case where $f(\lambda) \equiv 1$, relation (\ref{eq.diagA}) shows that $\chi(\bbA)=\bbF^{*}\, \bbF$, whereas similar arguments as above tell us that
$$
\bbF_{\oplus}\, \bbF_{\oplus}^{*} 
= \int_{\bbR}^{\oplus} \bbF_{k} \bbF_k^{*} \, \rmd k 
= \int_{\bbR}^{\oplus} \mathrm{Id}_{\hatH_k} \, \rmd k
= \mathrm{Id}_{\hatH},
$$
hence $\bbF\,\bbF^{*} = \bbF_{\oplus}\,\calF\,\calF^*\,\bbF_{\oplus}^* = \mathrm{Id}_{\hatH}$ since $\calF\,\calF^* = \mathrm{Id}_{\Hxk}.$ This completes the proof. 
\end{proof}

% ===============================================================================================================================

\subsection{Generalized Fourier transform for $\bbA$}
Theorem \ref{th.diagA} is the main result of the present paper. It is formulated here in an abstract form which will become clearer if we make more explicit the various objects involved in this theorem. This is the subject of this section.

\subsubsection{Characterization of the projection $\bbP$} 
By (\ref{eq.calfuncAAplus}), the orthogonal projection $\bbP := \chi(\bbA)$ can be equivalently written as 
$$
\bbP=\calF^{*} \, \chi(\bbA_{\oplus}) \, \calF 
\quad \mbox{where} \quad 
\chi(\bbA_{\oplus}) = \int_{\bbR}^{\oplus}   \chi(\bbA_{k}) \,\rmd k=\int_{\bbR}^{\oplus}  \bbP_k\, \rmd k,$$
where the $\bbP_k$'s are defined in Theorem \ref{th.diagAk}. This shows in particular that the range of $\bbP$ is given by
$$
\mathrm{Im} (\bbP) = \calF^{*} \left( \int_{\bbR}^{\oplus}  \mathrm{Im} (\bbP_k )\, \rmd k \right).
$$
where $\mathrm{Im} (\bbP_k )= \Hxdiv$ is defined in (\ref{eq.vuk}). As $\calF \div=\divk \calF$, we deduce that
\begin{equation}
\mathrm{Im} (\bbP)= \Hxydiv :=\{ (\bE, \bH, \bJ, \bK)\in \Hxy \mid \div \bH=0 \mbox{ in } \bbR^{2}_{\pm} \mbox{ and } \div \bK=0 \mbox{ in } \bbR^2_+\}.
\label{eq.Hxydiv}
\end{equation}

In the same way, $\ker(\bbP)=\ker(\bbA)\oplus \ker(\bbA-\Omegam)\oplus \ker(\bbA+\Omegam)$ is described by
$$
\ker (\bbA)=\calF^{*} \left( \int_{\bbR}^{\oplus} \ker (\bbA_k )\, \rmd k \right) \mbox{ and } \ker (\bbA\mp \Omegam)=\calF^{*} \left( \int_{\bbR}^{\oplus} \ker (\bbA_k \mp \Omegam )\, \rmd k \right),
$$
where $\ker(\bbA_k)$ and $\ker (\bbA_k \pm \Omegam )$ are characterized in Proposition \ref{prop.eigvalAk}. As $\calF \nabla= \nabla_k \calF$, we have
\begin{eqnarray}
\ker(\bbA) & = & \{ (0, \widetilde{\bPi}\, \nabla \phi, 0 , 0)^{\top} \ \mid \ \phi\in H_0^1(\bbR^2_-)\},  \label{eq.kernelA}\\
\ker(\bbA \mp \Omegam ) & = & \left\{ \left(0, \,\bPi \,\nabla \phi, \,0 , \pm \rmi\mu_0\Omegam \,\nabla \phi \right)^{\top} \ \mid \ \phi\in H_0^1(\bbR^2_+)\right\} \label{eq.eigenspace},
\end{eqnarray}
where $\widetilde{\bPi}$ is the extension by $0$ of a 2D vector field defined on $\bbR^2_-$ to the whole plane $\bbR^2$ and $H_0^1(\bbR^2_{\pm}) := \{ \phi \in H^1(\bbR^2_{\pm})\ \mid \ \phi(0,y) = 0 \mbox{ for a.e. } y\in \bbR \}.$

\subsubsection{Description of the spectral space $\hatH$}
Consider now the spectral space $\hatH$ defined in (\ref{eq.defspectralspaceH}) where  each fiber $\hatH_k$ is given in (\ref{eq.specspaceHk}).
The elements of $\hatH$ are then vector fields $\bhatU: k\to \bhatU_k \in \hatH_k$ such that
$$
\| \bhatU\|_{\hatH}^2:=\int_{\bbR} \| \bhatU_k\|^2_{\hatH_k}\, \rmd k<\infty.
$$
Each space $\hatH_k$ is composed of $L^2$-spaces defined on the various zones $\Lambda_{\scZ}(k),$ which are vertical sections of the spectral zones $\Lambda_{\scZ}$ represented in Figure \ref{fig.speczones} (and defined in (\ref{eq.defspeczones}) and (\ref{eq.defspeczonesZEE})). The above formula gathers the spaces associated with all sections to create a space of fields defined on the zones $\Lambda_{\scZ},$ $\scZ \in \calZ.$ Indeed, by Fubini's theorem, we see that $\hatH$ can be identified with the following direct sum:
\begin{equation}
\hatH=\bigoplus \limits_{\scZ \in \calZ} L^{2}(\Lambda_{\scZ})^{\operatorname{card}(J_{\scZ})}=L^2(\Lambda_{\DD})^2 \oplus L^2(\Lambda_{\DE})\oplus L^2(\Lambda_{\DI})^2 \oplus  L^2(\Lambda_{\EI})\oplus   L^2(\Lambda_{\EE}).
\label{eq.ident-hatH}
\end{equation}
As we did for the generalized eigenfunctions $\bw_{k,\lambda,j}$, we denote somewhat abusively by $\bhatU(k,\lambda,j)$ the fields of $\hatH$, where it is understood that $j\in J_{\scZ}$ while $(k,\lambda)\in \Lambda_{\scZ}$ for the various zones $\Lambda_{\scZ}$, $\scZ\in \calZ$. The norm in $\hatH$ can then be rewritten as
$$
\| \bhatU \|_{\hatH}^2=\sum_{\scZ \in \calZ\setminus\{\EE\}}\sum_{j\in J_\scZ} \int_{ \Lambda_{\scZ}} |\bhatU(k,\lambda,j) |^2 \,\rmd \lambda \,\rmd k+  \sum_{\pm }\int_{|k|>k_{\rm c}}| \bhatU(k,\pm \lambda_{E}(k),0) |^2  \,\rmd k.
$$

\subsubsection{The generalized Fourier transform $\bbF$ and its adjoint} 
We show here that, as for the reduced Hamiltonian, the generalized Fourier transform $\bbF$ appears as a ``decomposition'' operator on a family of generalized eigenfunctions of $\bbA,$ denoted by $(\bbW_{k,\lambda,j})$ and constructed from the generalized eigenfunctions $(\wlkj)$ of the reduced Hamiltonian $\bbA_k$ (see (\ref{eq:defwijk})) via the following formula:
\begin{equation}\label{eq.defvecgen2D}
\forall \scZ \in \calZ, \ \forall (k,\lambda) \in \Lambda_{\scZ}, \ \forall j\in J_{\scZ}, \ \forall (x,y) \in \bbR^2,
\quad \bbW_{k,\lambda,j}(x,y) := \wlkj(x) \, \frac{ \rme^{\rmi k y}}{\sqrt{2 \pi}} .
\end{equation}
Similarly the adjoint $\bbF^{*}$ is a ``recomposition'' operator in the sense that its ``recomposes'' a function $\bU\in \Hxy$ from its spectral components $\bhatU(k,\lambda,j) \in \hatH$ which appears as ``coordinates'' on the spectral basis $(\bbW_{k,\lambda,j})$ of $\bbA$. 

As $\bbF$ (respectively, $\bbF^{*}$) is bounded in $ \Hxy$ (respectively, $\hatH$), if suffices to define it on a dense subspace of $ \Hxy$ (respectively, $\hatH$). Consider first the case of the physical space $\Hxy.$ In the same way as the $1\rmD$-case (see Remark \ref{rem.wL21D}), we define 
$$
\Hxys := L^2_s(\bbR^2) \times L^2_s(\bbR^2)^2 \times L^2_s(\bbR^2_+) \times L^2_s(\bbR^2_+)^2 \quad\mbox{for } s \in \bbR,
$$
where $L^2_s(\calO) := \{ u \in L^2_{\rm loc}(\calO) \mid (1 + x^2)^{s/2}\,(1 + y^2)^{s/2}\,u \in L^2(\calO)\}$ for $\calO=\bbR^2$ or $\calO=\bbR^2_+.$ Note that $\Hxyms$ and $\Hxyps$ are dual spaces, the corresponding duality bracket being denoted $\langle \cdot, \cdot \rangle_{\indHxy,s}.$ It is clear that $\Hxys$ is dense in $\Hxy$ for all $s > 0,$ but we shall actually choose $s > 1/2:$ the key point is that each function $\bbW_{k,\lambda,j}$, being bounded, belongs to $\Hxyms$.

For the space $\hatH$, this is a little more tricky. We first define 
%the support (in the $(k,\lambda)$-plane) of $\bhatU \in \hatH$
%$$
%\mbox{supp } \bhatU := \bigcup_{\scZ \in \calZ} \bigcup_{j \in  J_{\scZ}} \mbox{supp } \bhatU(\cdot, \cdot,j)
%$$
%and 
$\hatH_{\rm c}$ the subspace of $\hatH$ made of compactly supported functions and we introduce the lines 
$$
D_0 := \bbR \times \{0\} \quad\mbox{and}\quad D_{\pm \rm m} := \bbR \times \{\pm \Omegam\},
$$
as well as the finite set $P_{\rm c} := \{(k_{\rm c}, \lambda_{\rm c}), (k_{\rm c}, -\lambda_{\rm c}), (-k_{\rm c}, \lambda_{\rm c}), (-k_{\rm c},- \lambda_{\rm c})\}$. Then we define the space
$$
\hatH_{\rm d} := \{ \bhatU \in \hatH_{\rm c} \, \mid \ \mbox{supp } \bhatU \cap \big( D_0 \cup D_{+\rm m} \cup D_{-\rm m} \cup P_{\rm c} \big) = \emptyset \}.
$$
Since $D_0$, $D_{\pm \rm m}$ and $P_{\rm c}$ have Lebesgue measure 0 in $\bbR^2$, $\hatH_{\rm d}$ is clearly dense in $\hatH.$

\begin{proposition}
Let $s > 1/2.$ The generalized Fourier transform $\bbF \bU$ of all $\bU \in \Hxyps$ is explicitly given in each zone $\Lambda_\scZ,$ $\scZ\in \calZ,$ by
\begin{equation}\label{eq.vecgen2d}
\bbF\bU(k,\lambda,j) =\langle  \bU, \bbW_{k,\lambda,j}  \rangle_{\indHxy,s} \quad\mbox{for all } (k,\lambda) \in \Lambda_\scZ \mbox{ and } j\in J_{\scZ},
\end{equation}
where the $\bbW_{k,\lambda,j}$'s are defined in {\rm (\ref{eq.defvecgen2D})}. Furthermore, for all $\bhatU \in \hatH_{\rm d},$ we have
\begin{equation}\label{eq.adjointtransform}
\bbF^{*}\bhatU = \sum_{\scZ \in \calZ \setminus \{\EE \} } \sum_{j\in J_{\scZ}} \int_{\Lambda_{\scZ} } \bhatU(k,\lambda,j)\, \bbW_{k,\lambda, j} \,\rmd \lambda\, \rmd k 
+ \sum_{\pm} \int_{|k|>k_{\rm c} } \bhatU(k,\pm \lambda_\scE(k) ,0)\, \bbW_{ k,\pm \lambda_\scE(k) ,0} \, \rmd k,
\end{equation}
where the integrals are Bochner integrals with values in $\Hxyms$.
\label{p.FandAdjoint}
\end{proposition}

\begin{remark}\label{propbbW}
The content of Remark {\rm \ref{rem.Fk}} could be transposed here with obvious changes. In particular, for general $\bU \in \Hxy$ or $\bhatU \in \hatH,$ the expressions of $\bbF\bU$ or $\bbF^{*}\bhatU$ are deduced from the above ones by a limit process on the domain of integration (exactly as for the usual Fourier transform of a square integrable function). In the sequel, this process will be implicitly understood when applying formulas {\rm (\ref{eq.vecgen2d})} and {\rm (\ref{eq.adjointtransform})} for general $\bU$ and $\bhatU.$
\end{remark}

\begin{proof}
Let $\bU\in \Hxyps$ with $s>1/2.$ By definition, $\bbF := \bbF_{\oplus}\,\calF$ where $\bbF_{\oplus}$ is defined in (\ref{eq.defspectralspaceH}). Hence, in each zone $\Lambda_\scZ,$ $\scZ \in \calZ,$ we have
\begin{equation*}
\bbF\bU(k,\lambda,j) = \bbF_k (\calF \bU(\cdot,k))(\lambda,j), 
\quad\mbox{for a.e. } (k,\lambda)\in \Lambda_\scZ \mbox{ and } j\in J_{\scZ}.
\end{equation*}
Note that $\Hxyps = \Hxps \otimes L^2_s(\bbR_y)$ and $\calF(L^2_s(\bbR_y))=H^{s}(\bbR_k)$ where $H^{s}(\bbR_k)$  stands for the Sobolev space of index $s$, thus $\calF \bU \in \Hxps \otimes  H^{s}(\bbR_k)$. As $s>1/2$, $H^{s}(\bbR_k)$ is included in  $C_0(\bbR_k),$ the space of continuous function on $\bbR_k$, so $\calF \bU \in C_0(\bbR_k, \Hxps )$ and one can define $\calF \bU (\cdot, k)\in \Hxps$ for all real $k$. Hence, using the respective definitions (\ref{eq.deffour}) and (\ref{eq.vecgen1d}) of $\calF$ and $\bbF_k$ leads to
$$
\bbF\bU(k,\lambda,j) = \left\langle  \int_{\bbR} \bU(\cdot,y)\, \frac{ \rme^{-\rmi k y}}{\sqrt{2\pi}}\, \rmd y \, , \wlkj \right\rangle_{\indHx,s},
\quad\mbox{for a.e. } (k,\lambda)\in \Lambda_\scZ \mbox{ and } j\in J_{\scZ}, 
$$
which yields (\ref{eq.vecgen2d}) thanks to a Fubini's like theorem for Bochner integrals (which applies here since $\bU \in \Hxyps$ and $\bbW_{k,\lambda,j} \in \Hxyms$).

We now prove (\ref{eq.adjointtransform}). Recall that the adjoint of a direct integral of operators is the direct integral of their adjoints \cite{Wan-06}. Hence, 
$$
\bbF^{*} = \calF^{*}\bbF_{\oplus}^{*} \quad\mbox{where}\quad
\bbF_{\oplus}^{*} = \int_{\bbR}^{\oplus} \bbF_k^{*}  \, \rmd k \mbox{ maps } \hatH \mbox{ to } \Hxk.
$$
Let $\bhatU\in \hatH_{\rm d}.$ For a.e. $k\in \bbR$, $\bhatU_k := \bhatU(k,\cdot,\cdot)\in \hatH_k$ and its support in $\lambda$ is compact.
Moreover $\bbF_{\oplus}^{*} \bhatU(\cdot,k) \equiv \bbF_k^{*}\bhatU_k$  vanishes for $k$ large enough. So
$$
\bbF^{*}\bhatU(x,y) = \calF^{*}\bbF_{\oplus}^{*} \bhatU (x,y) 
= \int_{\bbR} \bbF_k^{*}\bhatU_k(x) \, \frac{\rme^{\rmi k y}}{\sqrt{2\pi}} \, \rmd k \quad\mbox{for a.e. } (x,y) \in \bbR^2.
$$
Note that the support of $\bhatU_k$ does not contain $0$ and $\pm \Omegam$, neither $\pm \lambda_{\rm c}$ when $|k|=k_{\rm c}$. 
Thus, by Proposition \ref{prop.adjointtransformFk}, $ \bbF_k^{*}\bhatU_k$ is given by formula (\ref{eq.adjointtransformFk}) with
 $\bhatU_k$ instead of  $\bhatU$. Then it suffices to apply Fubini's theorem again to obtain formula  (\ref{eq.adjointtransform}), using the fact that the function $(k, \lambda) \mapsto \bbW_{k,\lambda, j}$ is bounded in $\Hxyms$ when $(k,\lambda)$ varies in the support of $\bhatU$.
\end{proof}

% ============================================================================
\subsection{Spectrum of $\bbA$}
The preceding results actually show that the spectrum $\sigma(\bbA)$ of $\bbA$ is obtained by superposing the spectra $\sigma(\bbA_k)$ of $\bbA_k$ for all $k \in \bbR.$ More precisely, we have the following property.

\begin{corollary}
The spectrum of $\bbA$ is the whole real line: $\sigma(\bbA)=\bbR.$ The point spectrum $\sigma_{\rm p}(\bbA)$ is composed of eigenvalues of infinite multiplicity: $\sigma_{\rm p}(\bbA)=\{-\Omegam,0,\Omegam\}$ if $\Omegae\neq \Omegam$ and $\sigma_{\rm p}(\bbA)=\{-\Omegam, -\Omegam/\sqrt{2}, 0, \Omegam/\sqrt{2},\Omegam\}$ if $\Omegae = \Omegam.$
\label{c.spectrumA}
\end{corollary}

\begin{proof}
It is based on Remark \ref{rem.spec}. First consider an interval $\Lambda = (a,b) \subset  \bbR\setminus\{-\Omegam, 0,\Omegam \}$ with $a<b.$ The diagonalization formula (\ref{eq.diagA}) applied to the indicator function $f={\bf 1}_{\Lambda}$ shows that the spectral projection $\bbE(\Lambda)$ is given by $\bbE(\Lambda) = \bbF^{*}\,{\bf 1}_{\Lambda}\,\bbF$ (since $\chi\,{\bf 1}_{\Lambda} = {\bf 1}_{\Lambda}).$ Moreover, from the identification (\ref{eq.ident-hatH}) of the spectral space $\hatH,$ we see that the operator of multiplication by ${\bf 1}_{\Lambda}(\lambda)$ in $\hatH$ cannot vanish. Hence $\bbE(\Lambda) \neq 0$ for all non empty $\Lambda = (a,b) \subset  \bbR\setminus\{-\Omegam, 0,\Omegam \}.$ As $\sigma(\bbA)$ is closed, we conclude that $\sigma(\bbA)=\bbR.$ 

Suppose now that $\Lambda = \{a\} \subset \bbR\setminus\{-\Omegam, 0,\Omegam \}.$ We still have $\bbE(\Lambda) = \bbF^{*}\,{\bf 1}_{\Lambda}\,\bbF$ but here, the operator of multiplication by ${\bf 1}_{\{a\}}(\lambda)$ in $\hatH$ always vanishes except if $\Omegae = \Omegam$ and $a = \pm\Omegam/\sqrt{2}.$ To understand this, first consider a two-dimensional zone $\Lambda_\scZ$ for $\scZ \in \calZ\setminus\{\EE\}.$ As $\bbR \times \{a\}$ is one-dimensional, its intersection with $\Lambda_\scZ$ has measure zero in $\Lambda_\scZ,$ so the operator of multiplication by ${\bf 1}_{\{a\}}(\lambda)$ in $L^2(\Lambda_\scZ)$ vanishes. On the other hand, for the one-dimensional zone $\Lambda_\EE,$ several situations may occur. If $\Omegae\neq \Omegam,$ the intersection of $\bbR \times \{a\}$ and $\Lambda_\EE$ is either empty or consists of two points (which are symmetric with respect to the $\lambda$-axis), hence this intersection still have measure zero in $\Lambda_\EE.$ If $\Omegae = \Omegam,$ this intersection is empty when 
$a \notin \{-\Omegam/\sqrt{2},+\Omegam/\sqrt{2}\},$ whereas it is one half of $\Lambda_\EE$ when $a = \pm\Omegam/\sqrt{2}$ (the half located in the half-plane $\pm \lambda > 0).$ In the latter case, we see that the range of  the projection $\bbE(\{\pm \Omegam/\sqrt{2}\})$ is isomorphic, via the generalized Fourier transform $\bbF$, to the infinite dimensional space $L^2(\{|k|>k_{\rm c}\})$. To sum up, if $\Omegae \neq \Omegam,$ there is no eigenvalue of $\bbA$ in $\bbR\setminus\{-\Omegam, 0,\Omegam \},$ whereas if $\Omegae = \Omegam,$ the only eigenvalues of $\bbA$ located in $\bbR\setminus\{-\Omegam, 0,\Omegam \}$ are $\pm\Omegam/\sqrt{2}$ and they both have infinite multiplicity.

Finally, formulas  (\ref{eq.kernelA}) and (\ref{eq.eigenspace}) show that $0$ and $\pm\Omegam$ are also eigenvalues of infinite multiplicity.
\end{proof}

% ============================================================================
\subsection{Generalized eigenfunction expansions for the evolution problem}
As an application of the above results, let us express the solution $\bU(t)$ of our initial time-dependent Maxwell's equations written as the Schr\"odinger equation (\ref{eq.schro}).

Consider first the free evolution of the system, that is, when $\bG = 0.$ In this case, $\bU(t) = \rme^{-\rmi\bbA t} \bU_0$ for all $t>0,$ where $\bU_0 = \bU(0)\in \Hxy$ is the initial state. Theorem \ref{th.diagA} provides us a diagonal expression of $\rme^{-\rmi\bbA t}.$ More precisely, $\bbP\,\rme^{-\rmi\bbA t} = \rme^{-\rmi\bbA t}\,\bbP = \bbF^*\,\rme^{-\rmi\lambda t}\,\bbF.$ Hence, if we restrict ourselves to initial conditions $\bU_0 \in \Hxydiv = \mathrm{Im} (\bbP)$ (see (\ref{eq.Hxydiv})), we simply have $\bU(t) = \bbF^*\,\rme^{-\rmi\lambda t}\,\bbF\,\bU_0.$ Thanks to (\ref{eq.vecgen2d}) and (\ref{eq.adjointtransform}), this expression becomes
$$
\bU(t) = \sum_{\scZ \in \calZ} \sum_{j \in J_{\scZ}} \int_{\Lambda_{\scZ} } \langle  \bU_0, \bbW_{k,\lambda,j} \rangle_{\indHxy,s} \ \bbW_{k,\lambda,j}\,\rme^{-\rmi \lambda t} \,\rmd\lambda\,\rmd k.
$$
For simplicity, we have condensed here the various sums in the right-hand side of (\ref{eq.adjointtransform}) in a single sum which includes the last term corresponding to $\scZ = \EE$ (which is of course abusive for this term, since it is a single integral represented here by a double integral). The above expression is a \emph{generalized eigenfunction expansion} of $\bU(t).$ It tells us that $\bU(t)$ can be represented as a superposition of the time-harmonic waves $\bbW_{k,\lambda,j}\,\rme^{-\rmi \lambda t}$ modulated by the spectral components $\langle  \bU_0, \bbW_{k,\lambda,j} \rangle_{\indHxy,s}$ of the initial state. Strictly speaking, the above expression is valid if we are in the context of Proposition \ref{p.FandAdjoint}, \textit{i.e.}, if $\bU_0 \in \Hxyps$ with $s > 1/2$ and $\bbF\,\bU_0 \in \hatH_{\rm d}.$ For general $\bU_0 \in \Hxydiv,$ the limit process mentioned in Remark \ref{propbbW} is implicitly understood.

Consider now equation (\ref{eq.schro}) with a non-zero excitation $\bG \in C^{1}(\bbR^{+},\Hxydiv).$ For simplicity, we assume zero initial conditions. In this case, the Duhamel integral formula (\ref{eq.duhamel}) writes as
$$
\bU(t) = \int_{0}^{t} \bbF^{*}\,\rme^{-\rmi \lambda (t-s)}\,\bbF \, \bG(s) \,\rmd s,
$$ 
which yields the following generalized eigenfunction expansion:
$$
\bU(t) = \sum_{\scZ \in \calZ} \sum_{j \in J_{\scZ}} \int_{0}^{t}  \int_{\Lambda_{\scZ} } \langle  \bG(s), \bbW_{k,\lambda,j} \rangle_{\indHxy,s} \ \bbW_{k,\lambda,j}\,\rme^{-\rmi \lambda (t-s)} \,\rmd\lambda \,\rmd k \,\rmd s.
$$

As mentioned in the introduction, we are especially interested in the case of a time-harmonic excitation which is switched on at an initial time, that is, $\bG(t) = H(t)\,\rme^{-\rmi \omega t}\,\bG_\omega$ for given $\omega \in \bbR$ and $\bG_\omega \in \Hxydiv,$ where $H(t)$ denotes the Heaviside step function (\textit{i.e.}, $H := {\bf 1}_{\bbR^{+}}).$ In this particular situation, Duhamel's formula simplifies as
$$
\bU(t) = \phi_{\omega}(\bbA,t) \,\bG_\omega = \bbF^{*}\, \phi_{\omega}(\lambda,t)\, \bbF \, \bG_\omega,
$$
where $\phi_{\omega}(\cdot,t)$ is the bounded continuous function defined for non negative $t$ and real $\lambda$ by
\begin{equation*}
\phi_{\omega}(\lambda,t):= \rme^{-\rmi \lambda t}\int_{0}^{t} \rme^{-\rmi s (\omega -\lambda )} \rmd s =\left\lbrace\begin{array}{ll}
\displaystyle \rmi \,\frac{\rme^{-\rmi \, \lambda\, t} -\rme^{-\rmi \, \omega \, t}}{\lambda-\omega} & \mbox{ if } \lambda \neq \omega , \\[10pt]
    t\, \rme^{-\rmi \omega \,t} & \mbox{ if } \lambda =\omega.
\end{array}
\right.
\end{equation*}
The generalized eigenfunction expansion of $\bU(t)$ then takes the form
$$
\bU(t) = \sum_{\scZ \in \calZ} \sum_{j \in J_{\scZ}} \int_{\Lambda_{\scZ} } \phi_{\omega}(\lambda,t)\,\langle  \bG_\omega, \bbW_{k,\lambda,j} \rangle_{\indHxy,s} \ \bbW_{k,\lambda,j} \,\rmd\lambda \,\rmd k.
$$
In the second part \cite{Cas-Haz-Jol-Vin} of the present paper, we will study the asymptotic behavior of this quantity for large time, in particular the validity of the \emph{limiting amplitude principle}, that is, the fact that $\bU(t)$ becomes asymptotically time-harmonic. We can already predict that this principle fails in the particular case $\Omegae = \Omegam$ if we choose $\omega = \pm\Omegam/\sqrt{2},$ since $\pm\Omegam/\sqrt{2}$ are eigenvalues of infinite multiplicity of $\bbA$ (see Corollary \ref{c.spectrumA}). Indeed, if $\bG_\omega$ is an eigenfunction associated to $\pm\Omegam/\sqrt{2},$ it is readily seen that the above expression becomes
$$
\bU(t) = \int_{|k|>k_{\rm c} } \phi_{\omega}(\omega,t)\,\langle  \bG_\omega, \bbW_{k,\omega,0} \rangle_{\indHxy}\, \bbW_{ k,\omega,0} \, \rmd k.
$$
Using the expression of $\phi_{\omega}(\omega,t),$ we obtain
$$
\bU(t) = t\, \rme^{-\rmi \omega \,t} \ \int_{|k|>k_{\rm c} } \langle  \bG_\omega, \bbW_{k,\omega,0} \rangle_{\indHxy}\, \bbW_{ k,\omega,0} \, \rmd k.
$$
which shows that the amplitude of $\bU(t)$ increases linearly in time. This \emph{resonance} phenomenon is compatible with formula (\ref{eq.incrlint}) which tells us that $\|\bU(t)\|_{\indHxy}$ increases at most linearly in time. It is similar to the resonances which can be observed in a bounded electromagnetic cavity filled with a non-dissipative dielectric medium, when the frequency of the excitation coincides with one of the eigenfrequencies of the cavity. But it is related here to \emph{surface plasmon polaritons} which are the waves that propagate at the interface of our Drude material and that are described by the eigenfunctions associated with the eigenvalues $\pm\Omegam/\sqrt{2}$ of $\bbA$. Moreover, unlike the eigenfrequencies of the cavity, these eigenvalues are of infinite multiplicity and embedded in the continuous spectrum. In the case of an unbounded stratified medium composed of standard non-dissipative dielectric materials such nonzero eigenvalues do not exist (see \cite{Wed-91}).

In \cite{Cas-Haz-Jol-Vin}, we will explore more deeply this phenomenon and we will investigate all the possible behaviors of our system submitted to a periodic excitation. This forthcoming paper is devoted to both \emph{limiting absorption} and \emph{limiting amplitude} principles, which yield two different but concurring processes that characterize time-harmonic waves. In the limiting absorption principle, the time-harmonic regime is associated to the existence of one-sided limits of the resolvent of the Hamiltonian on its continuous spectrum, whereas in the limiting amplitude principle, it appears as an asymptotic behavior for large time of the time-dependent regime.

%XXXXXXXXXXXXXXXXXXXXXXXXXXXXXXXXXXXXXXXXXXXXXXXXXXXXXXXXXXXXXXXXXXXXXXXXXXXXXXXXXXXXXXXXXXXXXXXXXXXXXXXXXXXXXXXXXXXXXXXXXXXXXXX
%XXXXXXXXXXXXXXXXXXXXXXXXXXXXXXXXXXXXXXXXXXXXXXXXXXXXXXXXXXXXXXXXXXXXXXXXXXXXXXXXXXXXXXXXXXXXXXXXXXXXXXXXXXXXXXXXXXXXXXXXXXXXXXX

\end{document}